\documentclass[10pt,american,reqno]{amsart}
\usepackage{amsmath,amsfonts,amssymb,amscd,amsthm,amsbsy,bbm,epsf,calc,graphicx,mathrsfs,latexsym}
\usepackage{color}
\usepackage{datetime}
\usepackage{mathtools}

\usepackage[colorlinks=true, citecolor=blue, urlcolor=cyan, linkcolor=red]{hyperref}

\usepackage{mathdots}
\usepackage{cleveref}
\usepackage{cancel}
\usepackage{lineno}
\newcommand{\norm}[1]{\left\lVert#1\right\rVert}

\numberwithin{equation}{section}

\renewcommand{\arraystretch}{1.3}

\newcounter{hours}\newcounter{minutes}

\theoremstyle{plain}
\newtheorem{theorem}{Theorem}[section]

\newtheorem{lemma}{Lemma}

\newtheorem{proposition}{Proposition}

\newtheorem{definition}{Definition}

\theoremstyle{definition}                  %% For unnumbered Remarks, etc.
\newtheorem{remark}{Remark}

\def\A1{\mathcal{A}_1}
\newcommand{\R}{\mathbb{R}}	% Set of real numbers.
\newcommand{\N}{\mathbb{N}}	% Set of natural numbers.
\newcommand{\pa}{\partial}		% Partial derivative.
\newcommand{\Div}{\textrm{div}\,}	% Divergence.
\newcommand{\na}{\nabla}		% Nabla.

\newcommand{\chf}[1]{{\raisebox{3pt}{\Large $\chi$}}_{#1}}
\newcommand{\ft}{f^{(\tau)}}
\newcommand{\tft}{\tilde{f}^{(\tau)}}

\newcommand{\red}[1]{\textcolor{red}{#1}}

\title{Non-Local Porous Media Equations with Fractional Time Derivative}

\begin{document}
\thanks{ ED is supported by the Austrian Science Fund (FWF) grants P30000 and W1245. MPG is supported by DMS-1514761 and would like to
thank NCTS Mathematics Division Taipei for their kind hospitality. NZ acknowledges support from the Alexander von Humboldt foundation.}

\author{Esther S.~Daus, Maria Pia Gualdani, Jingjing Xu, Nicola Zamponi, Xinyu Zhang}
\address{Institute of Analysis and Scientific Computing, TU Wien, Wiedner Hauptstraße 8–10,1040 Wien, Austria} 
\email{esther.daus@tuwien.ac.at}
	\address{George Washington University, Mathematics Department, 801 22nd St. NW, Room 739
Washington, DC 20052}
\email{jix29@gwmail.gwu.edu zxyhxz@gwmail.gwu.edu}
\address{The University of Texas at Austin
Mathematics Department RLM 8.100
2515 Speedway Stop C1200
Austin, Texas 78712-1202}
\email{gualdani@math.utexas.edu }
\address{University of Mannheim, School of Business Informatics and Mathematics, B6, 28, 68159 Mannheim (Germany)}
\email{nzamponi@mail.uni-mannheim.de}

\date{\today}

\begin{abstract}
In this paper we investigate existence of solutions for the system: 
\begin{equation*}
\left\{
\begin{array}{l}
D^{\alpha}_tu=\textrm{div}(u \nabla p),\\
D^{\alpha}_tp=-(-\Delta)^{s}p+u^{2},
\end{array}
\right.
\end{equation*}
in $\mathbb{T}^3$ for $0< s \leq 1$, and $0< \alpha \le 1$. The term $D^\alpha_t u$ denotes the Caputo derivative, which models memory effects in time.  The fractional Laplacian $(-\Delta)^{s}$ represents the L\'{e}vy diffusion. 
We prove global existence of nonnegative weak solutions that satisfy a variational inequality. The proof uses several approximations steps, including an implicit Euler time discretization. We show that the proposed discrete Caputo derivative satisfies several important properties, including positivity preserving, convexity and rigorous convergence towards  the continuous Caputo derivative. Most importantly, we give a strong compactness criteria for piecewise constant functions, in the spirit of Aubin-Lions theorem, based on bounds of the discrete Caputo derivative. 

%that allows for convergence of the discrete system to the continuous one. 

 %to the problem by employing time discretization. To pass the time step limit, we need compactness criteria such as Aubin-Lions lemma. For that we show here an equivalent of Aubin-Lions compactness criterion for a discrete version of the Caputo time derivative.

\end{abstract}

\iffalse
\fi

\maketitle

%\tableofcontents
\baselineskip=14pt
\pagestyle{headings}		% running headline and page nos. at top

\markboth{ }{ }
\section{Introduction}

In this manuscript we study existence of weak solutions to the following system: % a porous medium equation with non-local diffusion effects: 
\begin{equation}\label{1L}
\left\{
\begin{array}{l}
D^{\alpha}_tu=\textrm{div}(u \nabla p),\\
D^{\alpha}_tp=-(-\Delta)^{s}p+u^{2},\ 0<s\leq 1, 
\end{array}
\right.
\end{equation}
where the operator $D^{\alpha}_t$ denotes the Caputo-type time derivative
$$
D^{\alpha}_t f :=\frac{1}{\Gamma_{1-\alpha}}\int_0^t\frac{f'(s)}{(t-s)^{\alpha}}\ ds, \quad 0< \alpha <1.
$$
Here $u(x,t)\ge 0$ denotes the density function and $p(x,t)\ge 0$ the pressure.  The model describes the time evolution of a density function $u$ that evolves under the nonlocal continuity equation 
$$
D^{\alpha}_t u = \Div(u {\bf{v}}),
$$
where the velocity is conservative, ${\bf{v}}=\nabla p$, and $p$ is related to $u^2$ by the inverse of the fractional fully nonlocal heat operator $D^{\alpha}_t p + (-\Delta )^s$.
%, namely: 
%$$
%p = K {\ast}_{x} p_{in} + K{\ast}_{x,t}u^2 .
%$$

System (\ref{1L}) is a non-local-in-time version of the one recently studied in \cite{caffarelli2018non}. In  \cite{caffarelli2018non} the authors proved the existence of weak solutions to 
\begin{equation}\label{CGZ20}
    \begin{cases}
      \pa_t u = \Div(u\na p),\\%\qquad\quad \quad \mbox{in }\R^2\times(0,\infty),\\
      \pa_t p = -(-\Delta)^s p + u^\beta.%\qquad\mbox{in }\R^2\times(0,\infty) ,\\
%u(\cdot,0) = u_{in},\quad p(\cdot,0)=p_{in}\qquad \mbox{in }\R^2,
    \end{cases}
\end{equation}
for $x\in \R^2$, $ \frac{1}{\beta} < s < 1$ and $\beta>1$.  The literature on (\ref{CGZ20}) and his variants is quite large. See \cite{AS08, BIK15, CSV13, CV11, CV15, DGZ19, GLM00, LMG01, SV14, STV18} and references therein.

The presence of $D^{\alpha}_t$ makes our system quite different from  (\ref{CGZ20}). For example, techniques such as {a Div-Curl Lemma} do not work. The non-local structure prevents the equation from having  a comparison principle. The maximum principle does not give useful insights, since at any point of maximum for $u$ we only know that $D^{\alpha}_t u  \le u \Delta p$. We overcome these significant shortcomings with the introduction of ad-hoc regularization terms, together with suitable compact embeddings. We also provide a new strong compactness criterium for families of piecewise constant functions.%; this new theorem uses bounds of the discrete Caputo derivative. 

%%%%%%%%%%%%%%%%%%%% DA QUI 
The authors are interested in understanding the effects and challenges that a non-local time derivative brings to a mathematical model. Time-delay memory effects are, in fact, very common in real situations. Specifically, differential equations with non-local time derivatives are used in modeling many physical and engineering processes, including particles in heat bath and soft matter with viscoelasticity (\cite{coleman1961foundations,del1997concepts,diethelm2010analysis,gorenflo2015time,kilbas2006theory,kubo1966fluctuation,li2018generalized,zwanzig2001nonequilibrium}).  In the past few years the study of stochastic and deterministic partial differential equations with non-local time derivatives has seen an increasing interest, see \cite{allen2016parabolic, allenCV17, allenI, allenII, bernardis2016maximum,  fractionalBurger2020, gorenflo2015time, li2018some, LiLiu2018, taylorremarks} and the references therein.  
%%%%%%%%%%%%%%%%%%%%

The main result of this manuscript is summarized in the following two theorems: 
\begin{theorem}\label{thm:mainthm}
Let $u_{in}, p_{in}: \mathbb{T}^3 \rightarrow (0,+\infty )$ be functions such that $u_{in}\in L^2(\mathbb{T}^3)$, $p_{in}\in H^1(\mathbb{T}^3)$.
%$ \ \int_{\mathbb{T}^{3}}u_{in}^{2}+\frac{1}{2}|\nabla p_{in}|^{2}\;dx < +\infty$. 
For  $0< s \leq 1$ and $0 < \alpha \le  1$, there exist functions $u,p : \mathbb{T}^3 \times [0,+\infty)$ such that for every $T>0$
$$u\in L^\infty(0,T;L^2(\mathbb{T}^3)), \ p\in L^2(0,T;H^{1+s}(\mathbb{T}^3)),$$
$$D^{\alpha}_t u \in L^2(0,T;(W^{1,\infty}(\mathbb{T}^3))'), \ D^{\alpha}_t p \in  L^2(0,T;(L^\infty \cap H^1)'(\mathbb{T}^3)), $$
which satisfy the following variational inequalities:
\begin{align*}
\int_{0}^{T}\int_{\text{\ensuremath{\mathbb{T}}}^{3}} \left<  D_t^{\alpha}{u}, \phi \right> \;dxdt+\int_{0}^{T}\int_{\text{\ensuremath{\text{\ensuremath{\mathbb{T}}}^{3}}}}{u^{}} \nabla p^{} \cdot \nabla\phi \;dxdt  =0,\ \forall\phi\in  L^{2}(0,T;\ W^{1,\infty} (\mathbb{T}^{3})),
\end{align*}
\begin{align*}
\int_{0}^{T}\int_{\mathbb{T}^{3}} \left<  D_t^{\alpha}p^{}, \psi  \right> \;dxdt+\int_{0}^{T}\int_{\text{\ensuremath{\mathbb{T}}}^{3}}(-\Delta)^{s/2}p^{} (-\Delta)^{s/2}\psi \;dxdt&\nonumber \\
-\int_{0}^{T}\int_{\text{\ensuremath{\mathbb{T}}}^{3}}u^{2}\psi \;dxdt & \geq 0,\ \forall\psi\in  L^{2}(0,T;\ H^{1}\cap L^\infty(\mathbb{T}^{3})),
\end{align*}
$$\lim_{t\rightarrow0^+}u(t) = u_{in} ~ \mbox{strongly in} ~ (W^{1,\infty}(\mathbb{T}^3))', \ \lim_{t\rightarrow0^+}p(t) = p_{in} ~ \mbox{strongly in} ~ (L^\infty \cap H^1)'(\mathbb{T}^3).$$
\end{theorem}
The operator $(-\Delta)^{s}$ is the fractional Laplacian and, on the torus, is defined via its Fourier series. More precisely, the $n$-th Fourier coefficient of $(-\Delta)^{s} u$ is 
\begin{align*}
\widehat{(-\Delta)^{s} u } (n) =  |n|^{2s} \widehat{u}(n),
\end{align*}
with $\widehat{u}(n)$ the $n$-th Fourier coefficient of $u$:
$$
 \widehat{u}(n) = \frac{1}{(2\pi)^{3}} \int_{\mathbb{T}^{3}} u(x) e^{-in\cdot x}\;dx.
$$

The starting point about our analysis is the observation that %the {\em{free energy functional}}
$$ H[u,p] := \int_{\mathbb{T}^3} \left( u^2 + \frac{1}{2}|\na p|^2 \right)dx  $$
is a Lyapunov functional for (\ref{1L}) and satisfies the bound
$$
D^{\alpha}_t  H[u,p] +  \int_{\mathbb{T}^3} 
|(-\Delta)^{s/2}\na p |^2 dx= 0. 
$$
To gain uniform estimates for $u$ in Sobolev spaces we first add a viscosity term $\rho \Delta u$ in the equation for $u$. After, we discretize the resulting system using the implicit Euler method; the evolution problem is approximated by a system of elliptic problems. We introduce a novel discrete formulation for the Caputo derivative, namely
\begin{equation}
\label{dis_left_caputoIN}
(D^\alpha_\tau f)_{k} := \Gamma_\alpha\tau^{-\alpha}\sum_{j=0}^{k-1}
\lambda_{k-j}(f_{j+1}-f_j),\quad k\geq 1,
\end{equation}
where the sequence $\{\lambda_k\}$ is defined by linear recurrence:
\begin{align*}
\lambda_{k+1} &= \sum_{j=1}^k ((k-j+1)^{\alpha-1}-(k-j+2)^{\alpha-1})\lambda_j,
\quad k\geq 1,\qquad\lambda_1 = 1. %\label{rec_rela_Lambda}
\end{align*}
We provide several properties for the sequence $\{\lambda_k\}$, including monotonicity, boundedness and asymptotic behavior. In particular we show that, as $\tau \to 0$, a suitable piecewise constant function associated to
$\lambda_k$ converges, up to constants, to the function  $\frac{1}{t^{\alpha}}$. Moreover, $ (D^\alpha_\tau f)_{k} $ satisfies some important properties, such as a discrete formulation of the fundamental theorem of calculus 
$$
f_n = f_0 + \frac{\tau^\alpha}{\Gamma_\alpha}\sum_{k=1}^n (n-k+1)^{\alpha-1}(D^\alpha_\tau f)_k,
$$
and convergence towards the continuous Caputo derivative, see Proposition \ref{prop.conv.dis.byparts} and Proposition \ref{conv.dis.caputo}.  

Once we have established the well-posedness for the discrete system, we show the existence of a sequence of solutions whose limit is a solution of the original continuous  problem (with the extra viscosity term). In this step we apply strong compactness criteria. To handle our specific problem, we need strong compactness using bounds on the discrete Caputo derivative. To the best of the authors' knowledge, there is no result of this kind in the literature. To fill this gap, we have proven the following variant of the Aubin-Lions theorem:  
\begin{theorem}\label{thm:mainII}
Assume that $X$, $B$ and $Y$ are Banach spaces such that the embedding $X\hookrightarrow B$ is compact and the embedding $B\hookrightarrow Y$ is continuous. Let $1\leq p\leq \infty$, $0 < \alpha < 1$ and $\{f^{(\tau)}\}$ be a sequence of piecewise constant functions, satisfying
\begin{align}\label{cond_comp_theo}
\norm{f^{(\tau)}}_{L^p(0,T;X)}+ \norm{D^{\alpha}_\tau f^{(\tau)}} _{L^p(0,T;Y)} \leq C_0,
\end{align}
 \iffalse where $1>1/r-1/p$.\fi where $C_0$ is a constant independent of $\tau$ and $D^{\alpha}_{\tau}$ is the discrete Caputo derivative operator defined in (\ref{dis_left_caputoIN}). Then $\{f^{(\tau)}\}$ is relatively compact in $ L^p(0,T;B)$.
\end{theorem}

The proof of Theorem \ref{thm:mainII} uses Simon's version of Aubin-Lions' compactness theorem \cite{simon1986compact}. We first work with the linear interpolant functions $\tilde{f}^{(\tau)}$ of $f^{(\tau)}$ and show that $\|D_\tau^\alpha f^{(\tau)} \| \ge \|D_t^\alpha\tilde{ f}^{(\tau)}\|$; here we compare the discrete derivative of a piecewise constant function with the continuous derivative of the corresponding linear interpolant. After, we prove that the interpolants satisfy the estimate $\|\tilde{ f}^{(\tau)}(\cdot +h)-\tilde{ f}^{(\tau)}\|_{L^p(0,T,Y)}\| \le h^\alpha \|D_t^\alpha\tilde{ f}^{(\tau)}\|$ and use bound (\ref{cond_comp_theo}) and Simon's version of Aubin-Lions' compactness theorem to conclude that $\{\tilde{ f}^{(\tau)} \}$ is compact in $ L^p(0,T;B)$. In the end,  we show that convergence of the family of linear interpolants implies convergence of the piecewise constant functions. The limit is the same.

%It  states that a family of functions $F$ is compact in $L^p(0,T,B)$ if $(i)$ $F$ is bounded in $L^p(0,T,X)$,  $(ii)$ for any $g$ in $F$ we have that $\|g(\cdot +h)-g\|_{L^p(0,T,Y)}\to 0$ as $h\to 0$, $(iii)$ the spaces $X$, $Y$, $B$ are such that  $X\subset B\subset Y$ with $X$ compactly embedded in $B$. In our case, the family $F$ consists of the linear interpolant functions $\tilde{f}^{(\tau)}$ of $f^{(\tau)}$. We divide the proof into three steps. In the first we show that $\|D_\tau^\alpha f^{(\tau)} \| \ge \|D_t^\alpha\tilde{ f}^{(\tau)}\|$; we basically compare the discrete derivative of a piecewise constant function with the continuous derivative of the corresponding linear interpolant. After, we prove that the interpolants satisfy the estimate $\|\tilde{ f}^{(\tau)}(\cdot +h)-\tilde{ f}^{(\tau)}\|_{L^p(0,T,Y)}\| \le h^\alpha \|D_t^\alpha\tilde{ f}^{(\tau)}\|$ and use bound (\ref{cond_comp_theo}) and Simon's version of Aubin-Lions' compactness theorem to conclude that $\{\tilde{ f}^{(\tau)} \}$ is compact in $ L^p(0,T;B)$. Lastly,  we show that convergence of the family of linear interpolants implies convergence of the piecewise constant functions. The limit is the same. 

One could wonder if it is necessary to work with linear interpolants $\{\tilde{ f}^{(\tau)} \}$ instead with piecewise constant functions $\{{ f}^{(\tau)} \}$. In fact, Theorem 1 in \cite{dreher2012compact} says that an estimate of the form $\|{ f}^{(\tau)}(\cdot +h)-{ f}^{(\tau)}\|_{L^p(0,T,Y)} \le C h$ is sufficient to invoke Aubin-Lions' theorem. Unfortunately, due to the nature of the Caputo derivative, the only estimate one can get for $\{{ f}^{(\tau)} \}$ is: 
$$
\|{ f}^{(\tau)}(\cdot +\tau)-{ f}^{(\tau)}\|_{L^p(0,T,Y)} \le C \tau^\alpha, 
$$ 
which, alone, is not enough to guarantee compactness, see Proposition 2 in \cite{dreher2012compact}.

In the last step of the proof of  Theorem \ref{thm:mainthm} we remove the viscosity term $\rho \Delta u$. The major difficulty, in the approximation process, is the identification of the limit of ${u^{(\rho)}}^2$. The energy inequality  provides plenty of informations for the pressure $p$, but only uniform integrability in $L^\infty(0,T,L^2(\mathbb{T}^3))$ for $u$. It is unclear, at the moment, how to use the bounds for $\nabla p$ to get useful bounds for $\nabla u$ or $u$. The authors in \cite{caffarelli2018non} overcome a similar problem using the Div-Curl lemma, a tool commonly employed in the study of fluid-dynamic systems. The presence of the Caputo derivative makes this method not useful. One interesting question would be if the Div-Curl lemma is still true when one considers derivatives of order strictly less than one.  The authors have not explored this direction yet. Even if the lack of strong compactness prevents us from identifying the limit of $u^{(\rho)}$, we are able to provide a lower bound for: 
$$
\int_0^T\int_{\mathbb{T}^3} (u^{(\rho)})^2\;\psi \;dxdt.
$$ 
The above integral defines a functional 
$$\Psi(u) : = \int_{0}^{T}\int_{\text{\ensuremath{\mathbb{T}}}^{3}} 
u^{2}\psi \;dxdt$$ 
which is convex and continuous in the strong topology  of $L^2(0,T,L^2(\mathbb{T}^3))$ for suitable functions $\psi$. We conclude that 
$$
\liminf_{\rho \to 0} \Psi(u^{(\rho)}) \ge  \Psi(u),
$$
where $u$ is the weak limit in the $L^2(0,T,L^2(\mathbb{T}^3))$-topology.

\subsection{Outline}
 The rest of the paper is organized as follows. In \Cref{sec:pre}, we prove some properties of the discrete Caputo derivatives, then prove \Cref{thm:mainII}. \Cref{sec:eqn,sec:tauLimit,sec:epsilonLimit,sec:rhoLimit} concern the proof of Theorem \ref{thm:mainthm}. In the Appendix we show a formal $L^{3}(0,T;L^3(\mathbb{T}^3))$ estimate for $u$, provided $\frac{1}{2}<s\le1$.
 
 %: in \Cref{sec:eqn}, we add a viscosity term to the discretized problem to prove the existence of solutions to the discretized problem; in \Cref{sec:tauLimit}, we use our compactness criterion to pass to the limit in the time step in order to get the existence of solutions to the continuous problem; finally, in \Cref{sec:rhoLimit}, we remove the viscosity term.
 
\subsection{Notation}
We list here the notations that will be used consistently throughout the paper. %The following letters are fixed throughout the paper and always refer to:
\begin{itemize}
\item $\Gamma_{\alpha}$: the gamma function evaluated at $\alpha$.
\item $\lceil z \rceil$: the ceiling function, namely the smallest integer greater than or equal to $z$.
\item $\lfloor z \rfloor$: the floor function, namely the largest integer smaller than or equal to $z$.
\item $g_{+}:=\max\{g,0\}$ and $g_{-}:=\min\{g,0\}$ for measurable function $g$.
\end{itemize}

\section{Preliminaries and Compactness Criteria}\label{sec:pre}
\subsection{Preliminaries}
In this section we state various useful properties of the Caputo derivative. We first recall the definition \cite{li2018generalized}.
\begin{definition}\label{def:contCaputo}
If a function $f(t)$ is absolutely continuous in $(0,T]$, for $0< \alpha <1$ the left Caputo derivative of $f$ in variable $t\in (0,T)$ is defined as
\begin{equation}
D_t^{\alpha}f(t)=\frac{1}{\Gamma_{1-\alpha}}\int_0^t\frac{f'(s)}{(t-s)^{\alpha}}\ ds,\label{left_caputo1}
\end{equation}
and the right Caputo derivative is defined as
\begin{equation}
{}^*D_t^{\alpha}f(t)=\frac{1}{\Gamma_{1-\alpha}}\int_t^T\frac{f'(s)}{(s-t)^{\alpha}}\ ds,\label{right_caputo1}
\end{equation}
where $\Gamma_{z}$ is the gamma function  $\Gamma_z=\int_0^{\infty}e^{-x}x^{z-1}\;dx$.
\end{definition}
\begin{remark}
Alternative formulas for \eqref{left_caputo1} and \eqref{right_caputo1} are
\begin{equation}
D_t^{\alpha}f(t)=\frac{1}{\Gamma_{1-\alpha}}\left( \frac{f(t)-f(0)}{t^{\alpha}}+\alpha\int_0^t\frac{f(t)-f(s)}{(t-s)^{1+\alpha}}\ ds\right),\label{left_caputo2}
\end{equation}
and
\begin{equation}
{}^*D_t^{\alpha}f(t)=\frac{1}{\Gamma_{1-\alpha}}\left( \frac{f(T)-f(t)}{(T-t)^{\alpha}}+\alpha\int_t^T\frac{f(s)-f(t)}{(s-t)^{1+\alpha}}\ ds\right),\label{right_caputo2}
\end{equation}
respectively.
\end{remark}
For the purpose of our problem, we define a discrete version of \eqref{left_caputo1} and \eqref{right_caputo1}.
\begin{definition}\label{def:disCaputo}
Consider a function $f(t)$ in $[0,T]$. Subdivide the time interval $[0,T]$ into $N$ subintervals with uniform time step $\tau=T/N$. Denote $t_k:=k\tau$ and $f_k:=f(t_k)$ for $k=0,1, ..,N$. The left and right discrete Caputo derivatives of $f(t)$ of order $\alpha$ with $0<\alpha\leq 1$ at $t_k$ are approximated by the linear combination of values of $f$ at $t_0,t_1, ..,t_k$ as follows:
\begin{equation*}
\label{dis_left_caputo1}
(D^\alpha_\tau f)_{k} = \Gamma_\alpha\tau^{-\alpha}\sum_{j=0}^{k-1}
\lambda_{k-j}(f_{j+1}-f_j),\quad k\geq 1;
\end{equation*}
\begin{equation*}\label{dis_right_caputo1}
({}^*D^\alpha_\tau f)_{k} = \Gamma_\alpha\tau^{-\alpha}\sum_{j=k+1}^N
\lambda_{j-k}(f_j-f_{j-1}),\quad k\leq N-1,
\end{equation*}
where the sequence $\{\lambda_k\}$ is defined by the following infinite linear recurrence:
\begin{align}
\lambda_{k+1} &= \sum_{j=1}^k ((k-j+1)^{\alpha-1}-(k-j+2)^{\alpha-1})\lambda_j,
\quad k\geq 1,\qquad\lambda_1 = 1. \label{rec_rela_Lambda}
\end{align}
We extend the definition for any $t\in [0,T]$ by the backward finite difference operator:
\begin{equation}\label{eq:disCaputo2}
D^{\alpha}_{\tau}f(t):=\Gamma_{\alpha}\tau^{-1-\alpha}\int_0^{\lfloor t/\tau \rfloor  \tau-\tau}\lambda_{(\lfloor t/\tau \rfloor-\lfloor s/\tau \rfloor )}  (f(s+\tau)-f(s))\ ds,
\end{equation}
with $D^{\alpha}_{\tau}f(0)=0$. \\
When $f(t)$ is a piecewise constant function 
$f(t):= \sum_{k=1}^\infty f_k \chf{(t_{k-1},t_k]}(t)$ with 
$f(t) = f_{in}$ for $t\leq 0$, then \eqref{eq:disCaputo2} can be rewritten as
\begin{equation*}\label{eq:disCaputo3}
D_{\tau}^{\alpha}f^{(\tau)}(t)=\sum_{k=1}^n(D^{\alpha}_{\tau}f^{(\tau)})_k  \chf{(t_{k-1},t_k]}(t).
\end{equation*}
\end{definition}
 For $\alpha=1$ the operator $(D^\alpha_\tau f)_{k}$ reduces to the classical backward finite difference operator:
$$
(D^1_\tau f)_k = \frac{f_k-f_{k-1}}{\tau},\quad k\geq 1,
$$
since  $\lambda_j=0$ for all  $j\geq 2$. For $\alpha<1$ one cannot obtain an exact expression for $\lambda_j$. We have, however, an upper bound, the asymptotic behaviour, and the proof of monotonicity, as stated in the following proposition. 

\begin{proposition} 
Let $\{\lambda_k\}_{k\in \mathbb{N}}$ be defined as in (\ref{rec_rela_Lambda}). It holds that 
\begin{equation}
\lambda_{k+1} < 
\lambda_k \leq \frac{1}{k^{\alpha}}, \quad \sum_{j=1}^{k} \lambda_j  \le k^{1-\alpha},\qquad k\geq 1, \label{est_sum_lambda}
\end{equation}
and as $\tau\to 0$
\begin{equation}\label{weak.conv.lambda}
\tau^{-\alpha}\sum_{k=1}^\infty\lambda_k\chf{(t_{k-1},t_k]}(t)\rightharpoonup \frac{t^{-\alpha}}{\Gamma_{\alpha}\Gamma_{1-\alpha}} \mbox{ in } L^p(0,T) \mbox{ with } 1\leq\, p<1/\alpha.
\end{equation}
Moreover, for any sequence $f^{(\tau)} \rightarrow f$ strongly convergent in $L^q(0,T)$ for some $q\in [1,\infty)$, we have 
\begin{equation}\label{weak.conv.lambda.Convolution}
\frac{1}{\tau^{\alpha}}\int_0^{t_n}\sum_{k=1}^n \lambda_k  \chf{(t_{k-1},t_k]}(s) f^{(\tau)}(t-s)\; ds \rightarrow \frac{1}{\Gamma_{\alpha}\Gamma_{1-\alpha}}\int_0^t s^{-\alpha} f(t-s)\; ds,
\end{equation}
strongly in $L^{r}(0,T)$ for every $r < \frac{q}{1-(1-\alpha)q}$.
%strongly in $L^p(0,T)$.
\end{proposition}
\begin{proof}
We begin by showing that the sequence $(\lambda_k)_{k\in\N}$ is decreasing. Due to the definition of $\lambda_i$,
\begin{align}\label{ind.0}
\lambda_{i+1} = \sum_{j=1}^i (i-j+1)^{\alpha-1}\lambda_j -
\sum_{j=0}^{i-1}(i-j+1)^{\alpha-1}\lambda_{j+1},
\end{align}
which implies
\begin{align}\label{ind.1}
\sum_{j=1}^i (i+1-j)^{\alpha-1}(\lambda_j - \lambda_{j+1}) = (i+1)^{\alpha-1},
\qquad i\geq 1.
\end{align}
We prove by induction that $\lambda_{i+1}<\lambda_i$ for $i\geq 1$.
Since $\lambda_2 = 1 - 2^{\alpha-1} < 1 = \lambda_1$, the statement is true for $i=1$. Let us now assume that $\lambda_{i+1} < \lambda_i$ for 
$1\leq i \leq k-1$. From \eqref{ind.1} evaluated at $i=k$ it follows
\begin{align*}
1 - (k+1)^{1-\alpha}(\lambda_{k} - \lambda_{k+1}) =
\sum_{j=1}^{k-1} \left(1-\frac{j}{k+1}\right)^{\alpha-1}
(\lambda_j - \lambda_{j+1}) .
\end{align*}
Since $\lambda_j - \lambda_{j+1} > 0$ for $j=1,\ldots,k-1$ by inductive assumption and $\left(1-\frac{j}{k+1}\right)^{\alpha-1} < \left(1-\frac{j}{k}\right)^{\alpha-1}$ for $j=1,\ldots,k-1$, we deduce
\begin{align*}
1 - (k+1)^{1-\alpha}(\lambda_{k} - \lambda_{k+1}) < 
\sum_{j=1}^{k-1} \left(1-\frac{j}{k}\right)^{\alpha-1}
(\lambda_j - \lambda_{j+1}) .
\end{align*}
Evaluating \eqref{ind.1} at $i=k-1$ yields that the right-hand side of the above identity equals 1, so:
\begin{align*}
1 - (k+1)^{1-\alpha}(\lambda_{k} - \lambda_{k+1}) < 1,
\end{align*}
meaning that $\lambda_{k} - \lambda_{k+1} > 0$. Therefore the sequence $(\lambda_i)_{i\in\N}$ is decreasing.\\
We prove now the upper bounds for $\lambda_k$.
From \eqref{ind.0} it follows
\begin{align*}
\sum_{j=1}^{i+1} (i-j+2)^{\alpha-1}\lambda_{j} = 
\sum_{j=0}^{i} (i-j+1)^{\alpha-1}\lambda_{j+1} = 
\sum_{j=1}^i (i-j+1)^{\alpha-1}\lambda_j,\quad i\geq 1. 
\end{align*}
Thanks to the above identity, one can easily prove by induction that
\begin{equation}
\label{1}
\sum_{j=1}^i (i-j+1)^{\alpha-1}\lambda_j = 1,\quad i\geq 1.
\end{equation}
Thus, since $\{\lambda_k\}$ is a decreasing sequence, 
\begin{equation*}
1 = \sum_{j=1}^n \lambda_j(n-j+1)^{\alpha -1} \geq \sum_{j=1}^n \lambda_j n^{\alpha -1} \ge \lambda_n n^{\alpha },
\end{equation*}
which implies
$$\lambda_n \leq n^{-\alpha},\quad \sum_{j=1}^n \lambda_j  \le n^{1-\alpha} \quad \textrm{ for any $n\geq 1$}.
$$ 
Define now
$$\lambda^{(\tau)}(s):=\tau^{-\alpha}\sum_{j=1}^n\lambda_j \chf{(t_{j-1},t_j]}(s).$$ 
Since $\lambda^{(\tau)}(s) \leq \sum_{j=1}^n(j\tau)^{-\alpha}\chf{(t_{j-1},t_j]}(s) \leq  s^{-\alpha}$, the sequence $\lambda^{(\tau)}$ is uniformly bounded in $L^p(0,T)$ for every $p\in [1,1/\alpha)$ and $\lambda^{(\tau)}\rightharpoonup \lambda $ in $L^p(0,T)$ with $1 \leq\, p < 1/\alpha$ up to a subsequence.

Next we show (\ref{weak.conv.lambda.Convolution}). Let $n = \lfloor t/\tau \rfloor$, $t_n\leq t <t_{n+1}$, and $q\leq r < \frac{q}{1-(1-\alpha)q}$. 
Define
\begin{align*}
F_n^{(\tau)}(t) := &\int_0^{t_n} \lambda^{(\tau)}(s)  f^{(\tau)} (t-s)\; ds\\
=& \int_0^{t} \lambda^{(\tau)}(s)  f^{(\tau)} (t-s)\; ds - \int_{t_n}^t \lambda^{(\tau)}(s)  f^{(\tau)} (t-s)\; ds\\
=: &\;F_1^{(\tau)}(t) - F_2^{(\tau)}(t).
\end{align*}
%Firstly, we show that $F_n^{(\tau)}(t)$ converges weakly in $L^p(0,T)$ for some $1<p<1/\alpha$. 
%Thus, using From \eqref{1},% we have $\sum_{j=1}^n \lambda_j(n-j+1)^{\alpha -1} = 1$ for any $n \geq 1$. Thus
%\begin{equation*}
%1 = \sum_{j=1}^n \lambda_j(n-j+1)^{\alpha -1} \geq \sum_{j=1}^n \lambda_n n^{\alpha -1}.
%\end{equation*}
%Hence $\lambda_n \leq n^{-\alpha}$ for any $n\geq 1$. 
%Since $\lambda^{(\tau)}(s) \leq \sum_{j=1}^n(j\tau)^{-\alpha}\chf{(t_{j-1},t_j]}(s) \leq  s^{-\alpha}$, the sequence $\lambda^{(\tau)}$ is uniformly bounded in $L^p(0,T)$ and $\lambda^{(\tau)}\rightharpoonup \lambda $ in $L^p(0,T)$ with $1 < p<1/\alpha$ up to a subsequence. Now we re-write
The term $F_2^{(\tau)}$ tends to zero in $L^r(0,T)$ as $\tau\to 0$. 
In fact,
\begin{align*}
\norm{F_2^{(\tau)}}_{L^r(0,T)} 
&\leq \norm{\int_{t-\tau}^t\lambda^{(\tau)}(s)  f^{(\tau)}(t-s)\; ds}_{L^r(0,T)}\\
&=\norm{ (\lambda^{(\tau)}\chf{[0,T]})\ast (f^{(\tau)}\chf{[0,\tau]}) }_{L^r(0,T)} \\
&\leq \norm{\lambda^{(\tau)}}_{L^{1+\eta}(0,T)}
\norm{f^{(\tau)}}_{L^{q}(0,\tau) }\\
&\leq C \norm{f^{(\tau)} - f}_{L^q(0,T)} 
+ C\norm{f}_{L^q(0,\tau)},
\end{align*}
where $1+\eta = (1+1/r-1/q)^{-1} \in [1, 1/\alpha)$,
and the right-hand side of the above inequality tends to zero as $\tau\to 0$
since $f^{(\tau)}\to f$ strongly in $L^q(0,T)$.\\
On the other hand,
\begin{align*}
&\norm{F_1^{(\tau)} - \int_0^t \lambda(s)  f(t-s)\; ds}_{L^q(0,T) }\\
\leq & \norm{\int_0^t \left(\lambda^{(\tau)}-\lambda\right)(s)  f(t-s)\; ds}_{L^q(0,T)} + \norm{\int_0^t\lambda^{(\tau)}(s) \left(f^{(\tau)}-f\right)(t-s) \; ds}_{L^q(0,T)}\\
\leq & \norm{\int_0^t \left(\lambda^{(\tau)}-\lambda\right)(s)  f(t-s)\; ds}_{L^q(0,T)} + \norm{\lambda^{(\tau)}}_{L^1(0,T)} \norm{f^{(\tau)}-f}_{L^q(0,T)}.
\end{align*}
The second term above converges to $0$ as $\tau\rightarrow 0$. Now consider the first term.
Since $f\in L^{q}(0,T)$ and $\norm{\lambda^{(\tau)}-\lambda}_{L^1(0,T)} \leq C$, for $C$ not depending on $\tau$, by Corollary 4.28 in \cite{brezis2010functional}, the sequence
$$
I^{(\tau)}:= \int_0^t \left(\lambda^{(\tau)}-\lambda\right)(s)  f(t-s)\; ds
$$
has compact closure in $L^{q}(0,T)$. 
However, for every $\phi\in L^{q'}(0,T)$ ($q' = q/(q-1)$ for $q>1$, $q'=\infty$ for $q=1$),
\begin{align*}
\int_0^T I^{(\tau)}(t)\phi(t)dt &= \int_0^T\phi(t)
\int_0^t \left(\lambda^{(\tau)}-\lambda\right)(s)  f(t-s)\; ds\; dt\\
&= \int_0^T \left(\lambda^{(\tau)}-\lambda\right)(s)\int_s^T \phi(t) f(t-s)\; dt \; ds,
\end{align*}
and 
\begin{align*}
\left| \int_s^T \phi(t) f(t-s)\; dt \right| \leq 
\|\phi\|_{L^{q'}(s,T)}\|f\|_{L^q(0,T-s)}\leq 
\|\phi\|_{L^{q'}(0,T)}\|f\|_{L^q(0,T)}\quad\mbox{for }s\in [0,T],
\end{align*}
which means that $s\in [0,T]\mapsto\int_s^T \phi(t) f(t-s)\; dt\in\R$ is in $L^\infty(0,T)$. Given that $\lambda^{(\tau)}-\lambda\rightharpoonup 0$ weakly in $L^1(0,T)$, we deduce that $\int_0^T I^{(\tau)}(t)\phi(t)dt\to 0$ for every $\phi\in L^{q'}(0,T)$, that is, $I^{(\tau)}\rightharpoonup 0$ weakly in $L^q(0,T)$.
Since we already knew that $I^{(\tau)}$ is relatively compact in $L^q(0,T)$, we conclude that $I^{(\tau)}\to 0$ strongly in $L^q(0,T)$ and therefore
\begin{align*}
F_1^{(\tau)}\to \int_0^t \lambda(s)  f(t-s)\; ds\quad\mbox{strongly in }L^q(0,T).
\end{align*}
However, the above convergence is also strong in $L^r(0,T)$, since
$F_1^{(\tau)}$ is bounded in $L^{\tilde{r}}(0,T)$ for every $1\leq\tilde{r}<\frac{q}{1-(1-\alpha)q}$ being the convolution of $\lambda^{(\tau)}$, which is bounded in $L^p(0,T)$ for every $1\leq p < 1/\alpha$, and $f^{(\tau)}$, which is bounded in $L^q(0,T)$.
Summarizing up we have
\begin{equation}\label{strong.conv.lambda.f}
\int_0^{\tau\lfloor t/\tau\rfloor} \lambda^{(\tau)}(s)  f^{(\tau)} (t-s)\; ds\rightarrow \int_0^t\lambda(s)  f(t-s)\; ds\quad\mbox{strongly in $L^r(0,T)$ }
\end{equation}
for every $1\leq r < \frac{q}{1-(1-\alpha)q}$,
for any sequence $f^{(\tau)} \rightarrow f$ strongly convergent in $L^q(0,T)$ for some $q\in [1,\infty)$.

Our last step is finding the value of $\lambda(t)$. 
Since 
$$
\lambda_j (n-j+1)^{\alpha -1} = \frac{1}{\tau}\int_{t_{j-1}}^{t_j}\lambda_j (n-j+1)^{\alpha -1}\; ds = \frac{1}{\tau}\int_0^{t_n} \lambda_j (n-j+1)^{\alpha -1}  \chf{(t_{j-1},t_j]}(s)\; ds,
$$
from \eqref{1} one obtains
\begin{align*}
1 = \sum_{j=1}^n\lambda_j(n-j+1)^{\alpha -1}
=&\ \frac{1}{\tau}\int_0^{t_n}\sum_{j=1}^n\lambda_j (n-j+1)^{\alpha -1}  \chf{(t_{j-1},t_j]}(s)\; ds\\
=&\ \frac{1}{\tau}\int_0^{t_n} \left( \sum_{j=1}^n\lambda_j   \chf{(t_{j-1},t_j]}(s) \right) \left( \sum_{j=1}^n (n-j+1)^{\alpha -1}  \chf{(t_{j-1},t_j]}(s) \right)\; ds\\
=&\ \frac{1}{\tau}\int_0^{t_n} \left( \sum_{j=1}^n\lambda_j   \chf{(t_{j-1},t_j]}(s) \right) \left( n+1-\lceil s/\tau \rceil \right)^{\alpha -1}\; ds ,
\end{align*}
which means, given the definition of $\lambda^{(\tau)}$,
\begin{equation}\label{lim=1}
\int_0^{t_n}\lambda^{(\tau)}(s) \left( t_n +\tau -\tau   \lceil s/\tau \rceil \right)^{\alpha -1}\; ds = 1.
\end{equation}
It holds
\begin{align*}
\chf{[0,t_n]}(s)\left( t_n +\tau -\tau   \lceil s/\tau \rceil \right)^{\alpha -1} &\geq \chf{[0,t-\tau]}(s)\left( t +\tau -s \right)^{\alpha -1},\\
\chf{[0,t_n]}(s)\left( t_n +\tau -\tau   \lceil s/\tau \rceil \right)^{\alpha -1}
&\leq \chf{[0,t_n-\tau]}(s)\left( t_n - s \right)^{\alpha -1}
+ \chf{[t_n-\tau,t_n]}(s)\left( t_n - s \right)^{\alpha -1}\\
&\leq \chf{[0,t-\tau]}(s)\left( t - \tau - s \right)^{\alpha -1}
+ \chf{[t_n-\tau,t_n]}(s)\left( t_n - s \right)^{\alpha -1},
\end{align*}
%
%\begin{align*}
%1 &\geq \int_0^{t_n}\lambda^{(\tau)}(s) \left( t_n +\tau -s \right)^{\alpha -1}\; ds \geq \int_0^{t-\tau}\lambda^{(\tau)}(s) \left( t +\tau -s \right)^{\alpha -1}\; ds ,\\
%1 &\leq \int_0^{t_n}\lambda^{(\tau)}(s) \left( t_n - s \right)^{\alpha -1}\; ds\\
%&= \int_0^{t_n-\tau}\lambda^{(\tau)}(s) \left( t_n - s \right)^{\alpha -1}\; ds
%+ \int_{t_n-\tau}^{t_n}\lambda^{(\tau)}(s) \left( t_n - s \right)^{\alpha -1}\; ds\\
%&\leq \int_0^{t-\tau}\lambda^{(\tau)}(s) \left( t - \tau - s \right)^{\alpha -1}\; ds + C\lambda_n ,
%\end{align*}
which means for 
\begin{align*}
 f_1^{(\tau)}(z) \equiv (z+\tau)^{\alpha-1}\chf{[\tau,T]}(z),\quad 
f_2^{(\tau)}(z) \equiv (z-\tau)^{\alpha-1}\chf{[\tau,T]}(z)
\end{align*}
that
\begin{align}\label{lim=1.2.a}
&f_1^{(\tau)}(t-s) \leq \chf{[0,t_n]}(s)\left( t_n +\tau -\tau   \lceil s/\tau \rceil \right)^{\alpha -1} \leq f_2^{(\tau)}(t-s)
+ \chf{[t_n-\tau,t_n]}(s)\left( t_n - s \right)^{\alpha -1}.
% & f_1^{(\tau)}(z) \equiv (z+\tau)^{\alpha-1}\chf{[\tau,T]}(z),\quad 
% f_2^{(\tau)}(z) \equiv (z-\tau)^{\alpha-1}\chf{[\tau,T]}(z).
%\nonumber
\end{align}
Estimates \eqref{lim=1}, \eqref{lim=1.2.a} and the fact that $\lambda^{(\tau)}\equiv \tau^{-\alpha}\lambda_n = \tau^{-\alpha}\lambda_{\lfloor t/\tau\rfloor}$ on $(t_n-\tau,t_n]$ imply
\begin{align}\label{lim=1.2}
&\int_0^t\lambda^{(\tau)}(s)f_1^{(\tau)}(t-s)ds \leq 1 \leq
\int_0^t\lambda^{(\tau)}(s)f_2^{(\tau)}(t-s)ds + C\lambda_{\lfloor t/\tau\rfloor},\quad \tau\leq t\leq T ,
\end{align}
with $\lambda_{\lfloor t/\tau\rfloor}\to 0$ for $\tau \to 0$.
We wish to show for $\tau \to 0$ that
\begin{align}\label{lim=1.3}
f_i^{(\tau)}\to |\cdot|^{\alpha-1}\quad\mbox{strongly in }L^q(0,T),\quad\forall q\in [1,1/(1-\alpha)),\quad i=1,2.
\end{align}
Clearly $f_i^{(\tau)}(z)\to |z|^{\alpha-1}$ for $\tau \to 0$ for every $z>0$. Furthermore,
if $q\in [1,1/(1-\alpha))$,
\begin{align*}
\int_0^T |f_1^{(\tau)}(z)|^q dz\leq
\int_0^T |f_2^{(\tau)}(z)|^q dz = \int_\tau^T (z-\tau)^{q(\alpha-1)}dz = 
\frac{(T-\tau)^{q(\alpha-1)+1}}{q(\alpha-1)+1}\leq
\frac{T^{q(\alpha-1)+1}}{q(\alpha-1)+1}.
\end{align*}
Being $q\in [1,1/(1-\alpha))$ arbitrary, we deduce via dominated convergence
that \eqref{lim=1.3} holds.
From \eqref{lim=1.2} 
%and the fact that $\lambda_k\to 0$ as $k\to\infty$, 
we conclude that 

\begin{equation}\label{lambdaeta=1}
\int_0^t\lambda(s) (t-s)^{\alpha -1}\; ds = 1.
\end{equation}
Performing the Laplace transform on both sides of \eqref{lambdaeta=1} yields
\begin{equation*}
\int_0^{+\infty}\int_0^t e^{-kt}\lambda(s)(t-s)^{\alpha -1}\; dsdt = \int_0^{+\infty}e^{-kt}\; dt = \frac{1}{k}.
\end{equation*}
Then, interchanging the order of the integrals yields
\begin{equation*}
\int_0^{+\infty} e^{-ks}\lambda(s) \int_s^{+\infty} e^{-k(t-s)}(t-s)^{\alpha -1}\; dtds =  \frac{1}{k},
\end{equation*}
which indicates
\begin{equation*}
\int_0^{+\infty} e^{-ks}\lambda(s)\; ds = \frac{1}{k}\left( \int_0^{+\infty}e^{-ks}s^{\alpha -1}\; ds \right) ^{-1} = k^{\alpha -1} \Gamma_{\alpha}^{-1}.
\end{equation*}
Performing the inverse Laplace transform on the above equation leads to
\begin{equation*}
\lambda(t) = \frac{t^{-\alpha}}{\Gamma_{\alpha} \Gamma_{1-\alpha}},
\end{equation*}
which finishes the proof.
\end{proof}

%%%%%%%%%%%%%%%%%%%iffalse to thesis
\iffalse
\begin{remark}[\red{to be removed}]
A direct computation from \eqref{dis_left_caputo1} yields the following equivalent formula
\begin{equation}
\label{dis_left_caputo2}
(D^\alpha_\tau f)_{k} = \Gamma_\alpha\tau^{-\alpha}  \left( \lambda_k  (f_k-f_0) + \sum_{j=0}^{k-1}
(\lambda_{k-j}-\lambda_{k-j+1})  (f_k-f_j)\right),\quad k\geq 1,
\end{equation}
which is a discrete counterpart of \eqref{left_caputo2}.
\end{remark}
\fi
%%%%%%%%%%%%%%%%%%%fi to thesis

%The fractional integral of the continuous Caputo derivative recovers the original function, which is called a fractional version of fundamental theorem of calculus. 
Next we deal with the fractional version of fundamental theorem of calculus \cite{li2018some}: 
\begin{equation}
f(t) = f(0) + \frac{1}{\Gamma_{\alpha}}\int_0^t (t-s)^{\alpha -1}  D^{\alpha}_tf(s)\ ds,
\label{thm:funthm_calc_cont}
\end{equation}
and
\begin{equation}\label{thm:funthm_calc_cont2}
f(t) = f(T) - \frac{1}{\Gamma_{\alpha}}\int_t^T (s-t)^{\alpha -1}  {}^*D^{\alpha}_tf(s)\ ds.
\end{equation}
In the following proposition we propose a discrete version of \eqref{thm:funthm_calc_cont} and \eqref{thm:funthm_calc_cont2}.
\begin{proposition}[Discrete fundamental theorem of calculus]\label{prop:discreteFToC}
Let $f_n:=f(t_n)$ for $n=1, ... ,N$. We have the following identities:
\begin{align}\label{eq:discreteFToC}
f_n = f_0 + \frac{\tau^\alpha}{\Gamma_\alpha}\sum_{k=1}^n (n-k+1)^{\alpha-1}(D^\alpha_\tau f)_k,\quad
n\geq 1;
\end{align}
and
\begin{align}\label{eq:discreteFToC2}
f_n = f_N - \frac{\tau^\alpha}{\Gamma_\alpha}\sum_{k=n}^{N-1} (k-n+1)^{\alpha-1}({}^*D^\alpha_\tau f)_k,\quad
n\leq N-1.
\end{align}	
\end{proposition}
\begin{proof}
From (\ref{1}) one obtains
\begin{align*}
&\frac{\tau^{\alpha}}{\Gamma_{\alpha}} \sum_{k=1}^n (n-k+1)^{\alpha-1}
(D^\alpha_\tau f)_k 
= \sum_{k=1}^n (n-k+1)^{\alpha-1}
\sum_{j=0}^{k-1}\lambda_{k-j}(f_{j+1}-f_j)\\
&= \sum_{j=0}^{n-1} \sum_{k=j+1}^{n}
(n-k+1)^{\alpha-1}\lambda_{k-j}(f_{j+1}-f_j)\\
&= \sum_{j=0}^{n-1} \sum_{k=1}^{n-j}
(n-j-k+1)^{\alpha-1}\lambda_{k}(f_{j+1}-f_j) = 
\sum_{j=0}^{n-1}(f_{j+1}-f_j)
\end{align*}
which implies \eqref{eq:discreteFToC}. Furthermore,
\begin{align*}
&\frac{\tau^{\alpha}}{\Gamma_{\alpha}} \sum_{k=n}^{N-1} (k-n+1)^{\alpha-1}
({}^*D^\alpha_\tau f)_k 
= \sum_{k=n}^{N-1} (k-n+1)^{\alpha-1}
\sum_{j=k+1}^N\lambda_{j-k}(f_j-f_{j-1})\\
&= \sum_{j=n+1}^N \sum_{k=n}^{j-1}
(k-n+1)^{\alpha-1}\lambda_{j-k}(f_j-f_{j-1})\\
&= \sum_{j=n+1}^N \sum_{k=1}^{j-n}
(j-n-k+1)^{\alpha-1}\lambda_{k}(f_j-f_{j-1}) = 
\sum_{j=n+1}^N(f_j-f_{j-1})
\end{align*}
which implies \eqref{eq:discreteFToC2}. This concludes the proof.
\end{proof}

In the next proposition we show that the discrete Caputo derivative of an absolute continuous function converges, in a weak sense, to the continuous Caputo derivative. The proof is a straightforward consequence of Proposition \ref{prop:discreteFToC}. 

\begin{proposition}[Limit $\tau\to 0$]\label{conv.dis.caputo} Let $X, Y$ be Banach space, 
$f : [0,T]\to X$ absolutely continuous. Assume that $F_\tau\equiv \sum_{k=1}^N (D^\alpha_\tau f)_k\chf{[t_{k-1},t_k)}$ is weakly convergent in $L^p(0,T; Y)$ to some limit $\xi$ with $p>1/\alpha$. Then $\xi = D_t^\alpha f$.
\end{proposition}
\begin{proof}
From \eqref{eq:discreteFToC} it follows
\begin{align*}
f(t_n) &= f(0) + \frac{\tau^{\alpha-1}}{\Gamma_\alpha}\int_0^{t_n}\sum_{k=1}^n (n-k+1)^{\alpha-1}(D^\alpha_\tau f)_k\chf{[t_{k-1},t_k)}(s)\; ds\\
&= f(0) + \frac{1}{\Gamma_\alpha}\int_0^{t_n}\sum_{k=1}^n 
(t_n -t_k + \tau)^{\alpha-1}(D^\alpha_\tau f)_k\chf{[t_{k-1},t_k)}(s)\; ds\\
&= f(0) + \frac{1}{\Gamma_\alpha}\int_0^{t_n} \left(
\sum_{j=1}^n (t_n -t_j + \tau)^{\alpha-1}\chf{[t_{j-1},t_j)}(s)
\right) \left(  
\sum_{k=1}^n (D^\alpha_\tau f)_k\chf{[t_{k-1},t_k)}(s)\right)\; ds\\
&= f(0) + \frac{1}{\Gamma_\alpha}\int_0^{t_n} \left(
\sum_{j=1}^n (t_n -t_j + \tau)^{\alpha-1}\chf{[t_{j-1},t_j)}(s)
\right) F_\tau(s)\; ds,
\end{align*}
since the characteristic functions have disjoint support. We now pass to the limit $\tau \to 0$ in the above identity. From \eqref{lim=1.2.a}, \eqref{lim=1.3} we know that 
\begin{align*}%\label{strong.conv.cont.kernel}
\sum_{j=1}^n (t_n -t_j + \tau)^{\alpha-1}\chf{[t_{j-1},t_j)}\to 
(t- s  )^{\alpha-1}
\quad\mbox{strongly in $L^{p/(p-1)}(0,t)$},
\end{align*}
for $p>\frac{1}{\alpha}$.  Since $F_\tau\rightharpoonup \xi$ weakly in $L^p(0,T)$ by assumption, we deduce that
\begin{align*}
f(t) = f(0) + \frac{1}{\Gamma_\alpha}\int_0^t (t-s)^{\alpha-1}\xi,\quad t\in [0,T],
\end{align*}
which means, see \eqref{thm:funthm_calc_cont}, that $\xi = D^\alpha f$ a.e.~in $[0,T]$. This finishes the proof of the Proposition.
\end{proof}

We now recall three equivalent integration by parts formulas for \eqref{left_caputo1}. The first one \cite{ALMEIDA2012142} reads as
\begin{equation}
\int_0^TD^{\alpha}_tf(t) \phi(t)\ dt = -\frac{1}{\Gamma_{1-\alpha}}\int_0^Tf(t) \frac{d}{dt}\left(\int_t^T\frac{\phi(s)}{(s-t)^{\alpha}}\ ds\right)dt
-\frac{f(0)}{\Gamma_{1-\alpha}}\int_0^T\frac{\phi(s)}{s^{\alpha}}\ ds.\label{byparts.almeida}
\end{equation}
The second one can be found in \cite{allen2016parabolic}:
\begin{align}
\int_0^TD^{\alpha}_tf(t)  \phi(t)\ dt =& -\int_0^TD^{\alpha}_t\phi(t)  f(t)\ dt\label{byparts.allen} 
+ \frac{1}{\Gamma_{1-\alpha}}\int_0^T \phi(t)f(t)\left[ \frac{1}{(T-t)^{\alpha}}+\frac{1}{t^{\alpha}}\right]\ dt\\
&+ \frac{\alpha}{\Gamma_{1-\alpha}} \int_0^T\int_0^t \frac{(\phi(t)-\phi(s))(f(t)-f(s))}{(t-s)^{1+\alpha}}\ dsdt\nonumber
- \frac{1}{\Gamma_{1-\alpha}}\int_0^T \frac{\phi(t)f(0)+f(t)\phi(0)}{t^{\alpha}}\ dt.\nonumber
\end{align}
The last one reads as 
\begin{equation}\label{byparts.proposed}
\int_0^TD^{\alpha}_tf(t) \phi(t)\ dt = -\int_0^Tf(t)  {}^*D^{\alpha}_t\phi(t)\; dt +\frac{\phi(T)}{\Gamma_{1-\alpha}} \int_0^T \frac{f(t)}{(T-t)^{\alpha}}\; dt - \frac{f(0)}{\Gamma_{1-\alpha}}  \int_0^T \frac{\phi(s)}{s^{\alpha}}\; ds.
\end{equation}
The three formulas are equivalent; each one can be derived from the others using integration by parts. In the second part of this manuscript we will use all three. We will also use a discrete integration by parts formula, provided in the next proposition.
\begin{proposition}\label{prop:intbyparts_dis}
Let $\phi(t)$ be a bounded function in $(0,T]$, and $f^{(\tau)}(t)$ be a piecewise constant function defined by $f^{(\tau)}(t) = f_0\chf{\{0\}}(t)+ \sum_{k=1}^N f_k \chf{(t_{k-1},t_k]}(t)$.
Then we have the following discrete integration by parts formula:
\begin{align}\label{fml:intbyparts_dis}
    \int_0^T D^{\alpha}_{\tau}f^{(\tau)}(t)  \phi(t)\ dt
    =& -\int_0^{T-\tau} f^{(\tau)}(t)  {}^*D^{\alpha}_{\tau}\left(D^1_{\tau}\Phi\right)(t)\; dt \notag\\
    &+\frac{\Gamma_{\alpha}}{\tau^{\alpha}}  \Phi_N \sum_{j=1}^N \lambda_{N-j+1}  f_j - \frac{\Gamma_{\alpha}}{\tau^{\alpha}}  f_0  \sum_{k=1}^N \lambda_k  \Phi_k,
\end{align}
where $\Phi(t) := \int_0^t \phi(s)\; ds$ and $\Phi_k:=\int_{(k-1)\tau}^{k\tau}\phi(t)\ dt$.
\end{proposition}
\begin{proof}
%We assume that $t\in ((k-1)\tau,k\tau]$, thus $\lceil t/\tau \rceil=k$. 
It holds
\begin{align*}
   \int_0^T D^{\alpha}_{\tau}f^{(\tau)}(t)  \phi(t)\ dt=&\frac{\Gamma_{\alpha}}{\tau^{\alpha}} \sum\limits_{k=1}^{N}\left(\int_{(k-1)\tau}^{k\tau}\sum\limits_{j=0}^{k-1}\lambda_{k-j}(f_{j+1}-f_j)  \phi(t)\ dt\right)\\
   =&\frac{\Gamma_{\alpha}}{\tau^{\alpha}} \sum\limits_{k=1}^{N}\sum\limits_{j=0}^{k-1}\lambda_{k-j}(f_{j+1}-f_j)  \int_{(k-1)\tau}^{k\tau}\phi(t)\ dt.
\end{align*}
We have
$\left(D^1_{\tau} \Phi\right)_k= \left(\Phi(k\tau) - \Phi((k-1)\tau)\right)/\tau = \Phi_k/\tau .$
Now rewrite the double summation as follows:
\begin{align*}
   \sum_{k=1}^{N} \sum\limits_{j=0}^{k-1}\lambda_{k-j}  (f_{j+1}-f_j)\Phi_k
   =&\sum_{k=1}^{N} \sum\limits_{j=1}^{k}\lambda_{k-j+1}  f_j\Phi_k-\sum_{k=1}^{N} \sum\limits_{j=0}^{k-1}\lambda_{k-j}f_j\Phi_k\\
   =&\sum_{k=2}^{N+1}\sum_{j=1}^{k-1} \lambda_{k-j}  f_j  \Phi_{k-1} - \sum_{k=1}^N\sum_{j=0}^{k-1} \lambda_{k-j}  f_j  \Phi_k\\
   =&\sum_{k=2}^N\sum_{j=1}^{k-1} \lambda_{k-j}  f_j   (\Phi_{k-1}-\Phi_k) +\sum_{j=1}^N \lambda_{N-j+1}  f_j  \Phi_N - \sum_{k=1}^N \lambda_k   f_0  \Phi_k\\
   =&\sum_{k=1}^{N-1}f_k \sum_{j=k+1}^N \lambda_{j-k}  (\Phi_{j-1}-\Phi_j) + \Phi_N \sum_{j=1}^N \lambda_{N-j+1}  f_j - f_0  \sum_{k=1}^N \lambda_k  \Phi_k.
\end{align*}
%Substituting $\Phi_k=\int_{(k-1)\tau}^{k\tau}\phi(s)\ ds$ in above identity, we have
Hence
\begin{align*}
    \int_0^T D^{\alpha}_{\tau}f^{(\tau)}(t)  \phi(t)\ dt
    =&\;\frac{\Gamma_{\alpha}}{\tau^{\alpha}}  \sum_{k=1}^{N-1}f_k \sum_{j=k+1}^N \lambda_{j-k}  (\Phi_{j-1}-\Phi_j)\\
    &+\frac{\Gamma_{\alpha}}{\tau^{\alpha}}  \Phi_N \sum_{j=1}^N \lambda_{N-j+1}  f_j - \frac{\Gamma_{\alpha}}{\tau^{\alpha}}  f_0  \sum_{k=1}^N \lambda_k  \Phi_k\\
    =& -\int_0^{T-\tau} f^{(\tau)}(t)  {}^*D^{\alpha}_{\tau}\left(D^1_{\tau}\Phi\right)(t)\; dt \notag\\
    &+\frac{\Gamma_{\alpha}}{\tau^{\alpha}}  \Phi_N \sum_{j=1}^N \lambda_{N-j+1}  f_j - \frac{\Gamma_{\alpha}}{\tau^{\alpha}}  f_0  \sum_{k=1}^N \lambda_k  \Phi_k.
\end{align*}

\end{proof}

We have the following convergence theorem. 

\begin{proposition}\label{prop.conv.dis.byparts}
Provided $f^{(\tau)}(t):=\sum_{j=1}^N f_j  \chf{(t_{j-1},t_j]}(t)$ is strongly convergent in $L^{1}(0,T)$, for any smooth function $\phi\in C^1(0,T)$ we have 
\begin{align*}
\int_0^TD^{\alpha}_{\tau}f^{(\tau)}(t) \phi(t)\; dt \rightarrow & -\int_0^Tf(t)  {}^*D^{\alpha}_t\phi(t)\; dt\ +\ \frac{\phi(T)}{\Gamma_{1-\alpha}} \int_0^T \frac{f(t)}{(T-t)^{\alpha}}\; dt\ - \  \frac{f_0}{\Gamma_{1-\alpha}}  \int_0^T \frac{\phi(s)}{s^{\alpha}}\; ds,
\end{align*}
as $\tau\rightarrow 0$. Moreover, if function $f$ is smooth enough, it holds
$$\int_0^TD^{\alpha}_{\tau}f^{(\tau)}(t) \phi(t)\; dt \rightarrow \int_0^T D^{\alpha}_tf(t)\phi(t)\; dt, \quad\mbox{as $\tau\rightarrow 0$.}$$
\begin{proof}
Let $\Phi(t) := \int_0^t \phi(s)\; ds$ and $\Phi_k:=\int_{(k-1)\tau}^{k\tau}\phi(t)\ dt$ as in \Cref{prop:intbyparts_dis}. Proving the above convergence  is equivalent to showing that 
\begin{equation*}\label{dis.byparts.conv123}
-\int_0^{T-\tau} f^{(\tau)}(t)  {}^*D^{\alpha}_{\tau}\left(D^1_{\tau}\Phi\right)(t)\; dt\ +\ \frac{\Gamma_{\alpha}}{\tau^{\alpha}}  \Phi_N \sum_{j=1}^N \lambda_{N-j+1}  f_j\ - \ \frac{\Gamma_{\alpha}}{\tau^{\alpha}}  f_0  \sum_{k=1}^N \lambda_k  \Phi_k
\end{equation*}
\mbox{converges to }
\begin{equation*} -\int_0^Tf(t)  {}^*D^{\alpha}_t\phi(t)\; dt\ +\ \frac{\phi(T)}{\Gamma_{1-\alpha}} \int_0^T \frac{f(t)}{(T-t)^{\alpha}}\; dt\ - \  \frac{f_0}{\Gamma_{1-\alpha}}  \int_0^T \frac{\phi(s)}{s^{\alpha}}\; ds,
\end{equation*}
as $\tau\rightarrow 0$. To see why, one should compare \eqref{byparts.proposed} and \eqref{fml:intbyparts_dis}. 
Using (\ref{est_sum_lambda}), we have 
\begin{align*}
 \norm{{}^*D^{\alpha}_{\tau}\left(D^1_{\tau}\Phi\right)}_{L^\infty(0,T)} =& 
 \max_{0\leq k\leq N-1} \left| \sum_{j=k+1}^N\tau^{1-\alpha}\lambda_{j-k}(D^2_{\tau}\Phi)_j \right| \\%\right)^p\chf{(k\tau, (k+1)\tau)}(t)\; dt\\
\leq & \sup\limits_{t\in[0,T]}|D^2_t\Phi|   \tau^{1-\alpha} 
\max_{0\leq k\leq N-1}\sum_{j=k+1}^N\lambda_{j-k}\\
= & \sup\limits_{t\in[0,T]}|D^2_t\Phi|   \tau^{1-\alpha} \max_{0\leq k\leq N-1}\sum_{j=1}^{N-k}\lambda_{j} \\
\le&  \sup\limits_{t\in[0,T]}|D^2_t\Phi|   \tau^{1-\alpha} \max_{0\leq k\leq N-1}(N-k)^{1-\alpha}\\
\le&  \norm{\phi}_{C^1(0,T)}  T^{1-\alpha}.
\end{align*}
Hence ${}^*D^{\alpha}_{\tau}\left(D^1_{\tau}\Phi\right)\rightharpoonup^*\xi$
weakly* in $L^\infty(0,T)$. In particular 
${}^*D^{\alpha}_{\tau}\left(D^1_{\tau}\Phi\right)\rightharpoonup\xi$
weakly in $L^p(0,T)$ for every $p<\infty$, so we can employ \Cref{conv.dis.caputo} to identify the limit: $\xi = {}^*D^{\alpha}_t\phi$.
Since $f^{(\tau)}\to f$ strongly in $L^1(0,T)$ we deduce
\begin{equation*}\label{dis.bypart.conv1}
\int_0^{T-\tau} f^{(\tau)}(t)  {}^*D^{\alpha}_{\tau}\left(D^1_{\tau}\Phi\right)(t)\; dt\rightarrow \int_0^Tf(t)  {}^*D^{\alpha}_t\phi(t)\; dt,\ \textrm{ as }\tau \rightarrow 0.
\end{equation*}
With $\lambda^{(\tau)}(s):=\tau^{-\alpha}\sum_{j=1}^n\lambda_j \chf{(t_{j-1},t_j]}(s)$ we have
\begin{align*}
\frac{\Gamma_{\alpha}}{\tau^{\alpha}}  \Phi_N \sum_{j=1}^N \lambda_{N-j+1}  f_j
=& \ \Gamma_{\alpha} \int_{T-\tau}^{T}\frac{\phi(s)}{\tau}\; ds  \int_0^T \lambda^{(\tau)}(T-t)   f^{(\tau)}(t)\; dt,\notag
\end{align*}
which converges to
\begin{equation*}
\frac{\phi(T)}{\Gamma_{1-\alpha}} \int_0^T \frac{f(t)}{(T-t)^{\alpha}}\; dt
\end{equation*}
as $\tau\rightarrow 0$ by \eqref{weak.conv.lambda.Convolution}. Finally, using \eqref{weak.conv.lambda} we get 
\begin{align}\label{dis.bypart.conv3}
\frac{\Gamma_{\alpha}}{\tau^{\alpha}}  f_0  \sum_{k=1}^N \lambda_k  \Phi_k
\rightarrow \ \frac{f_0}{\Gamma_{1-\alpha}}\int_0^T\frac{\phi(s)}{s^{\alpha}}\; ds.
\end{align}
This finishes the proof.
\end{proof}
\end{proposition}

The following two lemmas will be useful.
\begin{lemma}\label{lem:fact1}
If $(D^{\alpha}_{\tau}f)_k\leq0$ for $1\leq k\leq n$, then
$f_{k}\leq f_{0}$ for $1\leq k\leq n$.
\end{lemma}
\begin{proof}
It follows directly from \eqref{eq:discreteFToC}.
\end{proof}

\begin{lemma}\label{lem:fact2} We have
$$f_k  (D^{\alpha}_{\tau}f)_k\geq \frac{1}{2}(D^{\alpha}_{\tau}f^2)_k,\quad
k\geq 1. $$
\end{lemma} 
\begin{proof}
We have to show for $k \geq 1$ that
$$
\sum\limits_{j=0}^{k-1}\lambda_{k-j}({f_{j+1}-f_{j}})f_{k}\geq\frac{1}{2}\sum_{j=0}^{k-1}\lambda_{k-j}({f_{j+1}^{2}-f_{j}^{2}}).
$$
First, we rewrite the left hand side of the above inequality as
\begin{align*}
\sum\limits_{j=0}^{k-1}\lambda_{k-j}({f_{j+1}-f_{j}})f_{k}
=& -\lambda_{k}f_{0}f_{k}-\sum_{j=1}^{k-1}(\lambda_{k-j+1}-\lambda_{k-j})f_{j}f_{k}+\lambda_1f_{k}^{2}\\
=& \frac{1}{2}\lambda_k(f_{0}-f_{k})^{2}+\frac{1}{2}\sum_{j=1}^{k-1}(\lambda_{k-j+1}-\lambda_{k-j})(f_{j}-f_{k})^{2}+\lambda_1f_{k}^{2}\\
& -\frac{1}{2}\lambda_kf_{0}^{2}-\frac{1}{2}\lambda_kf_{k}^{2}-\frac{1}{2}\sum_{j=1}^{k-1}(\lambda_{k-j}-\lambda_{k-j+1})(f_{j}^{2}+f_{k}^{2}),
\end{align*}
which implies
\begin{align*}
\sum\limits_{j=0}^{k-1}\lambda_{k-j}({f_{j+1}-f_{j}})f_{k}\geq & -\frac{1}{2}\lambda_kf_{0}^{2}-\frac{1}{2}\lambda_kf_{k}^{2}-\frac{1}{2}\sum_{j=1}^{k-1}(\lambda_{k-j}-\lambda_{k-j+1})(f_{j}^{2}+f_{k}^{2})+\lambda_1f_{k}^{2}\\
= & \frac{1}{2}\sum_{j=0}^{k-1}\lambda_{k-j}({f_{j+1}^{2}-f_{j}^{2}}).
\end{align*}
Thus, the lemma is proved.
\end{proof}

\subsection{Compactness results}
In this section we prove \Cref{thm:mainII}.
%iffalse%%%%%%%%%%%%%%%%%%%%%%%%%%%%%%%%%%%%%%%%%%
\iffalse, which is an equivalent of Aubin-Lions lemma for the discrete Caputo derivative defined in \eqref{dis_left_caputo1}.
\fi
%fi%%%%%%%%%%%%%%%%%%%%%%%%%%%%%%%%%%%%%%%%%%%%%%%
For convenience we first recall the classical version of the Aubin-Lions lemma.
\begin{lemma}[Theorem 5 in \cite{simon1986compact}]\label{lem:simon}
Assume that $X$, $B$ and $Y$ are Banach spaces, with $X\hookrightarrow B\hookrightarrow Y$, where $X$ is compactly embedded in $B$. For $1\leq r\leq \infty$, assume that $F$ is bounded in $L^r(0,T;X)$ and $\norm{f(t+h,x)-f(t,x)}_{L^r(0,T-h;Y)}\rightarrow 0$ as $h \rightarrow 0$, uniformly for all $f\in F$. Then $F$ is relatively compact in $L^r(0,T;B)$ (and in $C([0,T];B)$ if $r=\infty$).
\end{lemma}

Before we start with the proof of \Cref{thm:mainII} we need to prove three technical lemmas. 
\begin{lemma}\label{lem.1}
For $0<\alpha<1$, it holds that
\begin{align}\label{dfDaf}
\|f( \cdot + h)-f\|_{L^p(0,T-h; Y)}\leq 
\frac{2 h^{\alpha}}{\Gamma_\alpha\alpha}\|D_t^\alpha f\|_{L^p(0,T; Y)} ,\quad h>0,
\end{align}
for every $1\leq p\leq \infty$ and every $f : [0,T]\to Y$ such that $f$, $D_t^\alpha f\in L^p(0,T; Y)$.	
\end{lemma}
\begin{lemma}\label{lem.2}
Let $Y$ be a Banach space, and let $f^{(\tau)} : [0,T]\to Y$
be a piecewise constant-in-time function, namely $f^{(\tau)}(t) = f_0\chf{\{0\}}(t) + \sum_{n=1}^N
 f_n\chf{(t_{n-1},t_n]}(t)$, with $t_j = j\tau$, $f_j\in Y$ ($j=0,\ldots,N$), and $T=N\tau$. The linear interpolant $\tft$ of $f_0,\ldots,f_N$ is given by
\begin{equation}\label{eq:lin.intplt}
\tft(t) = \sum_{n=0}^{N-1}\left( \frac{t-t_n}{\tau}(f_{n+1}-f_n) + f_n \right)\chf{[t_n,t_{n+1})}(t) .
\end{equation} 
Then for $0<\alpha<1$, a constant $C_\alpha>0$ indepedent of $f_0,\ldots,f_N$ exists such that
\begin{align*}
\|D_t^\alpha\tft\|_{L^p(0,T; Y)}\leq C_\alpha \|D^\alpha_\tau\ft\|_{L^p(0,T; Y)},\quad\tau>0,\quad 1\leq p\leq\infty. 
\end{align*}
%{\color{red}{for every $f : [0,T]\to Y$ Lebesgue measurable and a.e.~nonnegative.}}
\end{lemma}
\begin{lemma}\label{lem:tau jump}
Let $Y$ and $f^{(\tau)}$ be as in Lemma \ref{lem.2}. Then
\begin{equation}\label{eq:tau jump}
\norm{f^{(\tau)}(t+\tau)-f^{(\tau)}(t)}_{L^p(0,T-\tau ;Y)} \leq \frac{2^{1+1/p}   \tau^{\alpha}}{\Gamma_{\alpha}}   \norm{D^{\alpha}_\tau f^{(\tau)}}_{L^p(0,T;Y)},\quad\tau>0,\quad 1\leq p\leq\infty. 
\end{equation}
%where $D^{\alpha}_{\tau}f^{(\tau)}$ is defined in \eqref{eq:disCaputo3}.
\end{lemma}

%We will use \Cref{lem:simon}--\Cref{lem:tau jump} to show \Cref{thm:main}. But first we highlight their proofs. 
\begin{proof}[Proof of Lemma \ref{lem.1}]
Since the mapping $D_t^\alpha f\mapsto f( \cdot + h)-f$ is linear, 
Riesz-Thorin interpolation Theorem \cite[Thr.~II.4.2]{werner2006funktionalanalysis} implies that it is enough to prove \eqref{dfDaf} for $p=1$ and $p=\infty$. 

We start by recalling the fundamental theorem of calculus
\begin{equation*}
%\label{ftc}
f(t) = f_{0} + \frac{1}{\Gamma_\alpha}\int_0^t (t-s)^{\alpha-1}D_t^\alpha f(s)ds ,\quad 0\leq t\leq T .
\end{equation*}
Then
\begin{align*}
& f(t+h)-f(t) = \frac{1}{\Gamma_\alpha}\int_0^{t+h} (t+h-s)^{\alpha-1}D_t^\alpha f(s)ds - \frac{1}{\Gamma_\alpha}\int_0^{t} (t-s)^{\alpha-1}D_t^\alpha f(s)ds\\
&\quad = \frac{1}{\Gamma_\alpha}\int_0^{t} ((t+h-s)^{\alpha-1}-(t-s)^{\alpha-1}) D_t^\alpha f(s)ds
+\frac{1}{\Gamma_\alpha}\int_t^{t+h}(t+h-s)^{\alpha-1}D_t^\alpha f(s)ds\\
&\quad =: \mathcal I_1(t) + \mathcal I_2(t) .
\end{align*}
Let us estimate
\begin{align*}
\int_0^{T-h}\|\mathcal I_1(t)\|_Y dt & \leq 
\frac{1}{\Gamma_\alpha}\int_0^{T-h}\int_0^{t} |(t+h-s)^{\alpha-1}-(t-s)^{\alpha-1}| \|D_t^\alpha f(s)\|_Y  ds dt\\
&= \frac{1}{\Gamma_\alpha}\int_0^{T-h}\int_s^{T-h} ((t-s)^{\alpha-1} - (t+h-s)^{\alpha-1})\|D_t^\alpha f(s)\|_Y dt ds\\
&= \frac{1}{\Gamma_\alpha\alpha}\int_0^{T-h} 
((T-h-s)^{\alpha} - (T-s)^{\alpha} + h^\alpha )\|D_t^\alpha f(s)\|_Y ds .
\end{align*}
We deduce that
\begin{align*}
\int_0^{T-h}\|\mathcal I_1(t)\|_Y dt\leq \frac{h^{\alpha}}{\Gamma_\alpha\alpha}\|D_t^\alpha f\|_{L^1(0,T; Y)}.
\end{align*}
On the other hand
\begin{align*}
\int_0^{T-h}\|\mathcal I_2(t)\|_Y dt
&\leq
\frac{1}{\Gamma_\alpha}\int_0^{T-h}\int_t^{t+h}(t+h-s)^{\alpha-1}
\|D_t^\alpha f(s)\|_Y ds dt \\
&=\frac{1}{\Gamma_\alpha}\int_0^{T}\int_{s-h}^{s}(t+h-s)^{\alpha-1}
\|D_t^\alpha f(s)\|_Y dt ds\\
&=\frac{h^{\alpha}}{\Gamma_\alpha\alpha}
\int_0^{T} \|D_t^\alpha f(s)\|_Y ds = \frac{h^{\alpha}}{\Gamma_\alpha\alpha}\|D_t^\alpha f\|_{L^1(0,T; Y)}.
\end{align*}
This means that \eqref{dfDaf} holds for $p=1$. Let us now consider the case $p=\infty$:
\begin{align*}
\|\mathcal{I}_1(t)\|_{Y} &\leq \|D_t^\alpha f\|_{L^\infty(0,T; Y)}
\frac{1}{\Gamma_\alpha}\int_0^{t} ((t-s)^{\alpha-1}-(t+h-s)^{\alpha-1})ds\\
&= \|D_t^\alpha f\|_{L^\infty(0,T; Y)}
\frac{1}{\Gamma_\alpha\alpha}( t^\alpha + h^{\alpha} - (t+h)^{\alpha})\\
&\leq \frac{h^\alpha}{\Gamma_\alpha\alpha}\|D_t^\alpha f\|_{L^\infty(0,T; Y)} .
\end{align*}
On the other hand,
\begin{align*}
\|\mathcal{I}_2(t)\|_{Y} &\leq
\|D_t^\alpha f\|_{L^\infty(0,T; Y)}
\frac{1}{\Gamma_\alpha}\int_t^{t+h}(t+h-s)^{\alpha-1}ds
= \frac{h^\alpha}{\Gamma_\alpha\alpha}
\|D_t^\alpha f\|_{L^\infty(0,T; Y)}.
\end{align*}
Therefore \eqref{dfDaf} holds also for $p=\infty$.
This finishes the proof of the Lemma.
\end{proof}
\begin{proof}[Proof of Lemma \ref{lem.2}]	
Let $t\in (0,T)$ arbitrary, $k = \lfloor t/\tau\rfloor$.
Let us begin by computing
\begin{align*}
D_t^\alpha\tft(t) &= \frac{1}{\Gamma_{1-\alpha}}\int_0^t\frac{D\tft(s)}{(t-s)^\alpha}ds\\
&= \frac{1}{\Gamma_{1-\alpha}}\int_0^t (t-s)^{-\alpha}
\sum_{n=0}^{N-1}\frac{f_{n+1} -f_n }{ \tau }
\chf{[t_n,t_{n+1})}(s)ds\\
&= \frac{1}{\Gamma_{1-\alpha}}
\sum_{n=0}^{k-1}\frac{f_{n+1} -f_n }{ \tau }
\int_{t_n}^{t_{n+1}} (t-s)^{-\alpha}ds + 
\frac{1}{\Gamma_{1-\alpha}}
\frac{f_{k+1} -f_k }{ \tau }
\int_{t_k}^{t} (t-s)^{-\alpha}ds\\
&= \frac{1}{\Gamma_{1-\alpha}}
\sum_{n=0}^{k-1}\frac{f_{n+1} -f_n }{ (1-\alpha)\tau }
( (t-t_n)^{1-\alpha} - (t-t_{n+1})^{1-\alpha} ) + 
\frac{1}{\Gamma_{1-\alpha}}
\frac{f_{k+1} -f_k}{(1-\alpha)\tau}
(t-t_k)^{1-\alpha},
\end{align*}
and so
\begin{align}
\label{Dtft}
D_t^\alpha\tft(t) &= \frac{\tau^{-\alpha}}{\Gamma_{1-\alpha}}
\sum_{n=0}^{N-1}\frac{f_{n+1} -f_n }{ (1-\alpha) }
( g_{n}(t) - g_{n+1}(t) ) ,\quad
0\leq t\leq T ,\\
g_{n}(t) &:= \tau^{\alpha-1}(t-t_n)_+^{1-\alpha},\quad n\geq 0.\nonumber
\end{align}
From \eqref{eq:discreteFToC} one obtains
\begin{align*}
f_{n+1} - f_n &= \frac{\tau^\alpha}{\Gamma_\alpha}\sum_{k=1}^{n+1}
w_{n+1-k}(D^\alpha_\tau f)_k,\quad n\geq 0,\\
w_s &:= \begin{cases}
1 & s=0\\
(s+1)^{\alpha-1} - s^{\alpha-1} & s\geq 1.
\end{cases} 
\end{align*}
Plugging the above identity into \eqref{Dtft} yields
\begin{align*}
D_t^\alpha\tft(t) &= \frac{1}{\Gamma_{1-\alpha}\Gamma_\alpha(1-\alpha)}
\sum_{n=0}^{N-1}( g_{n}(t) - g_{n+1}(t) ) 
\sum_{k=1}^{n+1} w_{n+1-k}(D^\alpha_\tau f)_k\\
&= \frac{1}{\Gamma_{1-\alpha}\Gamma_\alpha(1-\alpha)}
\sum_{k=1}^{N} \sum_{n=k-1}^{N-1}( g_{n}(t) - g_{n+1}(t) ) 
w_{n+1-k}(D^\alpha_\tau f)_k,
\end{align*}
which can be rewritten as 
\begin{align}
\label{DD}
D_t^\alpha\tft(t) &=
\frac{1}{\Gamma_{1-\alpha}\Gamma_\alpha(1-\alpha)}
\sum_{k=1}^{N}\mu_k(t)(D^\alpha_\tau f)_k ,\\
\nonumber
\mu_k(t) &:= \sum_{s=0}^{N-k}w_s (g_{s+k-1}(t) - g_{s+k}(t)) ,\quad 1\leq k\leq N.
\end{align}
We aim to show that a constant $C>0$ exists, depending only on $\alpha\in (0,1)$ and $T>0$, such that $\|\mu_k\|_{L^\infty(0,T)}\leq C$ for $1\leq k\leq N$, $\tau>0$. From this fact the statement of the Lemma follows easily.

Let $\rho \equiv t/\tau \in [0,N]$.
From the definitions of $w_s$ and $g_n(t)$ it follows
\begin{align*}
\mu_k(t) =& (\rho -k+1)_+^{1-\alpha} - (\rho -k)_+^{1-\alpha}\\
&+\sum_{s=1}^{N-k}(s+1)^{\alpha-1}(
(\rho -s-k+1)_+^{1-\alpha} - (\rho -s-k)_+^{1-\alpha})\\
&-\sum_{s=1}^{N-k} s^{\alpha-1}(
(\rho -s-k+1)_+^{1-\alpha} - (\rho -s-k)_+^{1-\alpha})\\
=& (N-k+1)^{\alpha-1}( (\rho-N+1)_+^{1-\alpha} - (\rho-N)_+^{1-\alpha})\\
&+\sum_{s=1}^{N-1}s^{\alpha-1}(
(\rho-s-k+2)_+^{1-\alpha} - 2(\rho-s-k+1)_+^{1-\alpha}
+(\rho-s-k)_+^{1-\alpha} ) .
\end{align*}
Since $0\leq\rho\leq N$ and $1\leq k\leq N$, it holds
\begin{align*}
&(N-k+1)^{\alpha-1} ( (\rho-N+1)_+^{1-\alpha} - (\rho-N)_+^{1-\alpha})\\ 
&\quad = (N-k+1)^{\alpha-1}(\rho-N+1)_+^{1-\alpha}\leq 1.
\end{align*}
It follows
\begin{equation}
\label{est.mu}
|\mu_k(t)|\leq 1 + \sum_{s=1}^{N-1}s^{\alpha-1}\left|
(\rho-s-k+2)_+^{1-\alpha} - 2(\rho-s-k+1)_+^{1-\alpha}
+(\rho-s-k)_+^{1-\alpha}\right| .
\end{equation}
We can clearly consider $\rho>s+k-2$, since for $\rho\leq s+k-2$ the sum on the right-hand side of \eqref{est.mu} vanishes. Let us distinguish two cases.\medskip\\
{\em Case 1: $s+k-2<\rho\leq s+k+1$.} Being $s$, $k$ integers, this means that
$\lfloor\rho\rfloor - k\leq s \leq \lfloor\rho\rfloor - k + 2$.
It follows
\begin{align*}
&\sum_{s=\max\{1, \lfloor\rho\rfloor - k\}}^{\lfloor\rho\rfloor - k + 2} s^{\alpha-1}|
(\rho-s-k+2)_+^{1-\alpha} - 2(\rho-s-k+1)_+^{1-\alpha}
+(\rho-s-k)_+^{1-\alpha}|\\
&\leq \sum_{s=\max\{1, \lfloor\rho\rfloor - k\}}^{\lfloor\rho\rfloor - k + 2} s^{\alpha-1}
\big(
|(\rho-s-k+2)_+^{1-\alpha} - (\rho-s-k+1)_+^{1-\alpha}|\\ 
&\qquad \qquad \qquad \qquad 
+ |(\rho-s-k+1)_+^{1-\alpha} -(\rho-s-k)_+^{1-\alpha}|\big).
\end{align*}
Since $x\mapsto x^{1-\alpha}$ is subadditive, it holds
$x^{1-\alpha}\leq (x-y)^{1-\alpha} + y^{1-\alpha} $ for $x\geq y\geq 0$, so
\begin{align*}
|(\rho-s-k+2)_+^{1-\alpha} - (\rho-s-k+1)_+^{1-\alpha}|\leq
( (\rho-s-k+2)_+ - (\rho-s-k+1)_+ )^{1-\alpha} \leq 1,\\
|(\rho-s-k+1)_+^{1-\alpha} - (\rho-s-k)_+^{1-\alpha}|\leq
( (\rho-s-k+1)_+ - (\rho-s-k)_+ )^{1-\alpha} \leq 1.
\end{align*}
It follows
\begin{align*}
&\sum_{s=1}^{N-1} s^{\alpha-1}|
(\rho-s-k+2)_+^{1-\alpha} - 2(\rho-s-k+1)_+^{1-\alpha}
+(\rho-s-k)_+^{1-\alpha}|\\
&\leq 2\sum_{s=\max\{1, \lfloor\rho\rfloor - k\}}^{\lfloor\rho\rfloor - k + 2} s^{\alpha-1} \leq 6.
\end{align*}
{\em Case 2: $\rho > s+k+1$.} This implies $s\leq \lfloor\rho\rfloor - k - 1$.
Lagrange's theorem yields
\begin{align*}
&(\rho-s-k+2)_+^{1-\alpha} - 2(\rho-s-k+1)_+^{1-\alpha}
+(\rho-s-k)_+^{1-\alpha} \\ 
&\quad = (\rho-s-k+2)^{1-\alpha} - 2(\rho-s-k+1)^{1-\alpha}
+(\rho-s-k)^{1-\alpha} \\
&\quad = -\alpha(1-\alpha)\xi^{-1-\alpha},\qquad
\rho-s-k\leq\xi\leq\rho-s-k+2 ,
\end{align*}
and so
\begin{align*}
\sum_{s=1}^{\lfloor\rho\rfloor - k - 1} s^{\alpha-1}
& |(\rho-s-k+2)_+^{1-\alpha} - 2(\rho-s-k+1)_+^{1-\alpha}
+(\rho-s-k)_+^{1-\alpha}| \\
&\leq \alpha(1-\alpha)\sum_{s=1}^{\lfloor\rho\rfloor - k - 1} s^{\alpha-1}(\rho-s-k)^{-1-\alpha} \\
&\leq \alpha(1-\alpha)\sum_{s=1}^{\lfloor\rho\rfloor - k - 1} 
(\lfloor\rho\rfloor - s - k)^{-1-\alpha} \\
&\leq \alpha(1-\alpha)\sum_{r=1}^\infty r^{-1-\alpha} < \infty .
\end{align*}
From \eqref{est.mu} the statement follows. This finishes the proof of the Lemma.
\end{proof}
\begin{proof}[Proof of Lemma \ref{lem:tau jump}]
From \eqref{eq:discreteFToC} we have
\begin{flalign*}
f_n = f_0 + \frac{\tau^\alpha}{\Gamma_\alpha}\sum_{k=1}^n (n-k+1)^{\alpha-1}(D^\alpha_\tau f)_k,\quad
n\geq 1.
\end{flalign*}
For convenience, denote $a_j := j^{\alpha - 1}$, then
$$f_{n+1}-f_n=\frac{\tau ^{\alpha}}{\Gamma_{\alpha}}\sum\limits_{k=1}^n(a_{n+2-k}-a_{n+1-k})  (D^{\alpha}_\tau f)_k + \frac{\tau ^{\alpha}}{\Gamma_{\alpha}} a_1  (D^{\alpha}_\tau f)_{n+1}.$$
Hence
\begin{equation*}\label{eq:I&II}
\renewcommand\arraystretch{1.5}
\begin{array}{ll}
\norm{f_{n+1}-f_n}_Y^p
&=\norm{\frac{\tau ^{\alpha}}{\Gamma_{\alpha}}\sum\limits_{k=1}^n(a_{n+2-k}-a_{n+1-k})  (D^{\alpha}_\tau f)_k + \frac{\tau ^{\alpha}}{\Gamma_{\alpha}} a_1  (D^{\alpha}_\tau f)_{n+1}}_Y^p\\
&\leq (2/\Gamma_{\alpha})^p  \Big[ \Big(  \tau ^{\alpha}\sum\limits_{k=1}^n(a_{n+1-k}-a_{n+2-k})  \norm{(D^{\alpha}_\tau f)_k}_Y\Big) ^p+\Big(  \tau ^{\alpha} a_1  \norm{(D^{\alpha}_\tau f)_{n+1}}_Y \Big) ^p \Big],
\end{array}
\end{equation*}
thanks to the triangular and Young's inequality.
Then we have the following bounds:
\begin{equation*}
\begin{array}{ll}
&\Big(  \tau ^{\alpha}\sum\limits_{k=1}^n(a_{n+1-k}-a_{n+2-k})  \norm{(D^{\alpha}_\tau f)_k}_Y\Big) ^p\\
=&\tau ^{(\alpha -1)p}   \left(  \sum\limits_{k=1}^n \tau (a_{n+1-k}-a_{n+2-k})  \norm{(D^{\alpha}_\tau f)_k}_Y\right) ^p\\
\leq &\tau ^{(\alpha -1)p}   \left( \sum\limits_{k=1}^n\tau (a_{n+1-k}-a_{n+2-k})\right) ^{p-1}   \left( \sum\limits_{k=1}^n\tau(a_{n+1-k}-a_{n+2-k})  \norm{(D^{\alpha}_\tau f)_k}_Y^p \right)\\
=&\tau ^{(\alpha -1)p}   \left( \tau (a_1-a_{n+1})\right) ^{p-1}   \left( \sum\limits_{k=1}^n\tau(a_{n+1-k}-a_{n+2-k})  \norm{(D^{\alpha}_\tau f)_k}_Y^p \right)\\
\leq &\tau ^{\alpha p}   \sum\limits_{k=1}^n (a_{n+1-k}-a_{n+2-k})  \norm{(D^{\alpha}_\tau f)_k}_Y^p,
\end{array}
\end{equation*}
where in the first inequality we have used a discrete version of H\"older's inequality:
\begin{equation*}
\begin{array}{ll}
\norm{g   h}_{L^1}
%&=\norm{f^{1-\frac{1}{p}}   f^{\frac{1}{p}}   g}_{L^1}\\
%&\leq \norm{f^{1-\frac{1}{p}}}_{L^{\frac{p}{p-1}}}   \norm{f^{\frac{1}{p}}  g}_{L^p}\\
&\leq\norm{g}_{L^1}^{\frac{p-1}{p}}   \norm{g  h^p}_{L^1}^{\frac{1}{p}}.
\end{array}
\end{equation*}
Summarizing we get
\begin{equation*}\label{eq:I&II2}
\norm{f_{n+1}-f_n}_Y^p = (2/\Gamma_{\alpha})^p  \tau ^{\alpha p}   \left[ \sum\limits_{k=1}^n (a_{n+1-k}-a_{n+2-k})  \norm{(D^{\alpha}_\tau f)_k}_Y^p + a_1^p  \norm{(D^{\alpha}_\tau f)_{n+1}}_Y^p \right],
\end{equation*}
which yields
\begin{align*}
&\norm{f^{(\tau)}(t+\tau)-f^{(\tau)}(t)}_{L^p(0,T-\tau ;Y)}^p
= \sum\limits_{n=1}^{N-1}\tau \norm{f_{n+1}-f_n}_Y^p\\
\leq &(2/\Gamma_{\alpha})^p  \tau ^{\alpha p}\sum\limits_{n=1}^{N-1} \tau   \left( \sum\limits_{k=1}^n(a_{n+1-k}-a_{n+2-k})  \norm{(D^{\alpha}_\tau f)_k}_Y^p + a_1^p  \norm{(D^{\alpha}_\tau  f)_{n+1}}_Y^p\right)\\
\leq &(2/\Gamma_{\alpha})^p   \tau ^{\alpha r+1}   \left( \sum\limits_{n=1}^{N-1}\sum\limits_{k=1}^n (a_{n+1-k}-a_{n+2-k})  \norm{(D^{\alpha}_\tau f)_k}_Y^p + \sum\limits_{n=1}^{N-1} \norm{(D^{\alpha}_\tau  f)_{n+1}}_Y^p\right).
\end{align*}
Then we interchange summations and get
\begin{align*}
&\norm{f^{(\tau)}(t+\tau)-f^{(\tau)}(t)}_{L^p(0,T-\tau ;Y)}^p\\
=&(2/\Gamma_{\alpha})^p   \tau ^{\alpha r+1}   \left( \sum\limits_{k=1}^{N-1}\sum\limits_{n=k}^{N-1} (a_{n+1-k}-a_{n+2-k})  \norm{(D^{\alpha}_\tau f)_k}_Y^p + \sum\limits_{n=1}^{N-1} \norm{(D^{\alpha}_\tau  f)_{n+1}}_Y^p\right)\\
%=&2^p   \tau ^{\alpha r+1}   \left( \sum\limits_{k=1}^{N-1}\norm{(D^{\alpha}_\tau u)_k}_Y^p \sum\limits_{n=k}^{N-1} (a_{n+1-k}-a_{n+2-k}) + \sum\limits_{n=1}^{N-1} \norm{(D^{\alpha}_\tau  u)_{n+1}}_Y^p\right)\\
=&(2/\Gamma_{\alpha})^p   \tau ^{\alpha r+1}   \left( \sum\limits_{k=1}^{N-1}\norm{(D^{\alpha}_\tau f)_k}_Y^p   (a_1-a_{N+1-k})+ \sum\limits_{n=1}^{N-1} \norm{(D^{\alpha}_\tau  f)_{n+1}}_Y^p\right).
\end{align*}
Since $a_1-a_{N+1-k}=1-(N+1-k)^{\alpha -1}\in (0, 1]$, one obtains
\begin{align*}
&\norm{f^{(\tau)}(t+\tau)-f^{(\tau)}(t)}_{L^p(0,T-\tau ;Y)}^p\\
\leq &(2/\Gamma_{\alpha})^p   \tau ^{\alpha p}   \left( \sum\limits_{k=1}^{N-1} \tau   \norm{(D^{\alpha}_\tau f)_k}_Y^p+ \sum\limits_{n=1}^{N-1} \tau   \norm{(D^{\alpha}_\tau  f)_{n+1}}_Y^p\right)\\
= &(2/\Gamma_{\alpha})^p   \tau ^{\alpha p}   \left( \norm{D^{\alpha}_\tau f^{(\tau)}}_{L^p(0,T-\tau ;Y)}^p + \norm{D^{\alpha}_\tau f^{(\tau)}}_{L^p(\tau ,T;Y)}^p\right)\\
\leq &\frac{2^{p+1}    \tau ^{\alpha p}}{\Gamma_{\alpha}^p}   \norm{D^{\alpha}_\tau f^{(\tau)}}_{L^p(0,T;Y)}^p.
\end{align*} 
\end{proof}

We are now ready to prove \Cref{thm:mainII}.
\begin{proof}[Proof of \Cref{thm:mainII}]
We denote with $\tilde{f}^{(\tau)}$ the linear interpolant \eqref{eq:lin.intplt} and with $f^{(\tau)}$ the piecewise constant function $f^{(\tau)} = f_N\chf{\{T\}} + 
\sum_{n=0}^{N-1} f_n \chf{[t_n,t_{n+1})}$. From \eqref{eq:lin.intplt}, we have
$$
\norm{\tilde{f}^{(\tau)}(t)}_X \leq \norm{f^{(\tau)}(t)}_X + \norm{f^{(\tau)}(t+h)}_X.
$$
Hence
\begin{equation}\label{eq:u_bdd}
\norm{\tilde{f}^{(\tau)}}_{L^p(0,T;X)} \leq 2\norm{f^{(\tau)}}_{L^p(0,T;X)} \leq C,
\end{equation}
by the assumption of the theorem.
One can apply \Cref{lem.1} with $f=\tft$ and use \Cref{lem.2} to deduce
$$
\|\tft( \cdot + h)-\tft\|_{L^p(0,T-h; Y)}\leq C_\alpha h^\alpha \|D^\alpha_\tau\ft\|_{L^p(0,T; Y)} .
$$
As a consequence, if $\|D^\alpha_\tau\ft\|_{L^p(0,T; Y)}\leq C$, then	
\begin{equation}\label{eq:Su_bdd}
\lim_{h\to 0}\sup_{\tau>0}\|\tft( \cdot + h)-\tft\|_{L^p(0,T-h; Y)} = 0.
\end{equation}
By \Cref{lem:simon}, estimates \eqref{eq:u_bdd} and \eqref{eq:Su_bdd} imply that $\tilde{f}^{(\tau)}$ is relatively--compact in $L^p(0,T;B)$ and therefore there exists a subsequence of $\tilde{f}^{(\tau)}$ (still denoted with $\tilde{f}^{(\tau)}$) such that 
\begin{equation}\label{eq:theorem2star1}
\tilde{f}^{(\tau)}\rightarrow f^*\ \mbox{strongly in}\ L^p(0,T;B).
\end{equation}
%Namely, given a sequence $\tilde{u}^{(\tau)}$ we can subtract a subsequence $\tilde{u}^{(\tau_k)}$ which is convergent in $L^r(0,T;B)$.

Next we show that convergence of $\tilde{f}^{(\tau)}$ implies convergence of the corresponding piecewise constant function $f^{(\tau)}$.
Since
\begin{equation*}
\begin{split}
\norm{(f^{(\tau)}-\tilde{f}^{(\tau)})(t)}_Y &= 
\sum_{n=0}^{N-1}\norm{\frac{\tau - (t-t_n)}{\tau}(f_{n+1}-f_n)}_Y   \chf{[t_n,t_{n+1})}(t) \\
&\leq \sum_{n=0}^{N-1}\norm{f_{n+1}-f_n}_Y   \chf{[t_n,t_{n+1})}(t),
\end{split}
\end{equation*}
from \Cref{lem:tau jump} we have
\begin{equation*}
\begin{split}
\norm{f^{(\tau)}-\tilde{f}^{(\tau)}}_{L^p(0,T;Y)} &\leq
\norm{f^{(\tau)}(t+\tau)-f^{(\tau)}(t)}_{L^p(0,T-\tau ;Y)} \leq C(\alpha , p)   \tau^{\alpha}   \norm{D^{\alpha}_{\tau} f^{(\tau)}}_{L^p(0,T;Y)},
\end{split}
\end{equation*}
which implies
\begin{equation}\label{eq:theorem2star2}
\norm{f^{(\tau)}-\tilde{f}^{(\tau)}}_{L^p(0,T;Y)}\rightarrow 0,\ as\ \tau\rightarrow 0,
\end{equation}
by assumption of the theorem.
From interpolation we know that there exists $\theta \in (0,1)$ and $C_{\theta}>0$ such that
$$
\norm{f^{(\tau)}-\tilde{f}^{(\tau)}}_{L^p(0,T;B)}\leq C_{\theta}  \norm{f^{(\tau)}-\tilde{f}^{(\tau)}}_{L^p(0,T;X)}^{\theta}  \norm{f^{(\tau)}-\tilde{f}^{(\tau)}}_{L^p(0,T;Y)}^{1-\theta}.
$$
Summarizing we have
\begin{align*}
\norm{f^{(\tau)}-f^*}_{L^p(0,T;B)}
&\leq \norm{f^{(\tau)}-\tilde{f}^{(\tau)}}_{L^p(0,T;B)} + \norm{\tilde{f}^{(\tau)}-f^*}_{L^p(0,T;B)}\\
&\leq C_{\theta}  \norm{f^{(\tau)}-\tilde{f}^{(\tau)}}_{L^p(0,T;X)}^{\theta}  \norm{f^{(\tau)}-\tilde{f}^{(\tau)}}_{L^p(0,T;Y)}^{1-\theta}
+ \norm{\tilde{f}^{(\tau)}-f^*}_{L^p(0,T;B)}\\
&\leq C_{\theta}  \norm{3f^{(\tau)}}_{L^p(0,T;X)}^{\theta}  \norm{f^{(\tau)}-\tilde{f}^{(\tau)}}_{L^p(0,T;Y)}^{1-\theta} + \norm{\tilde{f}^{(\tau)}-f^*}_{L^p(0,T;B)}.
\end{align*}
Therefore $\norm{f^{(\tau)}-f^*}_{L^p(0,T;B)}\rightarrow 0$ as $\tau\rightarrow 0$ thanks to \eqref{eq:theorem2star1} and \eqref{eq:theorem2star2}, namely $f^{(\tau)}$ is relatively compact in $L^p(0,T;B)$. \Cref{thm:mainII} is proved.
\end{proof}

\section{Porous Medium Equation with Caputo time derivative}\label{sec:eqn}
The rest of the manuscript is devoted to the proof of \Cref{thm:mainthm}. We consider the system %\eqref{eq1}
\begin{equation}
\left\{
\begin{array}{l}
D^{\alpha}_t u=\textrm{div}(u \nabla p),\ x\in\mathbb{T}^3,\ t\in (0,T),\\
D^{\alpha}_tp=-(-\Delta)^{s}p+u^{2},
\end{array}
\right.\label{eq1111}
\end{equation}
with initial data $u_{in}$ and $p_{in}$: $\mathbb{T}^3\rightarrow (0,+\infty)$ such that $\int_{\mathbb{T}^3}u_{in}^2+|\nabla p_{in}|^2\;dx<+\infty$. Moreover $ 0 < s \leq 1$ and $0 < \alpha \leq 1$.
The energy functional
$$
H[u,p]:=\int_{\mathbb{T}^{3}}u^{2}+\frac{1}{2}|\nabla p|^{2}\;dx
$$
formally satisfies the following inequality:\footnote{In the case $s=1$ we replace in \eqref{eq_energy_ine} $(-\Delta)^{1/2} \nabla $ with $\Delta$.}
\begin{align}
H[u(t),p(t)]+\frac{1}{\Gamma_{\alpha}}\int_{0}^{t}\int_{\mathbb{T}^{3}}\frac{|(-\Delta)^{\frac{s}{2}}\nabla p(s)|^{2}}{(t-s)^{1-\alpha}}\;dxds\leq H(u_{in},p_{in}),\textrm{ for }t\in[0,T].\label{eq_energy_ine}
\end{align}
To see this, we first test the second equation in \eqref{eq1111} by $\Delta p$, %which yields 
%\[
%\int_{\mathbb{T}^{3}}u^{2}\Delta p\;dx=\int_{\mathbb{T}^{3}}D_{t}^{\alpha}p \Delta p\;dx+\int_{\mathbb{T}^{3}}(-\Delta)^{s}\nabla p \nabla p\;dx.
%\]
and from divergence theorem we get 
\begin{align*}
\int_{\mathbb{\mathbb{T}}^{3}}\nabla u^2\cdot \nabla p\;dx %=\int_{\mathbb{\mathbb{T}}^{3}}\nabla(D_{t}^{\alpha}p) \nabla p\;dx+\int_{\mathbb{\mathbb{T}}^{3}}(-\Delta)^{s}\nabla \nabla p\;dx\\
 =\int_{\mathbb{\mathbb{T}}^{3}}(D_{t}^{\alpha}\nabla p) \cdot\nabla p\;dx+\int_{\mathbb{\mathbb{T}}^{3}}\lvert(-\Delta)^{\frac{s}{2}}\nabla p\lvert^{2}\;dx.
\end{align*}
Then we test the first equation in \eqref{eq1111} by $2u$ and get 
\begin{align*}
\int_{\mathbb{\mathbb{T}}^{3}}2u  D_{t}^{\alpha}u\;dx & =\int_{\mathbb{\mathbb{T}}^{3}}\textrm{div}(u\nabla p) 2u\;dx
 %& =-2\int_{\mathbb{\mathbb{T}}^{3}}u\nabla p \nabla u\;dx\\
  =-\int_{\mathbb{\mathbb{T}}^{3}}\nabla u^{2} \cdot\nabla p\;dx.
\end{align*}
The combination of the two yields the following equation:
\begin{equation*}
\int_{\mathbb{T}^{3}}2u  D_{t}^{\alpha}u\;dx+\int_{\mathbb{T}^{3}}(D_{t}^{\alpha}\nabla p) \cdot\nabla p\;dx+\int_{\mathbb{T}^{3}}|(-\Delta)^{\frac{s}{2}}\nabla p|^{2}\;dx=0.\label{eq3}
\end{equation*}
Thanks to the fact that $2g  D^\alpha_t g \geq D^\alpha_t g^2$ (see \cite{li2018some}), we get 
%\[
%D_{t}^{\alpha}u^{2}\leq2u  D_{t}^{\alpha}u\textrm{ and }D_{t}^{\alpha}|\nabla p|^{2}\leq2\nabla p  D_{t}^{\alpha}\nabla p,
%\]
%therefore, we get the inequality 
\begin{align*}
D_{t}^{\alpha}\left(\int_{\mathbb{T}^{3}}u^{2}+\frac{1}{2}|\nabla p|^{2}\;dx\right)+\int_{\mathbb{T}^{3}}|(-\Delta)^{\frac{s}{2}}\nabla p|^{2}\;dx
%\leq\int_{\mathbb{T}^{3}}2u  D_{t}^{\alpha}u\;dx+\int_{\mathbb{T}^{3}}(D_{t}^{\alpha}\nabla p) \nabla p\;dx
\leq 0.
\end{align*}
%Now define the energy functional 
%\[
%H[u,p]:=\int_{\mathbb{T}^{3}}u^{2}+\frac{1}{2}|\nabla p|^{2}\;dx,
%\]
%We can rewrite the above inequality as
%\[
%D_{t}^{\alpha}H[u(t),p(t)]+\int_{\mathbb{T}^{3}}|(-\Delta)^{\frac{s}{2}}\nabla %p|^{2}\;dx\leq0.
%\]
The fundamental theorem of calculus \eqref{thm:funthm_calc_cont} yields \eqref{eq_energy_ine}.%fractional form (see \cite{li2018some}) yields \eqref{eq_energy_ine}.
%\begin{align*}
%H[u(t),p(t)] %& %=H[u(0),p(0)]+\frac{1}{\Gamma_{\alpha}}\int_{0}^{t}\frac{D_{s}^{\alpha}H[u(s),p(s)]}{(t-s)^{1-\alpha}}\ ds\\
 %& 
% \leq H[u(0),p(0)]-\frac{1}{\Gamma_{\alpha}}\int_{0}^{t}\frac{\int_{\mathbb{T}^{3}}|(-\Delta)^{\frac{s}{2}}\nabla p(s)|^{2}\;dx}{(t-s)^{1-\alpha}}\ ds.
%\end{align*}

%Let $\delta>0$ be small enough and define the functional spaces 
%\[X:=%L^{2}\cap 
%L^{6-\delta}(\mathbb{T}^{3}),\textcolor{red}{FIXME: need to delete the defination and replace all the X and Y}\]
%\[Y:= H^{1}(\mathbb{T}^{3}).\]

%discretize \eqref{eq1111} in time
Next we divide interval $[0,T]$ into $N$ subintervals with length $\tau$ and discretize \eqref{eq1111} in time. We also add extra viscosity terms as follows.
For given constants $\varrho,\tau,\varepsilon>0,$ functions $u_{j}\in H^1(\mathbb{T}^{3})$
and $p_{j}\in H^{2s}(\text{\ensuremath{\mathbb{T}}}^{3})$ such that
$u_{j},p_{j}\geq0,\ j=0,...,k-1,$ a.e. in $\text{\ensuremath{\mathbb{T}}}^{3}$,
consider the weak formulation: 
\begin{equation}
\int_{\mathbb{T}^3}(D^{\alpha}_{\tau}u)_k  \phi \;dx+\int_{\text{\ensuremath{\text{\ensuremath{\mathbb{T}}}^{3}}}}u_k\nabla p_k \cdot\nabla\phi \;dx
+\varrho\int_{\text{\ensuremath{\text{\ensuremath{\mathbb{T}}}^{3}}}}\nabla u_k \cdot\nabla\phi \;dx  =0,\ \forall\phi\in H^1(\mathbb{T}^{3}),\label{eq5}
\end{equation}
\begin{equation}
\int_{\mathbb{T}^3}(D^{\alpha}_{\tau}p)_k  \psi \;dx+\int_{\text{\ensuremath{\mathbb{T}}}^{3}}(-\Delta)^{s/2}p_k (-\Delta)^{s/2}\psi \;dx
+\varepsilon\int_{\text{\ensuremath{\mathbb{T}}}^{3}}\nabla p_k \cdot\nabla\psi \;dx-\int_{\text{\ensuremath{\mathbb{T}}}^{3}}u_k^{2}\psi \;dx =0,\ \forall\psi\in H^{1}(\text{\ensuremath{\mathbb{T}}}^{3}).\label{eq6}
\end{equation}

\subsection{Existence of solutions for \eqref{eq5}-\eqref{eq6}}
%%%%%%%%%%%%%%%%%%iffalse Thesis
\iffalse
We first linearize \eqref{eq5}--\eqref{eq6} and add extra constant $\sigma\in  [0,1]$ in front of the drift term $\int_{\mathbb{T}^3}z^{+}\nabla  p_k \cdot\nabla\phi \; dx$, %use Lax-Milgram Theorem to show the existence of solution for the linearized discrete problem: 
with  $z\in H^1(\mathbb{T}^3)$ given. Consider
\fi
%%%%%%%%%%%%%%%%%% fi Thesis
We first consider the linearized system
\begin{equation}
\int_{\mathbb{T}^3}(D^{\alpha}_{\tau}u)_k  \phi \;dx+\sigma\int_{\text{\ensuremath{\text{\ensuremath{\mathbb{T}}}^{3}}}}z^{+}\nabla p_k \cdot\nabla\phi \;dx\\
+\varrho\int_{\text{\ensuremath{\text{\ensuremath{\mathbb{T}}}^{3}}}}\nabla u_k \cdot\nabla\phi \;dx =0,\ \forall\phi\in H^1(\mathbb{T}^{3}),\label{eq7}
\end{equation}
\begin{equation}
\int_{\mathbb{T}^3}(D^{\alpha}_{\tau}p)_k  \psi \;dx+\int_{\text{\ensuremath{\mathbb{T}}}^{3}}(-\Delta)^{s/2}p_k (-\Delta)^{s/2}\psi \;dx
+\varepsilon\int_{\text{\ensuremath{\mathbb{T}}}^{3}}\nabla p_k \cdot\nabla\psi \;dx-\int_{\text{\ensuremath{\mathbb{T}}}^{3}}z^{2}\psi \;dx =0,\ \forall\psi\in H^{1}(\text{\ensuremath{\mathbb{T}}}^{3}),\label{eq8}
\end{equation}
with $z\in H^1(\mathbb{T}^3)$ and $\sigma\in[0,1]$.
Define the bilinear forms
\[
B_{1}[u_k,\phi]:=\Gamma_{\alpha}\tau^{-\alpha}\int_{\text{\ensuremath{\mathbb{T}}}^{3}}u_k\phi \;dx+\varrho\int_{\text{\ensuremath{\text{\ensuremath{\mathbb{T}}}^{3}}}}\nabla u_k \cdot\nabla\phi \  \;dx,
\]
\[
B_{2}[p_k,\psi]:=\Gamma_{\alpha}\tau^{-\alpha}\int_{\text{\ensuremath{\mathbb{T}}}^{3}}p_k\psi \;dx+\int_{\text{\ensuremath{\mathbb{T}}}^{3}}(-\Delta)^{s/2}p_k (-\Delta)^{s/2}\psi \;dx+\varepsilon\int_{\text{\ensuremath{\mathbb{T}}}^{3}}\nabla p_k \cdot\nabla\psi \;dx,
\]
and the linear functionals 
\[
\langle F_{1},\phi\rangle_{H^1(\mathbb{T}^{3})^{'},H^1(\mathbb{T}^{3})}:=\Gamma_{\alpha}\tau^{-\alpha}\int_{\text{\ensuremath{\mathbb{T}}}^{3}}\left[\sum\limits_{j=0}^{k-2}\lambda_{k-j}({u_{j+1}-u_{j}})-u_{k-1}\right]\phi \;dx+\sigma\int_{\text{\ensuremath{\text{\ensuremath{\mathbb{T}}}^{3}}}}z^{+}\nabla p_k \cdot\nabla\phi \;dx,
\]
\[
\langle F_{2},\psi\rangle_{(H^{1})^{'},H^{1}}:=\Gamma_{\alpha}\tau^{-\alpha}\int_{\text{\ensuremath{\mathbb{T}}}^{3}}\left[\sum\limits_{j=0}^{k-2}\lambda_{k-j}({p_{j+1}-p_{j}})-p_{k-1}\right]\psi \;dx-\int_{\text{\ensuremath{\mathbb{T}}}^{3}}z^{2}\psi \;dx.
\]
Note that $B_{2}$ is bounded
\begin{align*}
B_{2}[p_k,\psi]% & =\frac{\tau^{-\alpha}}{\Gamma_{1-\alpha}}\int_{\text{\ensuremath{\mathbb{T}}}^{3}}p_k\psi \;dx+\int_{\text{\ensuremath{\mathbb{T}}}^{3}}(-\Delta)^{s/2}p_k (-\Delta)^{s/2}\psi \;dx+\varepsilon\int_{\text{\ensuremath{\mathbb{T}}}^{3}}\nabla p_k \nabla\psi \;dx\\
%\textrm{H\"older\ensuremath{\rightarrow}\quad}
%& \leq C_{1}\left\Vert p\right\Vert _{L^{2}}\left\Vert \psi\right\Vert_{L^{2}}+C_{2}\left\Vertp\right\Vert_{\dot{H^{s}}}\left\Vert \psi\right\Vert_{\dot{H^{s}}}+C_{3}\left\Vert \nabla p\right\Vert_{L^{2}}\left\Vert\nabla\psi\right\Vert _{L^{2}}\\
%\textrm{\ensuremath{H^{1}}Def\ensuremath{\rightarrow}\quad} 
& \leq C_{4}\left\Vert p_k\right\Vert _{H^{1}}\left\Vert \psi\right\Vert _{H^{1}}+C_{2}\left\Vert p_k\right\Vert _{\dot{H^{s}}}\left\Vert \psi\right\Vert _{\dot{H^{s}}}
%\textrm{SobolevEmbedding\ensuremath{\rightarrow}\quad} 
 \leq C\left\Vert p_k\right\Vert _{H^{1}}\left\Vert \psi\right\Vert _{H^{1},}
\end{align*}
and coercive 
\begin{align*}
B_{2}[p_k,p_k]% & =\frac{\tau^{-\alpha}}{\Gamma_{1-\alpha}}\int_{\text{\ensuremath{\mathbb{T}}}^{3}}p_k^{2}\;dx+\int_{\text{\ensuremath{\mathbb{T}}}^{3}}\lvert(-\Delta)^{s/2}p_k\lvert^{2}\;dx+\varepsilon\int_{\text{\ensuremath{\mathbb{T}}}^{3}}\lvert\nabla p_k\lvert^{2}\;dx
%\textrm{NormDef\ensuremath{\rightarrow}\quad} 
%& \geq C_{1}\left\Vert p\right\Vert _{L^{2}}^{2}+C_{2}\left\Vert p\right\Vert _{\dot{H^{s}}}^{2}+C_{3}\left\Vert p\right\Vert _{\dot{H^{1}}}^{2}\\
%\textrm{SquareNonNeg\&\ensuremath{H^{1}}Def\ensuremath{\rightarrow}\quad} 
\geq C\left\Vert p_k\right\Vert _{H^{1}}^{2}.
\end{align*}
Moreover,
\begin{align*}
\langle F_{2},\psi\rangle_{(H^{1})^{'},H^{1}}% & =\frac{\tau^{-\alpha}}{\Gamma_{1-\alpha}}\int_{\text{\ensuremath{\mathbb{T}}}^{3}}\left[\sum\limits_{j=0}^{k-2}\frac{{p_{j+1}-p_{j}}}{(k-j)^{\alpha}}-p_{k-1}\right]\psi \;dx-\int_{\text{\ensuremath{\mathbb{T}}}^{3}}z^{2}\psi \;dx\\
%\textrm{H\"older\ensuremath{\rightarrow}\quad} 
& \leq C_{1}\sum\limits_{j=0}^{k-2}\lambda_{k-j}\left[\norm{p_{j+1}} _{L^{2}}+\norm{p_{j}} _{L^{2}}\right]\left\Vert \psi\right\Vert _{L^{2}}+C_{2}\left\Vert z^{2}\right\Vert _{L^{2}}\left\Vert \psi\right\Vert _{L^{2}}\\
&\leq C\left\Vert \psi\right\Vert _{H^{1}},
\end{align*}
since $p_{j}\in L^{2},j=0,...,k-1,$ and $z\in  L^{6-\delta},$ for $\delta > 0$. The constant $C$ depends on $\tau$, $\left\Vert z^{}\right\Vert _{L^{4}}$, and $\left\Vert p_j\right\Vert _{L^{2}}$ for any $0\leq j\leq k-1$. %we know that $z\in L^{4}$, and consequently, $z^{2}\in L^{2}$, thus,\begin{align*}\langle F_{2},\psi\rangle_{(H^{1})^{'},H^{1}} & \leq C\left\Vert \psi\right\Vert _{L^{2}}\\
%\textrm{\ensuremath{H^{1}}Def\ensuremath{\rightarrow}\quad} & \leq C\left\Vert \psi\right\Vert _{H^{1}.}
%\end{align*}
By Lax-Milgram theorem there exists
a unique solution $p_{k}\in H^{1}(\mathbb{T}^3)$ to \eqref{eq8}. In addition, the elliptic regularity theory implies that $p_k\in H^{2}(\mathbb{T}^3).$
%%%%%%%%%%%%%%%%%%%%%%%%%%iffalse
\iffalse
\textcolor{black}{as } 
\[
-\varrho\Delta p_k +(-\Delta)^{s}p_k + C D^{\alpha}_t p_k = F,
\]
for $F\in L^{2}$.
\fi
%%%%%%%%%%%%%%%%%%%%%%%%%%fi
%%%%%%%%%%%%%%%%%%%%%%iffalse thesis
\iffalse
Now we need to show that there exists an unique solution $u_{k}\in H^{1}(\mathbb{T}^3)$ to \eqref{eq7}.
\fi
%%%%%%%%%%%%%%%%%%%%%% fi thesis
%%%%%%%%%%%%%%%%%%%%%%%%% iffalse
\iffalse
: 
\[
\left\Vert \nabla p\right\Vert _{L^{6}}\leq C\left\Vert \Delta p\right\Vert _{L^{2}},
\]
where $\frac{1}{6}=\frac{1}{2}-\frac{1}{3}.$ 
\fi
%%%%%%%%%%%%%%%%%%%%%%%%% fi
Similarly, $B_1$ is bounded
\begin{align*}
B_{1}[u_k,\phi]% & =\frac{\tau^{-\alpha}}{\Gamma_{1-\alpha}}\int_{\text{\ensuremath{\mathbb{T}}}^{3}}u_k\phi \;dx+\varrho\int_{\text{\ensuremath{\text{\ensuremath{\mathbb{T}}}^{3}}}}\nabla u_k \nabla\phi \;dx
%\textrm{H\"older\ensuremath{\rightarrow}\quad} 
%& \leq C_{1}\left\Vert u\right\Vert _{L^{2}}\left\Vert \phi\right\Vert _{L^{2}}+C_{2}\left\Vert \nabla u\right\Vert _{L^{2}}\left\Vert \nabla\phi\right\Vert _{L^{2}}\\
%\text{\textrm{\ensuremath{H^{1}}Def\ensuremath{\rightarrow}\quad}} 
\leq C_{4}\left\Vert u_k\right\Vert _{H^{1}}\left\Vert \phi\right\Vert _{H^{1}},
\end{align*}
and coercive
\begin{align*}
B_{1}[u_k,u_k]% & =\frac{\tau^{-\alpha}}{\Gamma_{1-\alpha}}\int_{\text{\ensuremath{\mathbb{T}}}^{3}}u_k^{2}\;dx+\varrho\int_{\text{\ensuremath{\text{\ensuremath{\mathbb{T}}}^{3}}}}\lvert\nabla u_k\lvert^{2}\;dx
%\textrm{\ensuremath{L^{2}}Def\ensuremath{\rightarrow}\quad} 
%& \geq C_{1}\left\Vert u\right\Vert _{L^{2}}^{2}+C_{2}\left\Vert \nabla u\right\Vert _{L^{2}}^{2}\\
%\textrm{\ensuremath{H^{1}}Def\ensuremath{\rightarrow}\quad} 
\geq C_{2}\left\Vert u_k\right\Vert _{H^{1}}^{2},
\end{align*}
and
%Then we need to show 
%\[
%\langle F_{1},\phi\rangle_{H^1(\mathbb{T}^{3})^{*},H^1(\mathbb{T}^{3})}\leq C\left\Vert \phi\right\Vert _{H^1(\mathbb{T}^{3})}.
%\]
\begin{align*}
\langle F_{1},\phi\rangle_{H^1(\mathbb{T}^{3})^{'},H^1(\mathbb{T}^{3})}% & =\frac{\tau^{-\alpha}}{\Gamma_{1-\alpha}}\int_{\text{\ensuremath{\mathbb{T}}}^{3}}\left[\sum\limits_{j=0}^{k-2}\frac{{u_{j+1}-u_{j}}}{(k-j)^{\alpha}}+u_{k-1}\right]\phi \;dx+\sigma\int_{\text{\ensuremath{\text{\ensuremath{\mathbb{T}}}^{3}}}}z^{+}\nabla p_k \nabla\phi \;dx\\
%\textrm{H\"older\ensuremath{\rightarrow}\quad} 
& \leq C_{1} \sum\limits_{j=0}^{k-2}\lambda_{k-j}\left[\norm{u_{j+1}}_{L^{2}}+\norm{u_{j+1}}_{L^{2}}\right]\left\Vert \phi\right\Vert _{L^{2}}+C_{2}\left\Vert z^{+}\nabla p_k\right\Vert _{L^{2}}\left\Vert \nabla\phi\right\Vert _{L^{2}}\\
&\leq C\left\Vert \phi\right\Vert _{H^1(\mathbb{T}^{3})},
\end{align*}
since $u_{j}\in L^{2},\ j=0,...,k-1,$ $z^{+}\in L^{4},$  $\nabla p_k\in L^{q},\ q\geq4.$ The constant $C$ depends on $\tau$, $\left\Vert z^{}\right\Vert _{L^{4}}$, $\left\Vert u_j\right\Vert _{L^{2}}$, and $\left\Vert \nabla p_k\right\Vert _{L^{4}}$ for any $0\leq j\leq k-1$.
Once more, Lax Milgram theorem yields existence and uniqueness of the solution $u_k \in H^1(\mathbb{T}^{3})$ to %the linearized discrete problem 
\eqref{eq7}.
 %we can conclude that $z^{+}\nabla p\in L^{2},$
%consequently, 
%\begin{align*}
%\langle F_{1},\phi\rangle_{H^1(\mathbb{T}^{3})^{*},H^1(\mathbb{T}^{3})} & \leq C\left\Vert \phi\right\Vert _{H^{1}}\\
%\textrm{H^1(\mathbb{T}^{3})-Def\ensuremath{\rightarrow}\quad} 
%& \leq C\left\Vert \phi\right\Vert _{H^1(\mathbb{T}^{3})}.
%\end{align*}
%Therefore, at the end of the iteration process, there exists an unique solution
%$u=u_{k}\in H^1(\mathbb{T}^{3})$ to the linearized discrete problem \eqref{eq7}.

Next we use a fixed point argument to show the existence of solutions to \eqref{eq5} and \eqref{eq6}. Thanks to the existence and uniqueness of solution to \eqref{eq7} and \eqref{eq8}, we can define the map
\[
\mathcal{T}:(z,\sigma)\in L^{6-\delta}(\mathbb{T}^{3})\times[0,1]\rightarrow u\in L^{6-\delta}(\mathbb{T}^{3}).
\]

%1
\begin{lemma}\label{lem:nonNegative}
Given any $\sigma\in [0,1]$,
any fixed point of $\mathcal{T}(\cdot,\sigma)$ is non-negative.
\end{lemma} 
\begin{proof}
Choose $\psi=p_{k-}$ as test functions and get:
\begin{equation*}
 %& \frac{\tau^{-\alpha}}{\Gamma_{1-\alpha}}\int_{\text{\ensuremath{\mathbb{T}}}^{3}}\sum\limits_{j=0}^{k-1}\frac{{p_{j+1}-p_{j}}}{(k-j)^{\alpha}}p_{k-}\;dx+\int_{\text{\ensuremath{\mathbb{T}}}^{3}}(-\Delta)^{s/2}p (-\Delta)^{s/2}p_{k-}\;dx\\
 %& +\varrho\int_{\text{\ensuremath{\mathbb{T}}}^{3}}\nabla p \nabla p_{k-}\;dx-\int_{\text{\ensuremath{\mathbb{T}}}^{3}}u^{2}p_{k-}\;dx\\
 \int_{\mathbb{T}^3}(D^{\alpha}_{\tau}p)_k  p_{k-}\;dx+\int_{\text{\ensuremath{\mathbb{T}}}^{3}}\lvert(-\Delta)^{s/2}p_{k-}\lvert^{2}\;dx\\
 +\varepsilon\int_{\text{\ensuremath{\mathbb{T}}}^{3}}\lvert\nabla p_{k-}\lvert^{2}\;dx+\int_{\text{\ensuremath{\mathbb{T}}}^{3}}u_k^{2}(-p_{k-})\;dx = 0.
\end{equation*}
%1
%2
All the terms except for the first one are non-negative. Using \Cref{dis_left_caputo1} we rewrite the first term as 
% By \eqref{dis_left_caputo1} the first term yields
% \begin{align*}
%  & \frac{{p_{1}-p_{in}}}{k^{\alpha}}p_{k-}+\frac{{p_{2}-p_{1}}}{(k-1)^{\alpha}}p_{k-}+...+\frac{{p_{j+1}-p_{j}}}{(k-j)^{\alpha}}p_{k-}\\
% = & -\frac{1}{k^{\alpha}}p_{in}p_{k-}-\Bigg[\frac{1}{(k-1)^{\alpha}}-\frac{1}{k^{\alpha}}\Bigg]p_{1}p_{k-}-...-\Bigg[\frac{1}{2^{\alpha}}-\frac{1}{3^{\alpha}}\Bigg]p_{k-2}p_{k-}-\Bigg[1-\frac{1}{2^{\alpha}}\Bigg]p_{k-1}p_{k-}+p_{k-}^{2}.
% \end{align*}
% Thus, the first integral can be written as
\begin{align*}
  \int_{\mathbb{T}^3}(D^{\alpha}_{\tau}p)_k  p_{k-}\;dx
%= & \frac{\tau^{-\alpha}}{\Gamma_{1-\alpha}}\int_{\text{\ensuremath{\mathbb{T}}}^{3}}-\frac{1}{k^{\alpha}}p_{in}p_{k-}-\Bigg[\frac{1}{(k-1)^{\alpha}}-\frac{1}{k^{\alpha}}\Bigg]p_{1}p_{k-}-\\
 %& ...-\Bigg[\frac{1}{2^{\alpha}}-\frac{1}{3^{\alpha}}\Bigg]p_{k-2}p_{k-}-\Bigg[1-\frac{1}{2^{\alpha}}\Bigg]p_{k-1}p_{k-}\;dx+\frac{\tau^{-\alpha}}{\Gamma_{1-\alpha}}\int_{\text{\ensuremath{\mathbb{T}}}^{3}}p_{k-}^{2}\;dx\\
=&  \Gamma_{\alpha}\tau^{-\alpha}\int_{\text{\ensuremath{\mathbb{T}}}^{3}}
\bigg(\lambda_kp_{in}(-p_{k-})
+\sum\limits_{j=1}^{k-1}(\lambda_{k-j}-\lambda_{k-j+1})p_{j}(-p_{k-})
\bigg)dx\\
&+\Gamma_{\alpha}\tau^{-\alpha}\int_{\text{\ensuremath{\mathbb{T}}}^{3}}p_{k-}^{2}\;dx.
\end{align*}
We can see that every term is non-negative, since
$\ p_{j}\geq0,\ j=0,1,...,k-1,$ and 
\[
\lambda_{k-1}-\lambda_k\geq0,\ k=2,3,...\ 
\]
\iffalse and also, the last term 
\[
\frac{\tau^{-\alpha}}{\Gamma_{1-\alpha}}\int_{\text{\ensuremath{\mathbb{T}}}^{3}}p_{k-}^{2}\;dx\geq0.
\]
Then we can conclude that the first two terms of $(8)$ are non-negative,
and the rest of the terms 
%2
%3
\[
\int_{\text{\ensuremath{\mathbb{T}}}^{3}}\lvert(-\Delta)^{s/2}p_{k-}\lvert^{2}\;dx\geq0,\ \varrho\int_{\text{\ensuremath{\mathbb{T}}}^{3}}\lvert\nabla p_{k-}\lvert^{2}\;dx\geq0\ \int_{\text{\ensuremath{\mathbb{T}}}^{3}}u^{2}(-p_{k-})\;dx\geq0,
\]
since the sum of them is equal to $0,$ then
\fi
%%%%%%%%%%%%%%%%%%%%%%%%%%%%%%%%%fi
Therefore, we must have $p_{k-}=0$,
in other words, $p_k\geq0.$

Now we show the non-negativity of $u_k.$ With $\phi = u_{k-}$ as test function we get:
\begin{equation*}
  \int_{\mathbb{T}^3}(D^{\alpha}_{\tau}u)_k  u_{k-}\;dx+\varrho\int_{\text{\ensuremath{\text{\ensuremath{\mathbb{T}}}^{3}}}}|\nabla u_{k-}|^{2}\;dx\\
=  -\sigma\int_{\text{\ensuremath{\text{\ensuremath{\mathbb{T}}}^{3}}}}u_k^{+}\nabla p_k \cdot\nabla u_{k-}\;dx.
%=  \frac{\sigma}{2}\int_{\text{\ensuremath{\text{\ensuremath{\mathbb{T}}}^{3}}}}u^{+}u_{k-}\Delta p\;dx
\end{equation*}
We have: 
\begin{align*}
\Gamma_{\alpha}\tau^{-\alpha}\int_{\text{\ensuremath{\mathbb{T}}}^{3}}\lambda_ku_{in}(-u_{k-})+\sum\limits_{j=1}^{k-1}(\lambda_{k-j}-\lambda_{k-j+1})u_{j}(-u_{k-})\;dx\\+\Gamma_{\alpha}\tau^{-\alpha}\int_{\text{\ensuremath{\mathbb{T}}}^{3}}u_{k-}^{2}\;dx
+\varrho\int_{\text{\ensuremath{\text{\ensuremath{\mathbb{T}}}^{3}}}}|\nabla u_{k-}|^{2}\;dx=0.
\end{align*}
Since every term on the left is non negative, we must have
$u_{k-}=0,$ in other words, $u_k\geq0.$ 
\end{proof}
%3

We next prove that $\mathcal{T}(\cdot,1)$ has a fixed point.
First, $\mathcal{T}(\cdot ,0)$ is constant since eq.~\eqref{eq7} becomes independent of $z$ when $\sigma=0$.
Next, we show that $\mathcal{T}$ is continuous and compact. %\begin{align*}
%\frac{\tau^{-\alpha}}{\Gamma_{1-\alpha}}\int_{\text{\ensuremath{\mathbb{T}}}^{3}}\left[\sum\limits_{j=0}^{k-2}\frac{{u_{j+1}-u_{j}}}{(k-j)^{\alpha}}+(u-u_{k-1})\right]\phi \;dx+\sigma\int_{\text{\ensuremath{\text{\ensuremath{\mathbb{T}}}^{3}}}}z^{+}\nabla p \nabla\phi \;dx\\
%+\varrho\int_{\text{\ensuremath{\text{\ensuremath{\mathbb{T}}}^{3}}}}\nabla u \nabla\phi \;dx& =0,\ \forall\phi\in H^1(\mathbb{T}^{3}),
%\end{align*}
%\begin{align*}
%\frac{\tau^{-\alpha}}{\Gamma_{1-\alpha}}\int_{\mathbb{T}^{3}}\left[\sum\limits_{j=0}^{k-2}\frac{{p_{j+1}-p_{j}}}{(k-j)^{\alpha}}+(p-p_{k-1})\right]\psi \;dx+\int_{\text{\ensuremath{\mathbb{T}}}^{3}}(-\Delta)^{s/2}p (-\Delta)^{s/2}\psi \;dx\\
%+\varrho\int_{\text{\ensuremath{\mathbb{T}}}^{3}}\nabla p \nabla\psi \;dx-\int_{\text{\ensuremath{\mathbb{T}}}^{3}}z^{2}\psi \;dx & =0,\ \forall\psi\in H^{1}(\text{\ensuremath{\mathbb{T}}}^{3}).
%\end{align*}
Taking $\phi=u_k$ in \eqref{eq7}, we have 
\begin{equation*}
\int_{\mathbb{T}^3}(D^{\alpha}_{\tau}u)_k  u_k\;dx+\sigma\int_{\text{\ensuremath{\text{\ensuremath{\mathbb{T}}}^{3}}}}z^{+}\nabla p_k \cdot\nabla u_k\;dx
+\varrho\int_{\text{\ensuremath{\text{\ensuremath{\mathbb{T}}}^{3}}}}\lvert\nabla u_k\lvert^{2}\;dx =0,%\label{eq11}
\end{equation*}
%\begin{align*}
%\frac{\tau^{-\alpha}}{\Gamma_{1-\alpha}}\int_{\mathbb{T}^{3}}\left[\sum\limits_{j=0}^{k-1}\frac{{p_{j+1}-p_{j}}}{(k-j)^{\alpha}}\right]\Delta p_k\;dx+\int_{\text{\ensuremath{\mathbb{T}}}^{3}}(-\Delta)^{s/2}p_k (-\Delta)^{s/2}\Delta p_k\;dx\nonumber \\
%+\varepsilon\int_{\text{\ensuremath{\mathbb{T}}}^{3}}\nabla p_k \nabla\Delta p_k\;dx-\int_{\text{\ensuremath{\mathbb{T}}}^{3}}z^{2}\Delta p_k\;dx & =0,%\label{eq12}
%\end{align*}
which implies
\begin{align*}
 & \Gamma_{\alpha}\tau^{-\alpha}\int_{\text{\ensuremath{\mathbb{T}}}^{3}}u_k^{2}\;dx+\varrho\int_{\text{\ensuremath{\text{\ensuremath{\mathbb{T}}}^{3}}}}|\nabla u_k|^{2}\;dx\\
= & -\sigma\int_{\text{\ensuremath{\text{\ensuremath{\mathbb{T}}}^{3}}}}z^{+}\nabla p_k \cdot\nabla u_k\;dx-\Gamma_{\alpha}\tau^{-\alpha}\int_{\text{\ensuremath{\mathbb{T}}}^{3}}\left[\sum\limits_{j=0}^{k-2}\lambda_{k-j}(u_{j+1}-u_{j})-u_{k-1}\right]u_k\;dx\\
%\leq & \sigma\left\vert\int_{\text{\ensuremath{\text{\ensuremath{\mathbb{T}}}^{3}}}}z^{+}\nabla p_k \nabla u_k\;dx\right\vert+\frac{\tau^{-\alpha}}{\Gamma_{1-\alpha}}\left\vert\int_{\text{\ensuremath{\mathbb{T}}}^{3}}\sum\limits_{j=0}^{k-2}\frac{{u_{j+1}-u_{j}}}{(k-j)^{\alpha}}u_k\;dx\right\vert\\
%&+\frac{\tau^{-\alpha}}{\Gamma_{1-\alpha}}\left\vert\int_{\text{\ensuremath{\mathbb{T}}}^{3}}u_{k-1}u_k\;dx\right\vert\\
%Young's\ Inequality\rightarrow\quad
\leq & \; \frac{\varrho}{2}\int_{\text{\ensuremath{\text{\ensuremath{\mathbb{T}}}^{3}}}}\lvert\nabla u_k\lvert^{2}\;dx+\frac{2\sigma^2}{\varrho}\int_{\text{\ensuremath{\text{\ensuremath{\mathbb{T}}}^{3}}}}(z^{+})^{2}\lvert\nabla p_k\lvert^{2}\;dx+\frac{\Gamma_{\alpha}\tau^{-\alpha}}{2}\int_{\text{\ensuremath{\mathbb{T}}}^{3}}u_k^{2}\;dx\\
%& +\frac{8\tau^{-\alpha}}{\Gamma_{1-\alpha}}\sum\limits_{j=0}^{k-2}\frac{1}{(k-j)^{2\alpha}}\int_{\text{\ensuremath{\mathbb{T}}}^{3}}u_{j+1}^2+u_{j}^2\;dx +\frac{4\tau^{-\alpha}}{\Gamma_{1-\alpha}}\int_{\text{\ensuremath{\mathbb{T}}}^{3}}u_{k-1}^{2}\;dx.
& +16\Gamma_{\alpha}\tau^{-\alpha}\sum\limits_{j=0}^{k-2}\lambda_{k-j}^2\int_{\text{\ensuremath{\mathbb{T}}}^{3}}u_{j}^2\;dx.
\end{align*}
Reorganizing the terms, we have
\begin{align}\label{eq:star}
  \frac{\Gamma_{\alpha}\tau^{-\alpha}}{2}\int_{\text{\ensuremath{\mathbb{T}}}^{3}}u_k^{2}\;dx+\frac{\varrho}{2}\int_{\text{\ensuremath{\text{\ensuremath{\mathbb{T}}}^{3}}}}|\nabla u_k|^{2}\;dx
%\leq & \frac{1}{\mu}\int_{\text{\ensuremath{\text{\ensuremath{\mathbb{T}}}^{3}}}}(z^{+})^{2}\lvert\nabla p\lvert^{2}\;dx+\frac{1}{\nu}\int_{\text{\ensuremath{\mathbb{T}}}^{3}}\sum\limits_{j=0}^{k-2}\Bigg[\frac{{u_{j+1}-u_{j}}}{(k-j)^{\alpha}}\Bigg]^{2}\;dx+\frac{1}{\nu}\int_{\text{\ensuremath{\mathbb{T}}}^{3}}u_{k-1}^{2}\;dx\\
%Young's\ Inequality\rightarrow\quad
  \leq   & {{\mu}}\int_{\text{\ensuremath{\text{\ensuremath{\mathbb{T}}}^{3}}}}(z^{+})^{4}\;dx+\frac{1}{ {{\mu}}}\int_{\text{\ensuremath{\text{\ensuremath{\mathbb{T}}}^{3}}}}\lvert\nabla p_k\lvert^{4}\;dx\\
% & +\frac{8\tau^{-\alpha}}{\Gamma_{1-\alpha}}\sum\limits_{j=0}^{k-2}\frac{1}{(k-j)^{2\alpha}}\int_{\text{\ensuremath{\mathbb{T}}}^{3}}u_{j+1}^2+u_{j}^2\;dx+\frac{4\tau^{-\alpha}}{\Gamma_{1-\alpha}}\int_{\text{\ensuremath{\mathbb{T}}}^{3}}u_{k-1}^{2}\;dx.\label{eq:star}
%= &  {{\mu}}\left\Vert z^{+}\right\Vert _{L^{4}}+\frac{1}{ {{\mu}}}\left\Vert \nabla p\right\Vert _{L^{4}}++\frac{1}{\nu}\int_{\text{\ensuremath{\mathbb{T}}}^{3}}\sum\limits_{j=0}^{k-2}\Bigg[\frac{{u_{j+1}-u_{j}}}{(k-j)^{\alpha}}\Bigg]^{2}\;dx+\frac{1}{\nu}\int_{\text{\ensuremath{\mathbb{T}}}^{3}}u_{k-1}^{2}\;dx.
  &+16\Gamma_{\alpha}\tau^{-\alpha}\sum\limits_{j=0}^{k-2}\lambda_{k-j}^2\int_{\text{\ensuremath{\mathbb{T}}}^{3}}u_{j}^2\;dx.\notag
\end{align}
Now we need to estimate $\nabla p_k$. We take $\psi = \Delta p_k$ in \eqref{eq8}, and get
\begin{equation*}
\int_{\mathbb{T}^3}(D^{\alpha}_{\tau}p)_k   \Delta p_k\;dx+\int_{\text{\ensuremath{\mathbb{T}}}^{3}}(-\Delta)^{s/2} p_k (-\Delta)^{s/2} \Delta p_k\;dx
+\varepsilon\int_{\text{\ensuremath{\mathbb{T}}}^{3}}\nabla p_k \cdot\nabla \Delta p_k\;dx-\int_{\text{\ensuremath{\mathbb{T}}}^{3}}(z^+)^{2}\Delta p_k\;dx =0,%\label{eq12}
\end{equation*}
which implies
\begin{align*}
 & \Gamma_{\alpha}\tau^{-\alpha}\int_{\text{\ensuremath{\mathbb{T}}}^{3}}
|\nabla p_k|^{2}\;dx
+\varepsilon\int_{\text{\ensuremath{\text{\ensuremath{\mathbb{T}}}^{3}}}}|\Delta p_k|^{2}\;dx\\
= & -\int_{\text{\ensuremath{\mathbb{T}}}^{3}}z^{2}\Delta p_k\;dx-\Gamma_{\alpha}\tau^{-\alpha}\int_{\text{\ensuremath{\mathbb{T}}}^{3}}\left[\sum\limits_{j=0}^{k-2}\lambda_{k-j}(\nabla p_{j+1}-\nabla p_{j})-\nabla p_{k-1}\right] \cdot\nabla p_k\;dx\\
\leq & \;\frac{2}{\varepsilon}\int_{\text{\ensuremath{\mathbb{T}}}^{3}}(z^+)^{4}\;dx +\frac{\varepsilon }{2} \int_{\text{\ensuremath{\mathbb{T}}}^{3}}\left|\Delta p_k\right|^2\;dx +\frac{\Gamma_{\alpha}\tau^{-\alpha}}{2}\int_{\text{\ensuremath{\mathbb{T}}}^{3}}\left|\nabla p_k\right|^{2}\;dx\\
%& +\frac{8\tau^{-\alpha}}{\Gamma_{1-\alpha}}\sum\limits_{j=0}^{k-2}\frac{1}{(k-j)^{2\alpha}}\int_{\text{\ensuremath{\mathbb{T}}}^{3}}\left|\nabla p_{j+1}\right|^2+\left|\nabla p_{j}\right|^2\;dx +\frac{4\tau^{-\alpha}}{\Gamma_{1-\alpha}}\int_{\text{\ensuremath{\mathbb{T}}}^{3}}\left|\nabla p_{k-1}\right|^{2}\;dx.
& +16\Gamma_{\alpha}\tau^{-\alpha}\sum\limits_{j=0}^{k-2}\lambda_{k-j}^2\int_{\text{\ensuremath{\mathbb{T}}}^{3}}\left|\nabla p_{j}\right|^2\;dx .
\end{align*}
Combining the similar terms, we get
\begin{align*}
  \frac{\Gamma_{\alpha}\tau^{-\alpha}}{2}\int_{\text{\ensuremath{\mathbb{T}}}^{3}}\left|\nabla p_k\right|^{2}\;dx+\frac{\varepsilon}{2}\int_{\text{\ensuremath{\text{\ensuremath{\mathbb{T}}}^{3}}}}|\Delta p_k|^{2}\;dx
%\leq &  \frac{2}{\epsilon}\int_{\text{\ensuremath{\mathbb{T}}}^{3}}z^{4}\;dx+\frac{8\tau^{-\alpha}}{\Gamma_{1-\alpha}}\sum\limits_{j=0}^{k-2}\frac{1}{(k-j)^{2\alpha}}\int_{\text{\ensuremath{\mathbb{T}}}^{3}}\left|\nabla p_{j+1}\right|^2+\left|\nabla p_{j}\right|^2\;dx\\ &+\frac{4\tau^{-\alpha}}{\Gamma_{1-\alpha}}\int_{\text{\ensuremath{\mathbb{T}}}^{3}}\left|\nabla p_{k-1}\right|^{2}\;dx.
\leq &  \frac{2}{\varepsilon}\int_{\text{\ensuremath{\mathbb{T}}}^{3}}(z^+)^{4}\;dx+16\Gamma_{\alpha}\tau^{-\alpha}\sum\limits_{j=0}^{k-2}\lambda_{k-j}^2\int_{\text{\ensuremath{\mathbb{T}}}^{3}}\left|\nabla p_{j}\right|^2\;dx.
\end{align*}
%Since $z^{+}\in L^{4},$ $\nabla p_k\in L^{6}%L4?
%,$ $u_{j}\in L^{2},\ j=0,...,k-1,$
%we can get 
%\[
%\frac{\tau^{-\alpha}}{2\Gamma_{1-\alpha}}\int_{\text{\ensuremath{\mathbb{T}}}^{3}}u_k^{2}\;dx+\frac{\varrho}{2}\int_{\text{\ensuremath{\text{\ensuremath{\mathbb{T}}}^{3}}}}|\nabla u_k|^{2}\;dx\leq C\left\Vert z^{+}\right\Vert _{L^{4}}.
%\]
%\textcolor{black}{The definition of $H^1(\mathbb{T}^{3})$ yields  }
%\begin{align*}
%\sup_{\left\Vert z^{+}\right\Vert _{L^{6-\delta}(\mathbb{T}^{3})}=1}\left\Vert u_k\right\Vert _{H^1(\mathbb{T}^{3})} & \leq C.
%\end{align*}
%Therefore, we have 
%\[
%\sup_{\left\Vert z^{+}\right\Vert _{L^{6-\delta}(\mathbb{T}^{3})}=1}\left\Vert u_k\right\Vert _{L^{6-\delta}(\mathbb{T}^{3})}\leq
%\sup_{\left\Vert z^{+}\right\Vert _{L^{6-\delta}(\mathbb{T}^{3})}=1}\left\Vert u_k\right\Vert _{H^1(\mathbb{T}^{3})}\leq C.
%\]
Given $z^{+}\in L^{4}$, we have $\nabla p_k\in H^1$, which suggests that $\nabla p_k$ is uniformly bounded in $L^4$, and therefore, $\norm{u_k}_{H^1}\leq C$ from \eqref{eq:star}. 
This shows that $\mathcal{T}$ is bounded as an operator $L^{6-\delta}(\mathbb{T}^3)\to H^{1}(\mathbb{T}^3)$. 
The compactness of $\mathcal{T}$ directly
follows the compact embedding $H^1(\mathbb{T}^{3})\hookrightarrow L^{6-\delta}(\mathbb{T}^{3})$, while the (sequential) continuity of $\mathcal{T}$ is proved via a standard argument.

Next, we show that fixed points of $\mathcal{T}(\cdot ,\sigma)$ are uniformly bounded in $\sigma$ for $\sigma\in [0,1]$. We consider $\phi=u_k$ and $\psi=\Delta p_k$ as test functions respectively in \eqref{eq5} and \eqref{eq6}, and take summation of the equations:
%%%%%%%%%%%%%%%%%%%%%%%%%%%iffalse
\iffalse
\begin{equation*}
 \frac{\tau^{-\alpha}}{\Gamma_{1-\alpha}}\int_{\text{\ensuremath{\mathbb{T}}}^{3}}\left[\sum\limits_{j=0}^{k-1}\frac{{u_{j+1}-u_{j}}}{(k-j)^{\alpha}}\right]u_k\;dx+\varrho\int_{\text{\ensuremath{\text{\ensuremath{\mathbb{T}}}^{3}}}}|\nabla u_k|^{2}\;dx = \frac{\sigma}{2} \frac{\tau^{-\alpha}}{\Gamma_{1-\alpha}}\int_{\text{\ensuremath{\text{\ensuremath{\mathbb{T}}}^{3}}}}u^{2}\Delta p_k\;dx,
\end{equation*}
and
\begin{align*}
\frac{\tau^{-\alpha}}{\Gamma_{1-\alpha}}\int_{\text{\ensuremath{\mathbb{T}}}^{3}}u_k^{2}\Delta p_k\;dx%= & \int_{\mathbb{T}^{3}}\left[\sum\limits_{j=0}^{k-2}\frac{{p_{j+1}-p_{j}}}{(k-j)^{\alpha}}+(p-p_{k-1})\right]\Delta p\;dx+\int_{\text{\ensuremath{\mathbb{T}}}^{3}}(-\Delta)^{s/2}p (-\Delta)^{s/2}\Delta p\;dx\\
 %& +\varrho\int_{\text{\ensuremath{\mathbb{T}}}^{3}}\nabla p \nabla\Delta p\;dx\\
= & -\frac{\tau^{-\alpha}}{\Gamma_{1-\alpha}}\int_{\mathbb{T}^{3}}\left[\sum\limits_{j=0}^{k-1}\frac{{\nabla p_{j+1}-\nabla p_{j}}}{(k-j)^{\alpha}}-(\nabla p_k-\nabla p_{k-1})\right]\cdot\nabla p_k\;dx\\
&-\int_{\text{\ensuremath{\mathbb{T}}}^{3}}|(-\Delta)^{s/2}\nabla p_k|^{2}\;dx-\varepsilon\int_{\text{\ensuremath{\mathbb{T}}}^{3}}(\Delta p_k)^{2}\;dx.
\end{align*}
Combining the above two formulations yields 
\fi
%%%%%%%%%%%%%%%%%%%%%%%%%%%%%%%%%fi
\begin{align*}
 & \Gamma_{\alpha}\tau^{-\alpha}\int_{\text{\ensuremath{\mathbb{T}}}^{3}}\left[\sum\limits_{j=0}^{k-1}\lambda_{k-j}(u_{j+1}-u_{j})\right]u_k\;dx + \frac{\sigma}{2}\Gamma_{\alpha}\tau^{-\alpha}\int_{\mathbb{T}^{3}}\left[\sum\limits_{j=0}^{k-1}\lambda_{k-j}(\nabla p_{j+1}-\nabla p_{j})\right]\cdot\nabla p_k\;dx\\
 &+\varrho\int_{\text{\ensuremath{\text{\ensuremath{\mathbb{T}}}^{3}}}}|\nabla u_k|^{2}\;dx
+\frac{\sigma}{2}\int_{\text{\ensuremath{\mathbb{T}}}^{3}}|(-\Delta)^{s/2}\nabla p_k|^{2}\;dx
+ \frac{\sigma}{2}\varepsilon\int_{\text{\ensuremath{\mathbb{T}}}^{3}}(\Delta p_k)^{2}\;dx=  0.
\end{align*}
From \Cref{lem:fact2} we have
\begin{align}\label{eq:energy.sigma}
 & \frac{1}{2}\Gamma_{\alpha}\tau^{-\alpha}\int_{\text{\ensuremath{\mathbb{T}}}^{3}}\sum\limits_{j=0}^{k-1}\lambda_{k-j}(u_{j+1}^{2}-u_{j}^{2})\;dx+ \frac{\sigma}{4}\Gamma_{\alpha}\tau^{-\alpha}\int_{\mathbb{T}^{3}}\sum\limits_{j=0}^{k-1}\lambda_{k-j}(|\nabla p_k|_{j+1}^{2}-|\nabla p_k|_{j}^{2})\;dx\\
 &+\varrho\int_{\text{\ensuremath{\text{\ensuremath{\mathbb{T}}}^{3}}}}|\nabla u_k|^{2}\;dx
+\frac{\sigma}{2}\int_{\text{\ensuremath{\mathbb{T}}}^{3}}|(-\Delta)^{s/2}\nabla p_k|^{2}\;dx+\frac{\sigma}{2}\varepsilon\int_{\text{\ensuremath{\mathbb{T}}}^{3}}(\Delta p_k)^{2}\;dx\leq 0.\notag
\end{align}
\Cref{lem:fact1} implies that $\int_{\mathbb{T}^3}u_k^2+\frac{\sigma}{2}|\nabla p_k|^2\; dx\leq\int_{\mathbb{T}^3}u_{in}^2+\frac{\sigma}{2}|\nabla p_k|^2\; dx$ for $\forall k=0, ..,N$.
Moreover, from \eqref{eq:energy.sigma} we have
$$
D^{\alpha}_{\tau}\left(\int_{\mathbb{T}^{3}}u_k^{2}+\frac{\sigma}{2}|\nabla p_k|^{2}\;dx\right) \leq -2\varrho\int_{\mathbb{T}^3}|\nabla u_k|^2\; dx-\sigma\int_{\text{\ensuremath{\mathbb{T}}}^{3}}|(-\Delta)^{s/2}\nabla p_k|^{2}\;dx-\sigma\varepsilon\int_{\text{\ensuremath{\mathbb{T}}}^{3}}(\Delta p_k)^{2}\;dx,
$$
which implies, using \eqref{eq:discreteFToC},
\begin{align*}
&\int_{\mathbb{T}^3}u_k^2+\frac{\sigma}{2}|\nabla p_k|^2\;dx + \frac{\tau ^{\alpha}}{\Gamma_{\alpha}}\sum\limits_{j=1}^k(k-j+1)^{\alpha-1} \left( 2\varrho\int_{\mathbb{T}^3}|\nabla u_{j}|^2\; dx +\sigma\int_{\text{\ensuremath{\mathbb{T}}}^{3}}|(-\Delta)^{s/2}\nabla p_{j}|^{2}\;dx+\sigma\varepsilon\int_{\text{\ensuremath{\mathbb{T}}}^{3}}(\Delta p_{j})^{2}\;dx\right)\\
&\leq \int_{\mathbb{T}^3}u_{in}^2+\frac{1}{2}|\nabla p_{in}|^{2}\;dx.
\end{align*}
The last estimate shows that $u_k$ is bounded in $H^1(\mathbb{T}^3)$ uniformly with respect to $\sigma$.

Hence Leray-Schauder fixed point theorem yields the existence of a fixed point $u_k\in H^1(\mathbb{T}^{3})$ for $\mathcal{T}(\cdot,1)$, that is,
a solution $(u_k,p_k)\in H^1(\mathbb{T}^{3})\times H^{2}(\mathbb{T}^{3})$ to 
\begin{equation}
\int_{\mathbb{T}^3}(D^{\alpha}_{\tau}u)_k  \phi \;dx+\int_{\text{\ensuremath{\text{\ensuremath{\mathbb{T}}}^{3}}}}u_k\nabla p_k \cdot\nabla\phi \;dx
+\varrho\int_{\text{\ensuremath{\text{\ensuremath{\mathbb{T}}}^{3}}}}\nabla u_k \cdot\nabla\phi \;dx  =0,\label{eq15-1}
\end{equation}
and 
\begin{equation}
\int_{\mathbb{T}^3}(D^{\alpha}_{\tau}p)_k  \psi \;dx+\int_{\text{\ensuremath{\mathbb{T}}}^{3}}(-\Delta)^{s/2}p_k (-\Delta)^{s/2}\psi \;dx
+\varepsilon\int_{\text{\ensuremath{\mathbb{T}}}^{3}}\nabla p_k \cdot\nabla\psi \;dx-\int_{\text{\ensuremath{\mathbb{T}}}^{3}}u_{k}^{2}\psi \;dx  =0,\label{eq16-1}
\end{equation}
for all $\phi,\psi\in H^1(\mathbb{T}^{3})$, such that $u_k,p_k\geq0$ a.e. in $\mathbb{T}^{3}$ and 
\begin{align*}
H_k+ \frac{\tau^{\alpha}}{\Gamma_{\alpha}}\sum\limits _{i=1}^k(k-i+1)^{\alpha-1}\left(\varrho\int_{\text{\ensuremath{\text{\ensuremath{\mathbb{T}}}^{3}}}}|\nabla u_{i}|^{2}\;dx+\frac{1}{2}\int_{\text{\ensuremath{\mathbb{T}}}^{3}}|(-\Delta)^{s/2}\nabla p_{i}|^{2}\;dx+\frac{\varepsilon}{2}\int_{\text{\ensuremath{\mathbb{T}}}^{3}}(\Delta p_{i})^{2}\;dx\right)
\leq H(u_{in},p_{in}),%\label{eq17-1}
\end{align*}
using \Cref{prop:discreteFToC}.
%with $f_k=H_k:=\int_{\mathbb{T}^3}u_k^2+\frac{\sigma}{2}|\nabla p_k|^{2}\;dx$ and $H(u_{in},p_{in}) = \int_{\mathbb{T}^3}u_{in}^2+\frac{\sigma}{2}|\nabla p_{in}|^{2}\;dx$.
{Taking test functions $\phi = 1$ and $\psi=1$ in \eqref{eq15-1} and \eqref{eq16-1} we get
$\int_{\mathbb{T}^3}(D^{\alpha}_{\tau}u)_k\;dx =0\iffalse\label{eq15-2}\fi$, and 
$\int_{\mathbb{T}^3}(D^{\alpha}_{\tau}p)_k \;dx-\int_{\text{\ensuremath{\mathbb{T}}}^{3}}u_{k}^{2}\;dx  =0\iffalse\label{eq16-2}\fi$.
Since $\int_{\mathbb{T}^3}(D^{\alpha}_{\tau}u)_k\;dx = \left( D^{\alpha}_{\tau}\int_{\mathbb{T}^3}u\;dx\right)_k$, 
\Cref{prop:discreteFToC} implies 
$$\int u_k\; dx=\int u_{in}\; dx \mbox{ for any }k=0, \dots,N.$$
Similarly,
$\left( D^{\alpha}_{\tau}\int_{\mathbb{T}^3}p\;dx\right)_k = \int_{{\mathbb{T}}^3}u_{k}^{2}\;dx \leq H(u_{in},p_{in})$ and
\begin{equation*}
\int_{\mathbb{T}^{3}}p_k\;dx\leq\int_{\mathbb{T}^3}p_{in}\;dx+\frac{\tau^{\alpha}}{\Gamma_{\alpha}}  H(u_{in},p_{in}) \sum\limits_{i=1}^k (k-i+1)^{\alpha -1}
\leq \int_{\mathbb{T}^{3}}p_{in}\;dx+\frac{1}{\alpha \Gamma_{\alpha}}H(u_{in},p_{in})  T^{\alpha},
\end{equation*}
for any $k=0, \cdots, N$.
}

\section{Limit $\tau\rightarrow0$}\label{sec:tauLimit}

For all $T>0$, let $N=T/\tau$. Define the piecewise constant
interpolant of $\{u_{k}\}$ and $\{p_{k}\}$, $k=0, ..,N$, respectively as
\begin{align*}
u^{(\tau)}(t)=u_{in}\chf{\{0\}}(t) + \sum\limits_{k=1}^{N}u_{k}\chf{ (t_{k-1},t_k]}(t),\\
p^{(\tau)}(t)=p_{in}\chf{\{0\}}(t) + \sum\limits_{k=1}^{N}p_{k}\chf{ (t_{k-1},t_k]}(t),\\
H^{(\tau)}(t)=\int_{\mathbb{T}^3}\bigg(
(u^{(\tau)}(t,x))^2+\frac{1}{2}|\nabla p^{(\tau)}(t,x)|^2\; 
\bigg)dx.
\end{align*}
By \eqref{eq:disCaputo2} we have 
$$
D_{\tau}^{\alpha}u^{(\tau)}(t)=\Gamma_\alpha\tau^{-\alpha} \sum_{k=1}^n \sum_{j=0}^{k-1}
\lambda_{k-j}(u_{j+1}-u_j)  \chf{(t_{k-1},t_k]}(t).
$$
So, \eqref{eq15-1} and \eqref{eq16-1} can be rewritten as
\begin{align}\label{eq26}
\int_{0}^{T-\tau}\int_{\text{\ensuremath{\mathbb{T}}}^{3}}\left(D_{\tau}^{\alpha}u^{(\tau)}\right)\phi \;dxdt+\int_{0}^{T-\tau}\int_{\text{\ensuremath{\text{\ensuremath{\mathbb{T}}}^{3}}}}u^{(\tau)}\nabla p^{(\tau)} \cdot\nabla\phi \;dxdt \\
+\varrho\int_{0}^{T-\tau}\int_{\text{\ensuremath{\text{\ensuremath{\mathbb{T}}}^{3}}}}\nabla u^{(\tau)} \cdot\nabla\phi \;dxdt & =0,\ \forall\phi\in L^{2}(0,T-\tau;\ H^{1}(\mathbb{T}^{3})),\notag
\end{align}
\begin{align}\label{eq27}
\int_{0}^{T-\tau}\int_{\mathbb{T}^{3}}\left(D_{\tau}^{\alpha}p^{(\tau)}\right)\psi \;dxdt+\int_{0}^{T-\tau}\int_{\text{\ensuremath{\mathbb{T}}}^{3}}(-\Delta)^{s/2}p^{(\tau)} (-\Delta)^{s/2}\psi \;dxdt\\
+\varepsilon\int_{0}^{T-\tau}\int_{\text{\ensuremath{\mathbb{T}}}^{3}}\nabla p^{(\tau)} \cdot\nabla\psi \;dxdt-\int_{0}^{T-\tau}\int_{\text{\ensuremath{\mathbb{T}}}^{3}}(u^{(\tau)})^{2}\psi \;dxdt & =0,\ \forall\psi\in L^{2}(0,T-\tau;\ H^{1}(\mathbb{T}^{3})),\notag
\end{align}
with the energy estimates
\begin{equation}\label{eq32}%\label{energy_est2}
H^{(\tau)}(T) + \frac{1}{\Gamma_{\alpha}  T^{1-\alpha}}\int_0^T \int_{\text{\ensuremath{\text{\ensuremath{\mathbb{T}}}^{3}}}}\left(\varrho |\nabla u^{(\tau)}|^{2}+\frac{1}{2}|(-\Delta)^{s/2}\nabla p^{(\tau)}|^{2}+\frac{\varepsilon}{2}(\Delta (p^{(\tau)})^{2}\right)\;dxdt \leq  H(u_{in},p_{in}).
\end{equation}
Moreover we have the conservation of mass for $u^{(\tau)}$
\begin{equation}
\int_{\text{\ensuremath{\mathbb{T}}}^{3}}u^{(\tau)}\;dx=\int_{\text{\ensuremath{\mathbb{T}}}^{3}}u_{in}\;dx,\label{eq30}
\end{equation}
and $L^1$-bound for $p^{(\tau)}$
\begin{equation}
\int_{\mathbb{T}^{3}}p^{(\tau)}\;dx\leq\int_{\mathbb{T}^{3}}p_{in}\;dx+\frac{1}{\alpha \Gamma_{\alpha}}H(u_{in},p_{in})  T^{\alpha}.\label{eq31}
\end{equation}
We estimate the time
derivative of the density function: 
\begin{align*}
\left\vert\int_{0}^{T}\int_{\text{\ensuremath{\mathbb{T}}}^{3}}\left(D_{\tau}^{\alpha}u^{(\tau)}\right)\phi dxdt\right\vert\leq &
% \left\vert\int_{0}^{T}\int_{\text{\ensuremath{\text{\ensuremath{\mathbb{T}}}^{3}}}}u^{(\tau)}\nabla p^{(\tau)} \cdot\nabla\phi dxdt\right\vert+\varrho\left\vert\int_{0}^{T}\int_{\text{\ensuremath{\text{\ensuremath{\mathbb{T}}}^{3}}}}\nabla u^{(\tau)} \cdot\nabla\phi dxdt\right\vert\\
 \int_{0}^{T}\left\Vert \nabla p^{(\tau)}\right\Vert _{L^{2}(\mathbb{T}^{3})}\left\Vert u^{(\tau)}\nabla\phi\right\Vert _{L^{2}(\mathbb{T}^{2})}dt\\
 & +\varrho\left\Vert u^{(\tau)}\right\Vert _{L^{2}(0,T;H^{1}(\mathbb{T}^{3}))}\left\Vert \phi\right\Vert _{L^{2}(0,T;H^{1}(\mathbb{T}^{3}))}\\
%\leq & \left\Vert \nabla p^{(\tau)}\right\Vert _{L^{\infty}(0,T;L^{2}(\mathbb{T}^{3}))}\left\Vert u^{(\tau)}\nabla\phi\right\Vert _{L^{1}(0,T;L^{2}(\mathbb{T}^{3}))}\\
% & +\varrho\left\Vert u^{(\tau)}\right\Vert _{L^{2}(0,T;H^{1}(\mathbb{T}^{3}))}\left\Vert \phi\right\Vert _{L^{2}(0,T;H^{1}(\mathbb{T}^{3}))}\\
\leq & \left\Vert \nabla p^{(\tau)}\right\Vert _{L^{\infty}(0,T;L^{2}(\mathbb{T}^{3}))}\left\Vert u^{(\tau)}\right\Vert _{L^{2}(0,T;L^{6}(\mathbb{T}^{3}))}\left\Vert \nabla\phi\right\Vert _{L^{2}(0,T;L^{3}(\mathbb{T}^{3}))}\\
 & +\varrho\left\Vert u^{(\tau)}\right\Vert _{L^{2}(0,T;H^{1}(\mathbb{T}^{3}))}\left\Vert \phi\right\Vert _{L^{2}(0,T;H^{1}(\mathbb{T}^{3}))}\\
\leq & C(T,\varrho)\left\Vert \phi\right\Vert _{L^{2}(0,T;W^{1,3}(\mathbb{T}^{3}))}.
\end{align*}
Therefore
\begin{align}
\left\Vert D_{\tau}^{\alpha}u^{(\tau)}\right\Vert _{L^{2}(0,T;(W^{1,3}(\mathbb{T}^{3}))')}\leq C(T,\varrho).\label{eqCaputoUBdd}
\end{align}
%Then Aubin-Lions
%lemma \Cref{thm:main} implies the strong convergence of $ {u}^{(\tau)}$
%in $L^{2}(0,T;L^{2}(\mathbb{T}^{3}))$.

From \eqref{eq32} we have that %$\nabla p^{(\tau)}$ is uniformly bounded in $ L^{\infty}(0,T;L^{2}(\mathbb{T}^{3}))$,and 
$u^{(\tau)}$ is uniformly bounded in $L^{2}(0,T;H^{1}(\mathbb{T}^{3}))$. %Then Gagliardo-Nirenberg-Sobolev inequality yields uniformly boundedness in $L^{2}(0,T;L^{6}(\mathbb{T}^{3}))$ for $u^{(\tau)}$.
Since $H^1(\mathbb{T}^{3})$ is compactly embedded in $L^{6-\delta}(\mathbb{T}^{3})$ for $0<\delta<6$, \Cref{thm:mainII} with $X = H^1(\mathbb{T}^3)$, $B = L^{6-\delta}(\mathbb{T}^3)$, $Y = (W^{1,3})'(\mathbb{T}^3)$ yields

\[
 {u}^{(\tau)}\rightarrow{u}\ in\ L^{2}(0,T; L^{6-\delta}(\mathbb{T}^{3})),%\ \forall\  \delta\in(0,5],
\]
and we also have

\[
 {p}^{(\tau)}\rightharpoonup^{*} {p}\ in\ L^{\infty}(0,T; L^{6}(\mathbb{T}^{3})),%\ \forall\  2\leq p<\infty,
\]
and \begin{align*}u^{(\tau)}\rightarrow\ u\ a.e.\ in\ \mathbb{T}^{3}\times[0,T].\label{AEConvU} \end{align*}

Now we start to take the limit as $\tau\rightarrow0.$ We first look at the second and third terms of \eqref{eq27}. From \eqref{eq32} we know that 
\[
(-\Delta)^{s/2}\nabla p^{(\tau)}\rightharpoonup(-\Delta)^{s/2}\nabla p,\ in\ L^{2}(0,T;L^{2}(\mathbb{T}^{3})),
\]
and
\[
\nabla p^{(\tau)}\rightharpoonup\nabla p,\ in\ L^q(0,T;L^{2}(\mathbb{T}^{3})),\forall\  q<\infty,
\]
as $\tau \rightarrow 0$, and therefore
\[
\int_{0}^{T-\tau}\int_{\text{\ensuremath{\mathbb{T}}}^{3}}(-\Delta)^{s/2}p^{(\tau)} (-\Delta)^{s/2}\psi \;dxdt\rightarrow\int_{0}^{T}\int_{\text{\ensuremath{\mathbb{T}}}^{3}}(-\Delta)^{s/2}p^{} (-\Delta)^{s/2}\psi \;dxdt,
\]
\iffalse
Now we consider the third term of \eqref{eq27}. %\[
%\varrho\int_{0}^{T-\tau}\int_{\text{\ensuremath{\mathbb{T}}}^{3}}\nabla p^{(\tau)} \nabla\psi \;dxdt.
%\] 
We know from the energy estimate \eqref{eq32} that $\Delta p^{(\tau)}\in L^{2}(0,T;L^{2}(\mathbb{T}^{3})),$
%with Gagliardo-Nirenberg-Sobolev inequality, we can see that $\nabla p^{(\tau)}\in L^{2}(0,T;L^{6}(\mathbb{T}^{3}))$.
%Since Sobolev embedding $L^{2}(0,T;H^{1}(\mathbb{T}^{3}))\hookrightarrow L^{2}(0,T;L^{6}(\mathbb{T}^{3})),$
we have $\nabla p^{(\tau)}\in L^{2}(0,T;H^{1}(\mathbb{T}^{3})),$ therefore,
$\nabla p^{(\tau)}\in L^{2}(0,T;L^{2}(\mathbb{T}^{3})),$ and the
uniform boundedness suggests weak convergence, i.e. $\forall\psi\in L^{2}(0,T-\tau;\ L^{2}(\mathbb{T}^{3})),$\fi
\[
\varepsilon\int_{0}^{T-\tau}\int_{\text{\ensuremath{\mathbb{T}}}^{3}}\nabla p^{(\tau)} \cdot\nabla\psi \;dxdt\rightarrow\varepsilon\int_{0}^{T}\int_{\text{\ensuremath{\mathbb{T}}}^{3}}\nabla p^{} \cdot\nabla\psi \;dxdt.
\]

Now consider the fourth term. %\[
%\int_{0}^{T-\tau}\int_{\text{\ensuremath{\mathbb{T}}}^{3}}( {u}^{(\tau)})^{2}\psi \;dxdt.
%\]
%From \eqref{eq32} we know that $( {u}^{(\tau)})^{2}\in L^{\infty}(0,T;L^{1}(\mathbb{T}^{3}))$ and $u^{(\tau)}\in L^{2}(0,T;L^{6}(\mathbb{T}^{3}))$,
%suggests that $(u^{(\tau)})^{2}\in L^{1}(0,T;L^{3}(\mathbb{T}^{3})),$
%then 
%\[
%( {u}^{(\tau)})^{2}\in L^{\infty}(0,T;L^{1}(\mathbb{T}^{3}))\cap L^{1}(0,T;L^{3}(\mathbb{T}^{3})).
%\]
%Interpolation yields 
%\[
%\left\Vert ( {u}^{(\tau)})^{2}\right\Vert _{L^{2}(0,T;L^{3/2}(\mathbb{T}^{3}))}\leq\left\Vert ( {u}^{(\tau)})^{2}\right\Vert _{L^{\infty}(0,T;L^{1}(\mathbb{T}^{3}))}\left\Vert ( {u}^{(\tau)})^{2}\right\Vert _{L^{1}(0,T;L^{3}(\mathbb{T}^{3}))}.
%\]
%%%%%%%%%%%%%%%%%%%%%%%%%%%%%%%%%iffalse
\iffalse\textit{\textcolor{black}{Claim:}}\textcolor{black}{{} $\forall g\in L^{3}(\mathbb{T}^{3}),$
$g\in C_{0}^{\infty},$ we have the weak convergence 
\[
( {u}^{(\tau)})^{2}\rightharpoonup{u}{}^{2},\ in\ L^{2}(0,T;L^{3/2}(\mathbb{T}^{3})).
\]
}\fi
%%%%%%%%%%%%%%%%%%%%%%%%%%%%%%%%%fi
%The uniform boundedness of $\left\Vert ( {u}^{(\tau)})^{2}\right\Vert _{L^{2}(0,T;L^{3/2}(\mathbb{T}^{3}))}$ suggests %the boundedness of $\left\Vert ( {u}^{(\tau)})^{2}\right\Vert _{L^{2}(0,T;L^{4/3}(\mathbb{T}^{3}))}$, we have 
%the weak convergence of $(u^{(\tau)})^2$ in $L^{2}(0,T;L^{4/3}(\mathbb{T}^{3}))$. 
We only need to prove that $(u^{(\tau)})^2$ converges to ${u}^{2}$. For all %$g\in L^{3}(\mathbb{T}^{3}),$
$\psi \in L^2(0,T;L^{4}(\mathbb{T}^{3})),$ consider
\begin{align}
\label{I1I2u^2}\int\int[( {u}^{(\tau)})^{2}-{u}^{2}] \psi \;dxdt & =\int\int[( {u}^{(\tau)})^{2}- {u}^{(\tau)}{u}+ {u}^{(\tau)}{u}-{u}^{2}]\psi \;dxdt\\
 & =\int\int {u}^{(\tau)}\psi ( {u}^{(\tau)}-{u})\;dxdt+\int\int{u}\psi ( {u}^{(\tau)}-{u})\;dxdt\nonumber\\
 & =I_{1}+I_{2}.\nonumber
\end{align}
First, we look at $I_{1}$:
\begin{align*}
I_{1}% & =\int\int {u}^{(\tau)}\psi( {u}^{(\tau)}-{u})\;dxdt\\
 & \leq\left\Vert  {u}^{(\tau)}\psi\right\Vert _{L^{2}(0,T;L^{4/3}(\mathbb{T}^{3}))}\left\Vert  {u}^{(\tau)}-{u}\right\Vert _{L^{2}(0,T;L^4(\mathbb{T}^{3}))}\\
 & \leq\left\Vert  {u}^{(\tau)}\right\Vert _{L^{\infty}(0,T;L^{2}(\mathbb{T}^{3}))}\left\Vert \psi\right\Vert _{L^{2}(0,T;L^{4}(\mathbb{T}^{3}))}\left\Vert  {u}^{(\tau)}-{u}\right\Vert _{L^{2}(0,T;L^{4}(\mathbb{T}^{3}))}.
\end{align*}
%As ${u}^{(\tau)}\rightarrow u\ in\ L^{2}(0,T; L^{6-\delta}(\mathbb{T}^{3})),$
%we have $\left\Vert  {u}^{(\tau)}-u\right\Vert _{L^{2}(0,T;L^{3}(\mathbb{T}^{3}))}\rightarrow0$
%as $\tau\rightarrow\infty,$ 
%Now we consider $ {u}^{(\tau)}$ and $g$.} \textcolor{black}{Since $g\in C_{0}^{\infty},$ and $C_{0}^{\infty}$ is dense in any $L^{p},$ for $p\geq1,$ 
%Since $ {u}^{(\tau)}$ converges to $u$ strongly  in $L^{2}(0,T;L^{1}\cap L^{6-\delta}),$
%also, $\left\Vert  {u}^{(\tau)}\right\Vert _{L^{2}(0,T;L^{3}(\mathbb{T}^{3}))}$ is bounded. %$ {u}^{(\tau)}g\in L^{2}(0,T;L^{2}(\mathbb{T}^{3})).$} 
%where $\delta\leq 2$. 
Therefore, $I_{1}\rightarrow 0$ as $\tau\rightarrow 0.$ Furthermore, the convergence $I_{2}\rightarrow 0$ as $\tau\rightarrow 0$ follows from the weak convergence of $u^{(\tau)}$ in $L^2(0,T; L^4(\mathbb{T}^3))$, and $u\psi \in L^2(0,T; L^\frac{4}{3}(\mathbb{T}^3))$.

%%%%%%%%%%%%%%%%%%%%%%%%%%%%%%%%%%%%%%iffalse
\iffalse
\textcolor{black}{Since 
\[
( {u}^{(\tau)})^{2}\rightharpoonup{u}{}^{2},\ in\ L^{2}(0,T;L^{3/2}(\mathbb{T}^{3})),
\]
}
and $\psi\in L^{2}(0,T-\tau;\ L^{3}(\mathbb{T}^{3})),$
we have the weak convergence:
\textcolor{black}{
\[
\int_{0}^{T-\tau}\int_{\text{\ensuremath{\mathbb{T}}}^{3}}( {u}^{(\tau)})^{2}\psi \;dxdt\rightarrow\int_{0}^{T}\int_{\text{\ensuremath{\mathbb{T}}}^{3}}{u}{}^{2}\psi \;dxdt.
\]
}
\fi
%%%%%%%%%%%%%%%%%%%%%%%%%%%%%%%%%%%%%%%%fi

By density argument and \Cref{prop.conv.dis.byparts} we obtain
\[
\int_{0}^{T}\int_{\text{\ensuremath{\mathbb{T}}}^{3}}\left(D_{\tau}^{\alpha} {u}^{(\tau)}\right)\phi \;dxdt\rightarrow \int_{0}^{T}\int_{\text{\ensuremath{\mathbb{T}}}^{3}}\langle D_{t}^{\alpha}{u}, \phi \rangle \;dxdt, \ as \ \tau\rightarrow 0.
\]

Now we look at the second term in \eqref{eq26}. %\[
%\int_{0}^{T-\tau}\int_{\text{\ensuremath{\text{\ensuremath{\mathbb{T}}}^{3}}}}u^{(\tau)}\nabla p^{(\tau)} \nabla\phi \;dxdt.
%\]
Since %$p^{(\tau)}$ is uniformly bounded  in $L^{\infty}(0,T;L^{1}(\mathbb{T}^{3}))$ from \eqref{eq31}, and 
$\nabla p^{(\tau)}$ is uniformly bounded in $L^{\infty}(0,T;L^{2}(\mathbb{T}^{3}))$ from \eqref{eq32}, we have the convergence %weak convergence
\begin{align*}
%\nabla p^{(\tau)} & \rightharpoonup\nabla p\ in\ L^{q}(0,T;L^{2}(\mathbb{T}^{3})), \ q<\infty. \\
\nabla p^{(\tau)}& \rightharpoonup^*\nabla p\ in\ L^{\infty}(0,T;L^{2}(\mathbb{T}^{3})).
\end{align*}
Moreover, $(-\Delta)^{s/2} p^{(\tau)}\in L^{2}(0,T;H^{1}(\mathbb{T}^{3}))$
from \eqref{eq32}, and, by Sobolev embedding, %Sobolev embedding suggests that $H^{s}(\mathbb{T}^{3})\hookrightarrow L^{p}(\mathbb{T}^{3}),$ where $s-\frac{3}{2}=0-\frac{3}{p},$ which implies $p=\frac{6}{3-2s},$ 
$\nabla p^{(\tau)}$ is uniformly bounded in $L^{2}(0,T;L^{\frac{6}{3-2s}}(\mathbb{T}^{3})).$
Thus, %we know that $\nabla p$ is uniformly bounded in $L^{2}(0,T;L^{\frac{6}{3-2s}}(\mathbb{T}^{3})),$ 
there exists
a subsequence $\nabla p^{(\tau)}$ that converges weakly in\\ $L^{2}(0,T;L^{\frac{6}{3-2s}}(\mathbb{T}^{3})).$
%Take $s=\frac{1}{2},$ we have $\nabla p^{(\tau)}\rightharpoonup\nabla p$
%in $L^{2}(0,T;L^{3}(\mathbb{T}^{3})).$
%\textit{\textcolor{black}{Claim:}}\textcolor{black}{{} $\forall g\in L^{4}(\mathbb{T}^{3}),$
%g\in C_{0}^{\infty},$ (therefore is dense in any $L^{p},$ for $p\geq1,$ )we have the weak convergence 
%Next, we show that
%\[
 %{u}^{(\tau)}\nabla p^{(\tau)}\rightharpoonup{u}\nabla p,\ in\ L^{2}(0,T;L^{6/5}(\mathbb{T}^{3})).
%\]
%\textcolor{black}{Since we have the uniform boundedness of $\left\Vert  {u}^{(\tau)}\nabla p^{(\tau)}\right\Vert _{L^{2}(0,T;L^{2}(\mathbb{T}^{3}))}$by H\"older's inequality, we have the weak convergence, now we only need to prove the difference is $0$.
%For all $\nabla \phi\in L^2(0,T;L^{6}(\mathbb{T}^{3}))$, 
Consider
\begin{align}\label{I1I2up}
\int\int[ {u}^{(\tau)}\nabla p^{(\tau)}-{u}\nabla p]\cdot\nabla \phi\;dxdt & =\int\int[ {u}^{(\tau)}\nabla p^{(\tau)}- {u}\nabla p^{(\tau)}+ {u}\nabla p^{(\tau)}-{u}\nabla p]\cdot\nabla \phi\;dxdt\\
 & =\int\int {u}\nabla \phi\cdot(\nabla p^{(\tau)}-\nabla p)\;dxdt+\int\int\nabla p^{(\tau)}\cdot\nabla \phi( {u}^{(\tau)}-{u})\;dxdt\notag\\
 & =I_{1}+I_{2}.\notag
\end{align}
%For $I_{1},$
%\begin{align*}
%I_{1} & =\int\int {u}g(\nabla p^{(\tau)}-\nabla p)\;dxdt.
%\end{align*}
{Using $\nabla p^{(\tau)} \rightharpoonup^*\nabla p\ \textrm{ in}\ L^{\infty}(0,T;L^{2}(\mathbb{T}^{3}))$, $\nabla \phi\in L^{2}(0,T;L^{\frac{2(6-\delta)}{4-\delta}}(\mathbb{T}^{3}))$ and $u \in L^{2}(0,T;L^{6-\delta}(\mathbb{T}^{3}))$, so that $u\nabla \phi$ is bounded in $L^{1}(0,T;L^{2}(\mathbb{T}^{3})),$  we have that $I_{1}\rightarrow0$ as $\tau\rightarrow0.$} We bound $I_{2}$ as 
\begin{align*}
I_{2}  %\int\int\nabla p^{(\tau)}g( {u}^{(\tau)}-{u})\;dxdt\\
  \leq\left\Vert \nabla p^{(\tau)}\right\Vert _{L^{\infty}(0,T;L^{2}(\mathbb{T}^{3}))}\left\Vert \nabla  \phi\right\Vert _{L^{2}(0,T;L^{\frac{2(6-\delta)}{4-\delta}}(\mathbb{T}^{3}))}\left\Vert  {u}^{(\tau)}-{u}\right\Vert _{L^{2}(0,T;L^{6-\delta}(\mathbb{T}^{3}))}.
\end{align*}
We have previously concluded that $ {u}^{(\tau)}\rightarrow{u}$ in $L^{2}(0,T;L^{6-\delta}(\mathbb{T}^{3})),$
and $\nabla p^{(\tau)}$ is uniformly bounded in $L^{\infty}(0,T;L^{2}(\mathbb{T}^{3})).$ Therefore,
$I_{2}\rightarrow0$ as $\tau\rightarrow0.$ Hence, $\forall\phi\in L^{2}(0,T-\tau;\ W^{1,q}(\mathbb{T}^{3}))$ with $q>3$ we have 
\[
\int_{0}^{T-\tau}\int_{\text{\ensuremath{\text{\ensuremath{\mathbb{T}}}^{3}}}}u^{(\tau)}\nabla p^{(\tau)} \cdot\nabla\phi \;dxdt\rightarrow\int_{0}^{T}\int_{\text{\ensuremath{\text{\ensuremath{\mathbb{T}}}^{3}}}}u^{}\nabla p^{} \cdot\nabla\phi \;dxdt.
\]

Next, we consider the third term in \eqref{eq26}. From the energy estimates \eqref{eq32} we have 
\[
\nabla {u}^{(\tau)}\rightharpoonup\nabla{u},\ in\ L^{2}(0,T;L^{2}(\mathbb{T}^{3})),
\]
and therefore, as $\tau \rightarrow 0$
\[
\varrho\int_{0}^{T-\tau}\int_{\text{\ensuremath{\text{\ensuremath{\mathbb{T}}}^{3}}}}\nabla {u}^{(\tau)} \cdot\nabla\phi \;dxdt\rightarrow\varrho\int_{0}^{T}\int_{\text{\ensuremath{\text{\ensuremath{\mathbb{T}}}^{3}}}}\nabla{u} \cdot\nabla\phi \;dxdt.
\]

Now we proceed to the first term of\textcolor{black}{{} {} }\eqref{eq27}:
%Since we have \textcolor{black}{$(-\Delta)^{s/2}\nabla p^{(\tau)}\in L^{2}(0,T;L^{2}(\mathbb{T}^{3}))$,$\nabla p\in L^{2}(0,T;L^{\frac{6}{3-2s}}(\mathbb{T}^{3})),$} and\textcolor{black}{$( {u}^{(\tau)})^{2}\in L^{2}(0,T;L^{3/2}(\mathbb{T}^{3})),$} we estimate the time derivative of the density function: 
\begin{align*}
\left\vert\int_{0}^{T}\int_{\text{\ensuremath{\mathbb{T}}}^{3}}\left(D_{\tau}^{\alpha}p^{(\tau)}\right)\psi \;dxdt\right\vert\leq &
%&\; \left\vert\int_{0}^{T-\tau}\int_{\text{\ensuremath{\mathbb{T}}}^{3}}(-\Delta)^{s/2}p^{(\tau)} (-\Delta)^{s/2}\psi \;dxdt\right\vert\\
% & +\varepsilon\left\vert\int_{0}^{T-\tau}\int_{\text{\ensuremath{\mathbb{T}}}^{3}}\nabla p^{(\tau)} \cdot\nabla\psi \;dxdt\right\vert+\left\vert\int_{0}^{T-\tau}\int_{\text{\ensuremath{\mathbb{T}}}^{3}}( {u}^{(\tau)})^{2}\psi \;dxdt\right\vert\\
\; \left\Vert (-\Delta)^{s/2}p^{(\tau)}\right\Vert _{L^{2}(0,T;L^{2}(\mathbb{T}^{3}))}\left\Vert (-\Delta)^{s/2}\psi\right\Vert _{L^{2}(0,T;L^{2}(\mathbb{T}^{3}))}\\
 & +\left\Vert \nabla p^{(\tau)}\right\Vert _{L^{2}(0,T;L^{4}(\mathbb{T}^{3}))}\left\Vert \psi\right\Vert _{L^{2}(0,T;L^{4/3}(\mathbb{T}^{3}))}\\
 & +\left\Vert ( {u}^{(\tau)})^{2}\right\Vert _{L^{2}(0,T;L^{3/2}(\mathbb{T}^{3}))}\left\Vert \psi\right\Vert _{L^{2}(0,T;L^{3}(\mathbb{T}^{3}))}.
%\leq &\; C(T)\left\Vert \psi\right\Vert _{L^{2}(0,T;H^{1}(\mathbb{T}^{3}))}.
\end{align*}
Since $( {u}^{(\tau)})^{2}$ is uniformly bounded in $L^{\infty}(0,T;L^{1}(\mathbb{T}^{3}))\cap L^{1}(0,T;L^{3}(\mathbb{T}^{3}))$ by \eqref{eq32}, interpolation yields 
\[
\left\Vert ( {u}^{(\tau)})^{2}\right\Vert _{L^{2}(0,T;L^{3/2}(\mathbb{T}^{3}))}\leq\left\Vert ( {u}^{(\tau)})^{2}\right\Vert _{L^{\infty}(0,T;L^{1}(\mathbb{T}^{3}))}\left\Vert ( {u}^{(\tau)})^{2}\right\Vert _{L^{1}(0,T;L^{3}(\mathbb{T}^{3}))}.
\]Then $D_{\tau}^{\alpha}p^{(\tau)}$ is uniformly bounded in the dual space of $L^{2}(0,T;H^{1}(\mathbb{T}^{3}))$ by some constant
function depending on $T$ and on the initial data: 
\begin{align}
\left\Vert D_{\tau}^{\alpha}p^{(\tau)}\right\Vert _{L^{2}(0,T; H^{-1}(\mathbb{T}^{3}))}\leq C(T, H_{in}).\label{eqCaputoPBound}
\end{align}
Moreover, taking into account that $p^{(\tau)}\in L^{2}(0,T;H^{2}(\mathbb{T}^{3}))$, we can use \Cref{thm:mainII} with $X = H^1(\mathbb{T}^3)$, $B = L^2(\mathbb{T}^3)$, $Y = H^{-1}(\mathbb{T}^3)$ to conclude that 
\begin{align*}
p^{(\tau)}\rightarrow p \ in \  L^{2}(0,T;L^{2}(\mathbb{T}^{3})),
\end{align*}
%Using H\"older's inequality we have
%\[
%p^{(\tau)}\rightarrow p \ in \  L^{1}(0,T;L^{1}(\mathbb{T}^{3})),
%\]
%for every $\Omega\subset \mathbb{T}^{3}$ bounded, 
and 
\begin{align*}
p^{(\tau)}\rightarrow\ p\ a.e.\ in\ \mathbb{T}^{3}\times[0,T].\label{AEConv}
\end{align*}
%We apply \Cref{prop:intbyparts_dis} to $\int_{0}^{T-\tau}\int_{\mathbb{T}^{3}}\left(D_{\tau}^{\alpha}p^{(\tau)}\right)\psi \;dxdt$, pass the limit as $\tau \rightarrow 0$ 
By \Cref{prop.conv.dis.byparts}, we get%then apply \eqref{byparts.almeida} and get
\[
\int_{0}^{T-\tau}\int_{\text{\ensuremath{\mathbb{T}}}^{3}}\left(D_{\tau}^{\alpha}p^{(\tau)}\right)\psi \;dxdt\rightarrow \int_{0}^{T}\int_{\text{\ensuremath{\mathbb{T}}}^{3}}\langle D_{t}^{\alpha}p^{}, \psi \rangle \;dxdt.
\]
Summarizing, we have (after the density argument)
\begin{align}
\int_{0}^{T}\int_{\text{\ensuremath{\mathbb{T}}}^{3}}\left<D_{t}^{\alpha}u,\phi\right> \;dxdt+\int_{0}^{T}\int_{\text{\ensuremath{\text{\ensuremath{\mathbb{T}}}^{3}}}}u\nabla p \cdot\nabla\phi \;dxdt\nonumber \\
+\varrho\int_{0}^{T}\int_{\text{\ensuremath{\text{\ensuremath{\mathbb{T}}}^{3}}}}\nabla u \cdot\nabla\phi \;dxdt & =0,\label{eq33ep}
\end{align}
\begin{align}
\int_{0}^{T}\int_{\mathbb{T}^{3}}\left<D_{t}^{\alpha}p,\psi\right> \;dxdt+\int_{0}^{T}\int_{\text{\ensuremath{\mathbb{T}}}^{3}}(-\Delta)^{s/2}p (-\Delta)^{s/2}\psi \;dxdt\nonumber \\
+\varepsilon\int_{0}^{T}\int_{\text{\ensuremath{\mathbb{T}}}^{3}}\nabla p \cdot\nabla\psi \;dxdt-\int_{0}^{T}\int_{\text{\ensuremath{\mathbb{T}}}^{3}}u^2\psi \;dxdt & =0,\label{eq34ep}
\end{align}
 for all $\phi\in L^{2}(0,T;\ W^{1,3}(\mathbb{T}^{3}))$ and $\psi\in L^{2}(0,T;\ H^{1}(\mathbb{T}^{3}))$. 

The next step is the limit $\tau \rightarrow 0$ in the energy estimate \eqref{eq32}. Taking $\liminf_{\tau\rightarrow0}$ on both sides,
%%%%%%%%%%%%%%%%%%%%%%%%%%%%%%%%%%iffalse(details for dissertation)
\iffalse
\begin{align*}
&\liminf_{\tau\rightarrow0}\int_{\mathbb{T}^{3}}\left(( {u}^{(\tau)})^{2}+\frac{1}{2}|\nabla p^{(\tau)}|^{2}\right)\;dx+K\liminf_{\tau\rightarrow0}\bigg( \varrho\int_{0}^{t}\int_{\text{\ensuremath{\text{\ensuremath{\mathbb{T}}}^{3}}}}|\nabla {u}^{(\tau)}(s,x)|^{2}\;dxds\notag\\
+&\frac{1}{2}\int_{0}^{t}\int_{\text{\ensuremath{\mathbb{T}}}^{3}}|(-\Delta)^{s/2}\nabla p^{(\tau)}(s,x)|^{2}\;dxds+\frac{\varepsilon}{2}\int_{0}^{t}\int_{\text{\ensuremath{\mathbb{T}}}^{3}}(\Delta p^{(\tau)}(s,x))^{2}\;dxdt\bigg) \notag \\
\leq & \int_{\mathbb{T}^{3}}\left(u_{in}^{2}+\frac{1}{2}|\nabla p_{in}|^{2}\right)\;dx.
\end{align*}

We look at the first integral:
$$\liminf_{\tau\rightarrow0}\int_{\mathbb{T}^{3}}\left(( {u}^{(\tau)})^{2}+\frac{1}{2}|\nabla p^{(\tau)}|^{2}\right)\;dx.$$
Since we know that $ {u}^{(\tau)}\rightharpoonup {u}$ in $L^{\infty}(0,T;L^2(\mathbb{T}^3))$, %by \textbf{Proposition 3.5(iii)} by Brezis (
\fi
%%%%%%%%%%%%%%%%%%%%%%%%%%%%fi
by the lower weak semicontinuity of $L^p$ norm, we get
%%%%%%%%%%%%%%%%%%%%%%%%%%%%iffalse (details for dissertation)
\iffalse we know that $\left\Vert  {u}\right\Vert _{L^{\infty}(0,T;L^2(\mathbb{T}^3))}\leq\liminf_{\tau\rightarrow0}\left\Vert  {u}^{(\tau)}\right\Vert _{L^{\infty}(0,T;L^2(\mathbb{T}^3))},$
and
$\left\Vert \nabla p\right\Vert _{L^{\infty}(0,T;L^2(\mathbb{T}^3))}\leq\liminf_{\tau\rightarrow0}\left\Vert \nabla p^{(\tau)}\right\Vert _{L^{\infty}(0,T;L^2(\mathbb{T}^3))},$
therefore, %combine with Fatou's Lemma, 
we have 
\[
\int_{\mathbb{T}^{3}}\left({u}^{2}+\frac{1}{2}|\nabla p|^{2}\right)\;dx\leq\liminf_{\tau\rightarrow0}\int_{\mathbb{T}^{3}}\left(( {u}^{(\tau)})^{2}+\frac{1}{2}|\nabla p^{(\tau)}|^{2}\right)\;dx.
\]

Similarly, if we apply the same on the rest of the terms in the energy estimate, we have\fi
%%%%%%%%%%%%%%%%%%%%%%%%%%%%%%%fi
\begin{align}\label{eq35ep}
 & \int_{\mathbb{T}^{3}}\left(u^2+\frac{1}{2}|\nabla p|^2\right)\;dx
+\frac{t^{\alpha -1}}{\Gamma_{\alpha}} \int_0^t\int_{\text{\ensuremath{\text{\ensuremath{\mathbb{T}}}^{3}}}}\Big[\varrho|\nabla u|^{2}+\frac{1}{2}|(-\Delta)^{s/2} \nabla p|^{2}+\frac{\varepsilon}{2}(\Delta p)^{2}\Big]\;dxds  \\
\leq & \int_{\mathbb{T}^{3}}\left(u_{in}^2+\frac{1}{2}|\nabla p_{in}|^{2}\right)\;dx.\notag
\end{align}
%Since we know that $u^{(\tau)}$ is uniformly bounded in $L^{2}(0,T;H^{1}(\mathbb{T}^{3}))$ from \eqref{eq32}, and $H^1(\mathbb{T}^{3})$ is compactly embedded in $L^{2}(\mathbb{T}^{3})$, \Cref{thm:mainII} with $X = H^1(\mathbb{T}^3)$, $B = L^2(\mathbb{T}^3)$, $Y = (W^{1,3})'(\mathbb{T}^3)$ yields
%\[
%u^{(\tau)}\rightarrow u \ in \  L^{2}(0,T;L^{2}(\mathbb{T}^{3})),
%\]
%Using H\"older's inequality we have
%\begin{align*}
%u^{(\tau)}\rightarrow u \ in \  L^{1}(0,T;L^{1}(\mathbb{T}^{3})),
%\end{align*}
%and \begin{align}u^{(\tau)}\rightarrow\ u\ a.e.\ in\ \mathbb{T}^{3}\times[0,T].\label{AEConvU} \end{align}
Furthermore, take $\liminf_{\tau\rightarrow0}$ on both sides of \eqref{eq30} and \eqref{eq31}, \iffalse
\begin{equation*}
\liminf_{\tau\rightarrow0}\int_{\text{\ensuremath{\mathbb{T}}}^{3}} {u}^{(\tau)}\;dx=\int_{\text{\ensuremath{\mathbb{T}}}^{3}}u_{in}\;dx.
\end{equation*}
Since we have the \fi 
everywhere convergence of $ {u}^{(\tau)}$ and $p^{(\tau)}$ %( \eqref{AEConv} and \eqref{AEConvU}) 
yields %combine with Fatou's Lemma, 
\begin{equation}\label{eq36ep}
\int_{\text{\ensuremath{\mathbb{T}}}^{3}} {u}\;dx\leq \int_{\text{\ensuremath{\mathbb{T}}}^{3}}u_{in}\;dx,
\end{equation}
and
\begin{equation}\label{eq37ep}
\int_{\mathbb{T}^{3}}p\;dx\leq\int_{\mathbb{T}^{3}}p_{in}\;dx+ \frac{1}{\alpha \Gamma_{\alpha}} H(u_{in},p_{in})  T^{\alpha}.
\end{equation}

\section{Limit $\varepsilon \rightarrow0$}\label{sec:epsilonLimit}
This section is devoted to the limit $\varepsilon\rightarrow 0$. %We denote $u$ and $p$ that we get from the last section by $u^{(\varepsilon)}$ and $p^{(\varepsilon)}$ respectively in this section.  
For convenience we recall the statement of a compactness result proven in \cite{li2018some}:
\begin{theorem}
\label{lem:main.cont}
Let $X$, $B$ and $Y$ be Banach spaces. $X\hookrightarrow B$ compactly and $B\hookrightarrow Y$ continuously. Let $1\leq r\leq \infty$, $0 < \alpha < 1$. Suppose $u\in L^1_{loc}(0,T;X)$ satisfies:
$$\norm{u}_{L^r(0,T;X)}+ \norm{D^{\alpha}_t u} _{L^r(0,T;Y)} \leq C_0.$$ \iffalse where $1>1/r-1/p$.\fi Then $u$ is relatively compact in $ L^r(0,T;B)$.
\end{theorem}

%From \eqref{eq35ep} we know that, as $\varepsilon\rightarrow 0$,
%\[
%(-\Delta)^{s/2}\nabla p^{(\varepsilon)}\rightharpoonup(-\Delta)^{s/2}\nabla p,\ in\ L^{2}(0,T;L^{2}(\mathbb{T}^{3})),
%\]
%and therefore
%\[
%\int_{0}^{T}\int_{\text{\ensuremath{\mathbb{T}}}^{3}}(-\Delta)^{s/2}p^{(\varepsilon)} (-\Delta)^{s/2}\psi \;dxdt\rightarrow\int_{0}^{T}\int_{\text{\ensuremath{\mathbb{T}}}^{3}}(-\Delta)^{s/2}p^{} (-\Delta)^{s/2}\psi \;dxdt.
%\]

%To pass to the limit, we first need the uniform boundedness of $( {u}^{(\varepsilon)})^{2}$ in ${L^{2}(0,T;L^{4/3}(\mathbb{T}^{3}))}.$ Again from \eqref{eq35ep} we know that $( {u}^{(\varepsilon)})^{2}\in L^{\infty}(0,T;L^{1}(\mathbb{T}^{3}))$ and $u^{(\varepsilon)}\in L^{2}(0,T;L^{6}(\mathbb{T}^{3}))$,
%suggests that $(u^{(\tau)})^{2}\in L^{1}(0,T;L^{3}(\mathbb{T}^{3})),$
%then 
%\[
%( {u}^{(\varepsilon)})^{2}\in L^{\infty}(0,T;L^{1}(\mathbb{T}^{3}))\cap L^{1}(0,T;L^{3}(\mathbb{T}^{3})),
%\]
%which implies
%\[
 %\left\Vert ( {u}^{(\varepsilon)})^{2}\right\Vert _{L^{2}(0,T;L^{4/3}(\mathbb{T}^{3}))}\leq \left\Vert ( {u}^{(\varepsilon)})^{2}\right\Vert _{L^{2}(0,T;L^{3/2}(\mathbb{T}^{3}))}\leq C.%\left\Vert ( {u}^{(\varepsilon)})^{2}\right\Vert _{L^{\infty}(0,T;L^{1}(\mathbb{T}^{3}))}\left\Vert ( {u}^{(\varepsilon)})^{2}\right\Vert _{L^{1}(0,T;L^{3}(\mathbb{T}^{3}))}.
%\]
%To pass to the limit we now only need to prove that $({u}^{(\varepsilon)})^{2}$ converges to ${u}^{2}$ in $L^2(0,T;L^{4/3}(\mathbb{T}^3))$. 
{Similar to \eqref{eqCaputoUBdd}, we can conclude from \eqref{eq33ep} that
\[
\left\Vert D_{t}^{\alpha}u^{(\varepsilon)}\right\Vert _{L^{2}(0,T;(W^{1,3}(\mathbb{T}^{3}))')}\leq C(T,\varrho).
\]
  Moreover, by \eqref{eq35ep} $u^{(\varepsilon)}$ is uniformly bounded in $L^{2}(0,T;H^{1}(\mathbb{T}^{3}))$.} %Then Gagliardo-Nirenberg-Sobolev inequality yields uniformly boundedness in $L^{2}(0,T;L^{6}(\mathbb{T}^{3}))$ for $u^{(\tau)}$.
Since $H^1(\mathbb{T}^{3})$ is compactly embedded in $L^{6-\delta}(\mathbb{T}^{3})$, \Cref{lem:main.cont} with $X = H^1(\mathbb{T}^3)$, $B = L^{6-\delta}(\mathbb{T}^3)$, $Y = (W^{1,3})'(\mathbb{T}^3)$ yields $$
 {u}^{(\varepsilon)}\rightarrow{u}\ in\ L^{2}(0,T; L^{6-\delta}(\mathbb{T}^{3})),$$ and $$u^{(\varepsilon)}\rightarrow\ u\ a.e.\ in\ \mathbb{T}^{3}\times[0,T].$$
Therefore, similarly as for \eqref{I1I2u^2}, %in \Cref{sec:tauLimit}, 
we conclude
\[
\int_{0}^{T}\int_{\text{\ensuremath{\mathbb{T}}}^{3}}( {u}^{(\varepsilon)})^{2}\psi \;dxdt\rightarrow\int_{0}^{T}\int_{\text{\ensuremath{\mathbb{T}}}^{3}}{u}{}^{2}\psi \;dxdt\textrm{, as }\varepsilon\rightarrow 0.
\]
%Now consider the first term of \eqref{eq33ep}. 
%We first integrate by parts using \eqref{byparts.almeida}, pass to the limit as $\varepsilon \rightarrow 0$, then use \eqref{byparts.almeida} again to get
%\[
%\int_{0}^{T}\int_{\text{\ensuremath{\mathbb{T}}}^{3}}\left<D_{t}^{\alpha}{u}^{(\varepsilon)},\phi\right> \;dxdt\rightarrow\int_{0}^{T}\int_{\text{\ensuremath{\mathbb{T}}}^{3}}\left<D_{t}^{\alpha}{u},\phi\right>\;dxdt.
%\]

The second term in \eqref{eq33ep} can be handled in the exact same way as  \eqref{I1I2up}. Summarizing up, after a density argument,
\begin{align}
\int_{0}^{T}\int_{\text{\ensuremath{\mathbb{T}}}^{3}}\left<D_{t}^{\alpha}u,\phi\right> \;dxdt+\int_{0}^{T}\int_{\text{\ensuremath{\text{\ensuremath{\mathbb{T}}}^{3}}}}u\nabla p \cdot\nabla\phi \;dxdt\nonumber \\
+\varrho\int_{0}^{T}\int_{\text{\ensuremath{\text{\ensuremath{\mathbb{T}}}^{3}}}}\nabla u \cdot\nabla\phi \;dxdt & =0,\label{eq33}
\end{align}
\begin{align}
\int_{0}^{T}\int_{\mathbb{T}^{3}}\left<D_{t}^{\alpha}p,\psi\right> \;dxdt+\int_{0}^{T}\int_{\text{\ensuremath{\mathbb{T}}}^3}(-\Delta)^{s/2}p (-\Delta)^{s/2}\psi \;dxdt\nonumber \\
-\int_{0}^{T}\int_{\text{\ensuremath{\mathbb{T}}}^{3}}u^2\psi \;dxdt & =0,\label{eq34}
\end{align}
 for all $\phi\in L^{2}(0,T;\ W^{1,3}(\mathbb{T}^{3}))$ and $\psi\in L^{2}(0,T;\ H^{1}(\mathbb{T}^{3}))$. 

The next step is the limit $\varepsilon \rightarrow 0$ in \eqref{eq35ep}. By the lower weak semicontinuity we have
\begin{align}\label{eq35}
 & \int_{\mathbb{T}^{3}}\left(u^2+\frac{1}{2}|\nabla p|^2\right)\;dx
+\frac{t^{\alpha -1}}{\Gamma_{\alpha}}\Big[ \varrho\int_0^t\int_{\text{\ensuremath{\text{\ensuremath{\mathbb{T}}}^{3}}}}|\nabla u|^{2}\;dxds
+\frac{1}{2}\int_0^t\int_{\text{\ensuremath{\mathbb{T}}}^{3}}|(-\Delta)^{s/2}\nabla p|^{2}\;dxds\Big] \notag\\
\leq &\; \int_{\mathbb{T}^{3}}\left(u_{in}^2+\frac{1}{2}|\nabla p_{in}|^{2}\right)\;dx.
\end{align}
{Now we try to get the everywhere convergence of $p^{(\varepsilon)}$. Same as \eqref{eqCaputoPBound}, %in \Cref{sec:tauLimit}, 
we can bound the first term of\textcolor{black}{{} {} }\eqref{eq34ep} as
\begin{align*}
\left\Vert D_{t}^{\alpha}p^{(\varepsilon)}\right\Vert _{L^{2}(0,T; H^{-1}(\mathbb{T}^{3}))}
\leq &\; \left\Vert (-\Delta)^{s/2}p^{(\tau)}\right\Vert _{L^{2}(0,T;L^{2}(\mathbb{T}^{3}))}\left\Vert (-\Delta)^{s/2}\psi\right\Vert _{L^{2}(0,T;L^{2}(\mathbb{T}^{3}))}\\
 & +\left\Vert \nabla p^{(\tau)}\right\Vert _{L^{2}(0,T;L^{4}(\mathbb{T}^{3}))}\left\Vert \psi\right\Vert _{L^{2}(0,T;L^{4/3}(\mathbb{T}^{3}))}\\
 & +\left\Vert ( {u}^{(\tau)})^{2}\right\Vert _{L^{2}(0,T;L^{3/2}(\mathbb{T}^{3}))}\left\Vert \psi\right\Vert _{L^{2}(0,T;L^{3}(\mathbb{T}^{3}))}\\
 \leq &\;  C(T, H_{in}) \left\Vert \psi\right\Vert _{L^{2}(0,T;H^{1}(\mathbb{T}^{3}))}. 
\end{align*}
\Cref{lem:main.cont} with $X = H^1(\mathbb{T}^3)$, $B = L^2(\mathbb{T}^3)$, $Y = H^{-1}(\mathbb{T}^3)$ yields
\[
p^{(\varepsilon)}\rightarrow p \ in \  L^{2}(0,T;L^{2}(\mathbb{T}^{3})),
\]}
%Using H\"older's inequality we have
%\[
%p^{(\varepsilon)}\rightarrow p \ in \  L^{1}(0,T;L^{1}(\mathbb{T}^{3})),
%\]
%for every $\Omega\subset \mathbb{T}^{3}$ bounded, 
and \begin{align*}p^{(\varepsilon)}\rightarrow\ p\ a.e.\ in\ \mathbb{T}^{3}\times[0,T].\label{eqPAEConv} \end{align*} 
%Along with the boundedness of $D_{t}^{\alpha}p^{(\varepsilon)}$ in $L^{2}(0,T; H^{-1}(\mathbb{T}^{3}))$ we get $D_{t}^{\alpha}p^{(\varepsilon)}\rightharpoonup D_{}^{\alpha}p^{}$ in $L^{2}(0,T; H^{-1}(\mathbb{T}^{3}))$, since $\forall \psi \in L^{2}(0,T; H(u_{in},p_{in})^{1}(\mathbb{T}^{3})), $ similar as before, \Cref{prop:intbyparts_dis} yields as $\varepsilon \rightarrow 0$,
\iffalse\begin{align*}
\int D_{t}^{\alpha}p^{(\varepsilon)}\psi dt\rightarrow\int D_{t}^{\alpha}p^{}\psi dt.
\end{align*}
therefore,
\[
\int_{0}^{T}\int_{\text{\ensuremath{\mathbb{T}}}^{3}}\left(D_{t}^{\alpha}p^{(\varepsilon)}\right)\psi \;dxdt\rightarrow \int_{0}^{T}\int_{\text{\ensuremath{\mathbb{T}}}^{3}}\left(D_{t}^{\alpha}p^{}\right)\psi \;dxdt.
\]
\fi
Now take $\liminf_{\varepsilon\rightarrow0}$ on \eqref{eq36ep} and \eqref{eq37ep},  almost everywhere convergence of $ {u}^{(\varepsilon)}$ and $p^{(\varepsilon)}$ yields %\eqref{eqPAEConv} and \eqref{eqUAEConv} yields
\begin{equation}\label{eq36}
\int_{\text{\ensuremath{\mathbb{T}}}^{3}} {u}\;dx \leq \int_{\text{\ensuremath{\mathbb{T}}}^{3}}u_{in}\;dx,
\end{equation}
and
\begin{equation}\label{eq37}
\int_{\mathbb{T}^{3}}p\;dx\leq\int_{\mathbb{T}^{3}}p_{in}\;dx + \frac{1}{\alpha \Gamma_{\alpha}} H(u_{in},p_{in})  T^{\alpha}.
\end{equation}
%with the estimates \eqref{eq35}, \eqref{eq36} and \eqref{eq37}.
%%%%%%%%%%%%%%%%%%%%%%%%%%%%%%%%%%%%%%iffalse
\iffalse
\begin{align}
&\int_{\mathbb{T}^{3}}\left({(u^{(\varrho)})}^{2}+\frac{1}{2}|\nabla p^{(\varrho)}|^{2}\right)\;dx+K\bigg( \varrho\int_{0}^{t}\int_{\text{\ensuremath{\text{\ensuremath{\mathbb{T}}}^{3}}}}|\nabla{u^{(\varrho)}}(s,x)|^{2}\;dxds+\frac{1}{2}\int_{0}^{t}\int_{\text{\ensuremath{\mathbb{T}}}^{3}}|(-\Delta)^{s/2}\nabla p^{(\varrho)}(s,x)|^{2}\;dxds\bigg) \notag \\
 &  \ \ \ \ \ \ \ \ \ \ \ \ \ \ \ \ \ \ \ \ \ \ \ \ \ \  \ \ \ \ \ \ \ \ \ \ \ \ \ \ \ \ \ \  \ \ \  \leq\int_{\mathbb{T}^{3}}\left(u_{in}^{2}+\frac{1}{2}|\nabla p_{in}|^{2}\right)\;dx,\label{eq35}
\end{align}
\begin{equation}
\int_{\text{\ensuremath{\mathbb{T}}}^{3}}{u^{(\varrho)}}\;dx \leq \int_{\text{\ensuremath{\mathbb{T}}}^{3}}u_{in}\;dx.\label{eq36}
\end{equation}
\begin{equation}
\int_{\mathbb{T}^{3}}p^{(\varrho)}\;dx\leq\int_{\mathbb{T}^{3}}p_{in}\;dx+T^\alpha {\Gamma_{1-\alpha}}H(u_{in},p_{in}).\label{eq37}
\end{equation}
\fi
%%%%%%%%%%%%%%%%%%%%%%%%%%%%%%%%%%%%%%fi
\section{Limit $\varrho\rightarrow0$}\label{sec:rhoLimit}
This section is devoted to the limit $\varrho\rightarrow 0$. Define $$\Psi(u) : = \int_{0}^{T}\int_{\text{\ensuremath{\mathbb{T}}}^{3}}{(u^{})}^{2}\psi \;dxdt,$$
with $\psi\in L^2(0,T; H^1\cap L^\infty(\mathbb{T}^3))$. 
We have that 
\begin{align*}
\left| \Psi(u_1)-\Psi(u_2)\right| = & \int_{0}^{T}\int_{\text{\ensuremath{\mathbb{T}}}^{3}}\left| u_1^2-u_2^2\right|\psi \ dxdt\\
\leq & \norm{u_1-u_2}_{L^2(0,T; L^2(\mathbb{T}^3))}\left[\int_{0}^{T}\int_{\text{\ensuremath{\mathbb{T}}}^{3}}\psi^2(u_1+u_2)^2 \ dxdt\right]^\frac{1}{2},\\
\leq & \norm{u_1-u_2}_{L^2(0,T; L^2(\mathbb{T}^3))}  \norm{\psi}^2_{L^2(0,T; L^\infty(\mathbb{T}^3))}\norm{u_1+u_2}^2_{L^\infty(0,T; L^2(\mathbb{T}^3))}.
\end{align*}
Therefore $\Psi (\cdot)$ is continuous for the strong topology $\norm{\cdot  }_{L^2(0,T; L^2(\mathbb{T}^3))}$. Since $\Psi$ is convex, $u^{(\varrho)} \rightharpoonup  u$ in $L^2(0,T,L^2(\mathbb{T}^3))$ and $u^{(\varrho)}$ is uniformly bounded in $L^\infty(0,T,L^2(\mathbb{T}^3) )$, we have that (Corollary III.8 \cite{brezis2010functional})
$$
 \liminf_{\varrho\rightarrow 0} \Psi(u^{(\varrho)} ) \ge  \Psi(u).
$$
In other words, 
$$\int_{0}^{T}\int_{\text{\ensuremath{\mathbb{T}}}^{3}}u^{2}\psi \;dxdt \leq \liminf_{\varrho\rightarrow 0} \int_{0}^{T}\int_{\text{\ensuremath{\mathbb{T}}}^{3}}{(u^{(\varrho)})}^{2}\psi \;dxdt.$$

{{Next, we take a look at the second term:
$$\int_{0}^{T}\int_{\text{\ensuremath{\text{\ensuremath{\mathbb{T}}}^{3}}}}{u^{(\varrho)}}\nabla p^{(\varrho)} \cdot\nabla\phi \;dxdt.$$
With $\phi \in L^2(0,T,W^{1,\infty}(\mathbb{T}^{3}))$, consider
\begin{align*}
\int\int[ {u}^{(\varrho)}\nabla p^{(\varrho)}-{u}\nabla p]\cdot\nabla \phi\;dxdt & =\int\int[ {u}^{(\varrho)}\nabla p^{(\varrho)}- {u}^{(\varrho)}\nabla p+ {u}^{(\varrho)}\nabla p-{u}\nabla p]\cdot\nabla \phi\;dxdt\\
 & =\int\int {u}^{(\varrho)}\nabla \phi\cdot(\nabla p^{(\varrho)}-\nabla p)\;dxdt+\int\int\nabla p\cdot\nabla \phi( {u}^{(\varrho)}-{u})\;dxdt\notag\\
 & =I_{1}+I_{2}.\notag
\end{align*}
%To get the strong convergence of $\nabla p$ in $L^{2}(0,T;L^{2}(\mathbb{T}^{3}))$, we first
We can have $I_2 \rightarrow 0$ since $u^{(\varrho)}\rightharpoonup^* u$ in $L^{\infty}(0,T;L^{2}(\mathbb{T}^{3}))$ and  $\nabla p\cdot \nabla \phi$ in $L^{1}(0,T;L^{2}(\mathbb{T}^{3}))$. 
%Here we use $\nabla p \in L^\infty(0,T,L^2(\mathbb{T}^{3}))$ since $\nabla p^{(\varrho)} \rightharpoonup^* \nabla p$ in $L^\infty(0,T,L^2(\mathbb{T}^{3}))$. 

To handle $I_1$, we need to show strong convergence for $\nabla p^{(\varrho)}$. For that we first bound $D_{t}^{\alpha}p^{(\varrho)}$ as follows: % get from \eqref{eq34} that 
\begin{align*}
\left\vert\int_{0}^{T}\int_{\text{\ensuremath{\mathbb{T}}}^{3}}\left<D_{t}^{\alpha}p^{(\varrho)},\psi\right> \;dxdt\right\vert%\leq & \left\vert\int_{0}^{T}\int_{\text{\ensuremath{\mathbb{T}}}^{3}}(-\Delta)^{s/2}p^{(\varrho)} \cdot(-\Delta)^{s/2}\psi \;dxdt\right\vert\\
 %& +\left\vert\int_{0}^{T}\int_{\text{\ensuremath{\mathbb{T}}}^{3}}( {u}^{(\varrho)})^{2}\psi \;dxdt\right\vert\\
\leq & \left\Vert (-\Delta)^{s/2}p^{(\varrho)}\right\Vert _{L^{2}(0,T;L^{2}(\mathbb{T}^{3}))}\left\Vert (-\Delta)^{s/2}\psi\right\Vert _{L^{2}(0,T;L^{2}(\mathbb{T}^{3}))}\\
 & +\left\Vert ( {u}^{(\varrho)})^{2}\right\Vert _{L^{\infty}(0,T;L^{1}(\mathbb{T}^{3}))}\left\Vert \psi\right\Vert _{L^{2}(0,T;L^{\infty}(\mathbb{T}^{3}))}\\
\leq & C(T)\left\Vert \psi\right\Vert _{L^{2}(0,T;L^\infty \cap H^1(\mathbb{T}^{3}))},
\end{align*}
which implies%Consequently, $D_{t}^{\alpha}p^{(\varrho)}$ is bounded by some constant function depending on $T$ in the dual space of $L^{2}(0,T;H^{1}(\mathbb{T}^{3})):$
\[
\left\Vert D_{t}^{\alpha}p^{(\varrho)}\right\Vert _{L^{2}(0,T; (L^\infty \cap H^1)'(\mathbb{T}^3)))}\leq C(T).
\]
Then \Cref{lem:main.cont} with $X = H^{s+1}(\mathbb{T}^3)$, $B = H^1(\mathbb{T}^3)$, $Y = (L^\infty \cap H^1)'(\mathbb{T}^3)$ yields $p^{(\varrho)}\rightarrow p$ in $L^{2}(0,T;H^{1}(\mathbb{T}^{3}))$. 
%Using $u^{(\varrho)}$ uniformly bounded in $L^{3}(0,T;L^{3}(\mathbb{T}^{3}))$ and $\phi\in L^{6}(0,T;W^{1,6}(\mathbb{T}^{3}))$, 
We conclude that $I_1 \rightarrow 0$, since
\begin{align}
I_1\leq\norm{u^{(\varrho)}}_{L^\infty(0,T;L^2(\mathbb{T}^3))}\norm{\nabla \phi}_{L^2(0,T;L^\infty(\mathbb{T}^3))}\norm{\nabla p^{(\varrho)}-\nabla p}_{L^2(0,T;L^2(\mathbb{T}^3))}\notag.
\end{align}

The last term 
$$\varrho\int_{0}^{T}\int_{\text{\ensuremath{\text{\ensuremath{\mathbb{T}}}^{3}}}}\nabla{u^{(\varrho)}} \cdot\nabla\phi \;dxdt\rightarrow 0,$$
since $ \sqrt{\varrho}\nabla{u^{(\varrho)}}$
is uniformly bounded in $L^{2}(0,T;L^{2}(\mathbb{T}^{3}))$. %and we have the weak convergence
%$$\sqrt{\varrho}\nabla{u^{(\varrho)}}\rightharpoonup \sqrt{\varrho}\nabla{u^{}}, \ in  \ L^{2}(0,T;L^{2}(\mathbb{T}^{3})) \ as \  \varrho \rightarrow 0,$$
%therefore, 
%$$\varrho\int_{0}^{T}\int_{\text{\ensuremath{\text{\ensuremath{\mathbb{T}}}^{3}}}}\nabla{u^{(\varrho)}} \nabla\phi \;dxdt = \sqrt{\varrho}\int_{0}^{T}\int_{\text{\ensuremath{\text{\ensuremath{\mathbb{T}}}^{3}}}}\sqrt{\varrho}\nabla{u^{(\varrho)}} \nabla\phi \;dxdt\rightarrow \sqrt{\varrho}\int_{0}^{T}\int_{\text{\ensuremath{\text{\ensuremath{\mathbb{T}}}^{3}}}}\sqrt{\varrho}\nabla{u^{}} \nabla\phi \;dxdt.$$ 
%Therefore, we have 
%\begin{align*}
%\int_{0}^{T}\int_{\text{\ensuremath{\mathbb{T}}}^{3}}\left(D_{t}^{\alpha}{u}\right)\phi \;dxdt+\int_{0}^{T}\int_{\text{\ensuremath{\text{\ensuremath{\mathbb{T}}}^{3}}}}{u^{}}\nabla p^{} \nabla\phi \;dxdt\nonumber  =0,\ \forall\phi\in L^{2}(0,T;\ H^{1}(\mathbb{T}^{3})).
%\end{align*}
}}

Furthermore, we consider the energy estimates \eqref{eq35}. By taking $\liminf_{\varrho\rightarrow 0}$ on both sides %and lower weak semicontinuity of $L^p$ norm, 
we have
\begin{equation*}
 \int_{\mathbb{T}^{3}}\left(u^2+\frac{1}{2}|\nabla p|^2\right)\;dx\notag
+\frac{t^{\alpha -1}}{\Gamma_{\alpha}}
\int_0^t\int_{\text{\ensuremath{\mathbb{T}}}^{3}}|(-\Delta)^{s/2}\nabla p(s,x)|^{2}\;dxds
\leq \int_{\mathbb{T}^{3}}\left(u_{in}^2+\frac{1}{2}|\nabla p_{in}|^{2}\right)\;dx.
\end{equation*}
Follow the same logic, with Fatou's lemma, we can pass the limits on \eqref{eq36} and \eqref{eq37}:
\begin{equation*}
\int_{\text{\ensuremath{\mathbb{T}}}^{3}}{u^{}}\;dx \leq \int_{\text{\ensuremath{\mathbb{T}}}^{3}}u_{in}\;dx,
\end{equation*}
\begin{equation*}
\int_{\mathbb{T}^{3}}p^{}\;dx\leq\int_{\mathbb{T}^{3}}p_{in}\;dx + \frac{1}{\alpha \Gamma_{\alpha}} H(u_{in},p_{in})  T^{\alpha}.
\end{equation*}
This concludes the proof of Theorem \ref{thm:mainthm}.

\section{Appendix}

{{In the following lemma we show a formal $L^{3}(0,T;L^3(\mathbb{T}^3))$-estimate for $u$, provided $\frac{1}{2}<s\le1$. We will not use this estimate in this manuscript. We add it here for completeness, and because we believe it can  be useful in the future to show compactness for $u$. 
%this However, Lemma \ref{lemma:L3L3} in an important first step to show strong compactness for $u$.
%, however, that will be essential to prove compactness for $u$ 

\begin{lemma} \label{lemma:L3L3}
If  $\frac{1}{2}<s\le1$ the function $u$, solution to (\ref{1L}), is bounded in $L^{3}(0,T;L^3(\mathbb{T}^3))$ as 
$$
\|u\|_{L^3(0,T,L^3(\mathbb{T}^{3}{))}  } \le C(T, H(u_{in},p_{in})).
$$
%by a constant that only depends on the initial data. 
%uniformly with respect to $\varrho$ and $\varepsilon$.
\end{lemma}
\begin{proof}
We test the equations in (\ref{1L}) with $p $ and $u$ respectively. We get
\begin{align}
\int_{0}^{T}\int_{\text{\ensuremath{\mathbb{T}}}^{3}}\left<D_{t}^{\alpha}{u },p \right> \;dxdt+\int_{0}^{T}\int_{\text{\ensuremath{\text{\ensuremath{\mathbb{T}}}^{3}}}}{u }\left|\nabla p \right|^2 \;dxdt =0%\nonumber \\
%+\varrho\int_{0}^{T}\int_{\text{\ensuremath{\text{\ensuremath{\mathbb{T}}}^{3}}}}\nabla{u } \cdot\nabla p  \;dxdt & =0 
\label{l3l3.eq1},
\end{align}
and
\begin{align}
\int_{0}^{T}\int_{\mathbb{T}^{3}}\left<D_{t}^{\alpha}p ,u  \right>\;dxdt+\int_{0}^{T}\int_{\text{\ensuremath{\mathbb{T}}}^{3}}(-\Delta)^{s/2}p  (-\Delta)^{s/2}u  \;dxdt =\int_{0}^{T}\int_{\text{\ensuremath{\mathbb{T}}}^{3}}{(u )}^{3} \;dxdt. \label{l3l3.eq2}
\end{align}
Integration by parts using \eqref{byparts.allen} yields
\begin{align}
\int_0^T\left<D^{\alpha}_tp (t),  u (t)\right>\ dt +\int_0^T&\left<D^{\alpha}_tu (t)  ,p (t)\right> \ dt = \frac{1}{\Gamma_{1-\alpha}}\int_0^T u (t)p (t)\left[ \frac{1}{(T-t)^{\alpha}}+\frac{1}{t^{\alpha}}\right]\ dt\label{l3l3.eq4}\\
&+ \frac{\alpha}{\Gamma_{1-\alpha}} \int_0^T\int_0^t \frac{(u (t)-u (s))(p (t)-p (s))}{(t-s)^{1+\alpha}}\ dsdt\nonumber\\
&- \frac{1}{\Gamma_{1-\alpha}}\int_0^T \frac{u (t)p_{in}+p (t)u_{in}}{t^{\alpha}}\ dt.\nonumber
\end{align}
%\begin{align}
%\int_0^TD^{\alpha}_tp (t)  u (t)\ dt +\int_0^TD^{\alpha}_tu (t)  & p (t)\ dt = \int_0^T u (t)p (t)\left[ \frac{1}{(T-t)^{\alpha}}+\frac{1}{t^{\alpha}}\right]\ dt\label{l3l3.eq4}\\
%&+ \alpha \int_0^T\int_0^t \frac{(u (t)-u (s))(p (t)-p (s))}{(t-s)^{1+\alpha}}\ dsdt\nonumber\\
%&- \int_0^T \frac{u (t)p_{in}+p (t)u_{in}}{t^{\alpha}}\ dt.\nonumber
%\end{align}
Adding \eqref{l3l3.eq1} to \eqref{l3l3.eq2} and using \eqref{l3l3.eq4}, we get
\begin{align}
\int_{0}^{T}\int_{\text{\ensuremath{\mathbb{T}}}^{3}}{(u )}^{3}\;dxdt =& \int_{0}^{T}\int_{\text{\ensuremath{\text{\ensuremath{\mathbb{T}}}^{3}}}}{u }\left|\nabla p \right|^2 \;dxdt \label{l3l3.eq5}\\
%&+\varrho\int_{0}^{T}\int_{\text{\ensuremath{\text{\ensuremath{\mathbb{T}}}^{3}}}}\nabla{u } \cdot\nabla p  \;dxdt\nonumber\\
&+\int_{0}^{T}\int_{\text{\ensuremath{\mathbb{T}}}^{3}}(-\Delta)^{s/2}p  (-\Delta)^{s/2}u  \;dxdt\nonumber\\
%&+ \varepsilon\int_{0}^{T}\int_{\mathbb{T}^{3}}\nabla p\cdot\nabla u\; dxdt \nonumber \\
&+ \frac{1}{\Gamma_{1-\alpha}}\int_0^T\int_{\mathbb{T}^3} u (t)p (t)\left[ \frac{1}{(T-t)^{\alpha}}+\frac{1}{t^{\alpha}}\right]\ dxdt\nonumber\\
&+ \frac{\alpha}{\Gamma_{1-\alpha}} \int_0^T\int_0^t\int_{\mathbb{T}^3} \frac{(u (t)-u (s))(p (t)-p (s))}{(t-s)^{1+\alpha}}\ dxdsdt\nonumber\\
&- \frac{1}{\Gamma_{1-\alpha}}\int_0^T \int_{\mathbb{T}^3} \frac{u (t)p_{in}+p (t)u_{in}}{t^{\alpha}}\ dxdt\nonumber\\
= & \  I_1 +I_2+I_3+I_4+I_5.\nonumber
\end{align}
%Since we can bound the first three terms using Holder's inequality, and $I_6$ is good due to the negative sign and the positivity of $u $ and $p $, 
We use H\"{o}lder's inequality to bound $I_1$:
\begin{align*}
I_1 = \int_0^T\int_{\mathbb{T}^3} u  \left|\nabla p \right|^2\ dxdt 
\leq & \left[\int_0^T\int_{\mathbb{T}^3} (u )^3\ dxdt\right]^\frac{1}{3} \left[\int_0^T\int_{\mathbb{T}^3} \left[\left|\nabla p \right|^2\right]^\frac{3}{2}\ dxdt\right]^\frac{2}{3}\\
%= & \left[\int_0^T\int_{\mathbb{T}^3} (u )^3\ dxdt\right]^\frac{1}{3} \left[\int_0^T\int_{\mathbb{T}^3} \left|\nabla p \right|^3\ dxdt\right]^\frac{2}{3}\\
%\leq & \eta^2 \int_0^T\int_{\mathbb{T}^3} (u )^3\ dxdt + \frac{1}{\eta}\int_0^T\int_{\mathbb{T}^3} \left|\nabla p \right|^3\ dxdt\\
\leq & \frac{1}{4} \int_0^T\int_{\mathbb{T}^3} (u )^3\ dxdt + 2\norm{\nabla p }_{L^{3}(0,T;L^3(\mathbb{T}^3))}.
\end{align*}
Since $s>\frac{1}{2}$, the $\norm{\nabla p }_{L^{3}(0,T;L^3(\mathbb{T}^3))}$ term can be bounded via interpolation of $\nabla p \in L^\infty(0,T; L^{2}(\mathbb{T}^3))$ with $\nabla p \in L^2(0,T; L^{\frac{6}{3-2s}}(\mathbb{T}^3))$. We bound $I_2$  using \eqref{eq_energy_ine} after integrating by parts.
%are bounded directly following Young's Inequality and the estimates from \eqref{eq35}.
%For $I_4$ consider: 
%\begin{align*}
%\varepsilon\int_{0}^{T}\int_{\mathbb{T}^{3}}\nabla p\cdot\nabla u\; dxdt  & \le \frac{\varepsilon^2}{\rho}\int_{0}^{T}\int_{\mathbb{T}^{3}} |\nabla p|^2\;dxdt + \rho \int_{0}^{T}\int_{\mathbb{T}^{3}} |\nabla u|^2\;dxdt \\
%& \le \left( 1 +  \frac{\varepsilon^2}{\rho}\right) C(u_{in}.p_{in}).
%\end{align*}
Next we analyze $I_3$:
\begin{align*}
I_3 &= \frac{1}{\Gamma_{1-\alpha}}\int_0^T\int_{\mathbb{T}^3} u (t)p (t)\left[ \frac{1}{(T-t)^{\alpha}}+\frac{1}{t^{\alpha}}\right]\ dxdt\\ %&\leq \int_0^T \left[ \frac{1}{(T-t)^{\alpha}}+\frac{1}{t^{\alpha}}\right]\;dt   \norm{u }_{L^{\infty}(0,T;L^2(\mathbb{T}^3))}\norm{p }_{L^{\infty}(0,T;L^2(\mathbb{T}^3))}\\
&\leq \frac{2}{\Gamma_{2-\alpha}}  T^{1-\alpha}  \norm{u }_{L^{\infty}(0,T;L^2(\mathbb{T}^3))}\norm{p }_{L^{\infty}(0,T;L^2(\mathbb{T}^3))}\nonumber \\
&\leq C(T,H(u_{in}, p_{in})),
\end{align*}
thanks to \eqref{eq_energy_ine}.
Now take a close look at $I_4$:
\begin{align*}
I_4 = \ &\frac{\alpha}{\Gamma_{1-\alpha}}\int_0^T\int_0^t\int_{\mathbb{T}^3} \frac{(u (t)-u (s))(p (t)-p (s))}{(t-s)^{1+\alpha}}\ dxdsdt \\
\leq & \  \frac{\alpha}{\Gamma_{1-\alpha}} \int_0^T\int_0^t\int_{\mathbb{T}^3} \frac{(u (t)-u (s))^2}{(t-s)^{1+\alpha}}\ dxdsdt + \frac{\alpha}{\Gamma_{1-\alpha}} \int_0^T\int_0^t\int_{\mathbb{T}^3} \frac{(p (t)-p (s))^2}{(t-s)^{1+\alpha}}\ dxdsdt.
\end{align*}
Integration by parts formula \eqref{byparts.allen} gives
\begin{align*}
 \  \frac{\alpha}{\Gamma_{1-\alpha}} \int_0^T\int_0^t\int_{\mathbb{T}^3} \frac{(u (t)-u (s))^2}{(t-s)^{1+\alpha}}\ dxdsdt =& \ 2\alpha \int_0^T \int_{\mathbb{T}^3} u (t)D^\alpha_t u (t)\ dxdt + \frac{\alpha}{\Gamma_{1-\alpha}}\int_0^T \int_{\mathbb{T}^3} \frac{2u (t)u_{in}}{t^\alpha}\ dxdt\\
 &-\frac{\alpha}{\Gamma_{1-\alpha}}\int_0^T \int_{\mathbb{T}^3} (u (t))^2\left[\frac{1}{(T-t)^\alpha}+\frac{1}{t^\alpha}\right]\ dxdt \\
 =& \  2\alpha \int_0^T \int_{\mathbb{T}^3} u (t)\textrm{div} \ ( u (t)   \nabla p (t))\ dxdt\\
 %&+ \alpha\varrho \int_0^T \int_{\mathbb{T}^3} \left|\nabla u (t)\right|^2 \ dxdt
& + \frac{\alpha}{\Gamma_{1-\alpha}}\int_0^T \int_{\mathbb{T}^3} \frac{2u (t)u_{in}}{t^\alpha}\ dxdt\\
 &-\frac{\alpha}{\Gamma_{1-\alpha}}\int_0^T \int_{\mathbb{T}^3} (u (t))^2\left[\frac{1}{(T-t)^\alpha}+\frac{1}{t^\alpha}\right]\ dxdt \\
 \leq& - 2\alpha \int_0^T \int_{\mathbb{T}^3} \nabla u (t) \cdot(u (t)   \nabla p (t))\ dxdt\\
 & + \frac{\alpha}{\Gamma_{1-\alpha}}\int_0^T \int_{\mathbb{T}^3} \frac{2u (t)u_{in}}{t^\alpha}\ dxdt\\%+ \alpha\varrho \int_0^T \int_{\mathbb{T}^3} \left|\nabla u (t)\right|^2 \ dxdt\\
 =& \ A_1 +A_2.
\end{align*}
%Since $\sqrt{\varrho}\nabla u $ is uniformly bounded in $L^2(0,T;L^2{(\mathbb{T}^3}))$, $A_3$ is bounded by $C(T)H(u_{in},p_{in})$. Similarly for $A_2$: %uses $u (t)\in L^\infty(0,T,L^2(\mathbb{T}^3))$: %\leq \int_0^T \frac{1}{t^\alpha} \ dt \norm{ u (t)}_{L^{\infty }(0,T;L^2(\mathbb{T}^3))}\norm{ u_{in}}_{L^{\infty }(0,T;L^2(\mathbb{T}^3))}
For $A_2$ we have:
$$A_2 \leq \frac{ T^{1-\alpha}}{\Gamma_{2-\alpha}} \norm{ u (t)}_{L^{\infty }(0,T;L^2(\mathbb{T}^3))}\norm{ u_{in}}_{L^{\infty }(0,T;L^2(\mathbb{T}^3))}\leq C(T,H(u_{in},p
_{in})).$$
In $A_1$, integration by parts yields
\begin{align*}
A_1 %&- \int_0^T \int_{\mathbb{T}^3} \nabla (u (t)) ^2   \cdot\nabla p (t)\ dxdt\\
%=&\int_0^T \int_{\mathbb{T}^3}  (u (t)) ^2  \Delta p (t)\ dxdt\\
%= & \int_{0}^{T}\int_{\mathbb{T}^{3}}\left(D_{t}^{\alpha}p \right)\Delta p  \;dxdt+\int_{0}^{T}\int_{\text{\ensuremath{\mathbb{T}}}^{3}}(-\Delta)^{s/2}p  (-\Delta)^{s/2}\Delta p  \;dxdt\\
= & -\int_{0}^{T}\int_{\mathbb{T}^{3}}\left(D_{t}^{\alpha}\nabla p \right)\cdot\nabla p  \;dxdt-\int_{0}^{T}\int_{\text{\ensuremath{\mathbb{T}}}^{3}}\left|(-\Delta)^{s/2}\nabla p \right|^2 \;dxdt\\
\leq & -\frac{1}{2}\int_{0}^{T}\int_{\mathbb{T}^{3}}D_{t}^{\alpha}(\nabla p )^2 \;dxdt.
%\leq & \int_0^T\norm{D^{\alpha}_t(|\nabla p (t)|^2)}_{L^1(\mathbb{T}^3)}\;dt\\
%\leq & T^{1-\alpha}  \int_0^T \frac{\norm{D^{\alpha}_t(|\nabla p (t)|^2)}_{L^1(\mathbb{T}^3)}}{(T-s)^{1-\alpha}}\;ds,
%= & -\frac{1}{2}\int_{\mathbb{T}^{3}}\int_{0}^{T}D_{t}^{\alpha}(\nabla p )^2 \;dtdx,
\end{align*}
%Since we don't have the estimate on $D_{t}^{\alpha}(\nabla p )^2$, 
%Apply integration by parts formula \eqref{byparts.almeida} to \eqref{l3l3.eq2} we have
%\begin{align}
%\int_0^TD^{\alpha}_tu (t)  p (t)\ dt = -\frac{1}{\Gamma_{1-\alpha}}\int_0^Tu (t) \frac{d}{dt}\left(\int_t^T\frac{p (s)}{(s-t)^{\alpha}}\ ds\right)dt
%-\frac{u_{in}}{\Gamma_{1-\alpha}}\int_0^T\frac{p (s)}{s^{\alpha}}. \label{l3l3.eq3}
%\end{align}
Next we use formula \eqref{byparts.almeida} with $\phi =1$ and get:
\begin{align*}
A_1
\leq \ & \frac{1}{2\Gamma_{1-\alpha}}\int_{\mathbb{T}^{3}} \int_0^T
\left|\nabla p \right|^2(t) \frac{d}{dt}\left(\int_t^T\frac{1}{(s-t)^{\alpha}}\ ds\right)dtdx +\frac{1}{2\Gamma_{1-\alpha}}\int_{\mathbb{T}^{3}}\nabla p_{in}^2\int_0^T\frac{1}{s^{\alpha}}\ dsdx\\
=& -\frac{1}{2\Gamma_{1-\alpha}}\int_0^T\int_{\mathbb{T}^{3}}\frac{\left|\nabla p \right|^2(t)}{(T-t)^{\alpha}}\ dxdt+ \frac{T^{1-\alpha}}{2\Gamma_{2-\alpha}} \norm{\nabla p_{in}}^2_{L^2(\mathbb{T}^{3})} \\
%\leq & \ C(T) \norm{\nabla p_{in}}^2_{L^2(\mathbb{T}^{3})} \\
\leq & \ C(T,H(u_{in},p_{in})).
\end{align*}
%Since the second time integral is calculable, we only need to estimate the first one:
%\begin{align*}
%A_1
%\leq &-\frac{1}{2\Gamma_{1-\alpha}}\int_{\mathbb{T}^{3}}\int_0^T(\nabla p )^2(t) (T-t)^{-\alpha}dtdx\\
%&+\frac{\norm{(\nabla p )(0)}^2_{L^2(\mathbb{T}^{3})}}{2\Gamma_{1-\alpha}}\frac{T^{1-\alpha}}{1-\alpha} \lvert \mathbb{T}^{3} \rvert \\
%=&-\frac{1}{2\Gamma_{1-\alpha}}\int_0^T\int_{\mathbb{T}^{3}}\frac{(\nabla p )^2(t)}{(T-t)^{\alpha}}dxdt+C(\alpha, T, |\mathbb{T}^3|) \norm{(\nabla p )(0)}^2_{L^2(\mathbb{T}^{3})} \\
%\leq &-\frac{1}{2\Gamma_{1-\alpha}}\int_0^T\int_{\mathbb{T}^{3}}\frac{(\nabla p )^2(t)}{T^{\alpha}}dxdt+ C(\alpha, T, |\mathbb{T}^3|) \norm{(\nabla p )(0)}^2_{L^2(\mathbb{T}^{3})} \\
%\leq &-\frac{1}{2\Gamma_{1-\alpha}T^{\alpha}}\int_0^T \int_{\mathbb{T}^{3}}(\nabla p )^2(t)dxdt+C(\alpha, T, |\mathbb{T}^3|) \norm{(\nabla p )(0)}^2_{L^2(\mathbb{T}^{3})} \\
%\leq &C(\alpha, T, |\mathbb{T}^3|) \norm{(\nabla p )(0)}^2_{L^2(\mathbb{T}^{3})}. \\
%T^{1-\alpha}  \Gamma (\alpha)  \norm{|\nabla p (t)|^2(T)-|\nabla p (t)|^2(0)}_{L^1(\mathbb{T}^3)}\\
%\leq & C(\alpha, T, |\mathbb{T}^3|)  \left(\norm{p (t)_{in}}_{H^1(\mathbb{T}^3)} + \norm{p (t)}_{L^{\infty}(0,T;H^1(\mathbb{T}^3))}\right).
%\end{align*}
Summarizing
\begin{align*}
  \frac{\alpha}{\Gamma_{1-\alpha}} \int_0^T\int_0^t\int_{\mathbb{T}^3} \frac{(u (t)-u (s))^2}{(t-s)^{1+\alpha}}\ dxdsdt%\leq &  C(\alpha, T, |\mathbb{T}^3|) \norm{(\nabla p )(0)}^2_{L^2(\mathbb{T}^{3})}\\
 %&+ \int_0^T \frac{1}{t^\alpha} \ dt \norm{ u (t)}_{L^{\infty }(0,T;L^2(\mathbb{T}^3))}\norm{ u_{in}}_{L^{\infty }(0,T;L^2(\mathbb{T}^3))}\\
 %\leq &  C(\alpha, T, |\mathbb{T}^3|) \norm{(\nabla p )(0)}^2_{L^2(\mathbb{T}^{3})}\\
 \leq &  C(T,H(u_{in},p_{in})).
 %+\frac{ T^{1-\alpha}}{1-\alpha} \norm{ u (t)}_{L^{\infty }(0,T;L^2(\mathbb{T}^3))}\norm{ u_{in}}_{L^{\infty }(0,T;L^2(\mathbb{T}^3))}.
\end{align*}
Now we look at the second term of $I_4$:
\begin{align*}
&\frac{\alpha}{\Gamma_{1-\alpha}} \int_0^T\int_0^t\int_{\mathbb{T}^3} \frac{(p (t)-p (s))^2}{(t-s)^{1+\alpha}}\ dxdsdt \\
= &\; 2 \int_0^T \int_{\mathbb{T}^3} p (t)D^\alpha_t p (t)\ dxdt + \frac{1}{\Gamma_{1-\alpha}}\int_0^T \int_{\mathbb{T}^3} \frac{2p (t)p_{in}}{t^\alpha}\ dxdt-\frac{1}{\Gamma_{1-\alpha}}\int_0^T \int_{\mathbb{T}^3} (p (t))^2\left[\frac{1}{(T-t)^\alpha}+\frac{1}{t^\alpha}\right]\ dxdt \\
\leq & -\int_0^T\int_{\mathbb{T}^3} [(-\Delta)^\frac{s}{2}p ]^2 \ dxdt + \int_0^T\int_{\mathbb{T}^3} (u )^2 p \ dxdt + \frac{1}{\Gamma_{1-\alpha}}\int_0^T \int_{\mathbb{T}^3} \frac{2p (t)p_{in}}{t^\alpha}\ dxdt\\
\leq &  \int_0^T\int_{\mathbb{T}^3} (u )^2 p \ dxdt+ \frac{1}{\Gamma_{1-\alpha}}\int_0^T \int_{\mathbb{T}^3} \frac{2p (t)p_{in}}{t^\alpha}\ dxdt\\
= & B_1+B_2.
\end{align*}
The bound for $B_2$ follows from
$$B_2 \leq \frac{ T^{1-\alpha}}{\Gamma_{2-\alpha}} \norm{ p (t)}_{L^{\infty }(0,T;L^2(\mathbb{T}^3))}\norm{ p_{in}}_{L^{\infty }(0,T;L^2(\mathbb{T}^3))}\leq C(T,H(u_{in},p
_{in})).$$
For  $B_1$ we have:
\begin{align*}
B_1
\leq &\; \frac{1}{2} \int_0^T\int_{\mathbb{T}^3} (u )^3\ dxdt + 4\int_0^T\int_{\mathbb{T}^3} (p )^3\ dxdt\\
\leq &\; \frac{1}{2}  \int_0^T\int_{\mathbb{T}^3} (u )^3\ dxdt + 4\norm{ p }_{L^{3}(0,T;L^3(\mathbb{T}^3))}.
\end{align*}
%First term
%\begin{align*}
%-\int_0^T\int_{\mathbb{T}^3} [(-\Delta)^\frac{s}{2}p ]^2 \ dxdt\leq & \norm{ p (t)}_{L^{2 }(0,T;H^s(\mathbb{T}^3))}.
%\end{align*}
%Second term
%\begin{align*}
%\int_0^T\int_{\mathbb{T}^3} (u )^2 p \ dxdt %\leq & \left[\int_0^T\int_{\mathbb{T}^3} [(u )^2]^\frac{3}{2}\ dxdt\right]^\frac{2}{3} \left[\int_0^T\int_{\mathbb{T}^3} (p )^3\ dxdt\right]^\frac{1}{3}\\
%= & \left[\int_0^T\int_{\mathbb{T}^3} (u )^3\ dxdt\right]^\frac{2}{3} \left[\int_0^T\int_{\mathbb{T}^3} (p )^3\ dxdt\right]^\frac{1}{3}\\
%\leq & 2 \int_0^T\int_{\mathbb{T}^3} (u )^3\ dxdt + \frac{1}{4}\int_0^T\int_{\mathbb{T}^3} (p )^3\ dxdt\\
%\leq & 2 \int_0^T\int_{\mathbb{T}^3} (u )^3\ dxdt + \frac{1}{4}\norm{ p }_{L^{3}(0,T;L^3(\mathbb{T}^3))}.
%    \end{align*}
The second term is bounded in accordance to the interpolation of $ p \in L^\infty(0,T; L^{2}(\mathbb{T}^3))\cap L^2(0,T; L^6(\mathbb{T}^3))$. Summarizing we have:
$$\frac{1}{4} \int_0^T\int_{\mathbb{T}^3} (u )^3\ dxdt\leq C(T,H(u_{in},p_{in})).$$
\end{proof}
}}

%\bibliographystyle{siamplain}
%\bibliography{references}
\end{document}